\documentclass[11pt]{article}

\usepackage[margin=1in]{geometry}
\usepackage{amsmath,amssymb,amsthm,mathtools}
\usepackage{enumitem}
\usepackage{booktabs}
\usepackage{array}
\usepackage[numbers]{natbib}
\usepackage{microtype}
\usepackage{graphicx}
\usepackage{float}
\graphicspath{{./figures/}{./}}
\usepackage{pgfplots}
\pgfplotsset{compat=1.18}

\usepackage{subcaption}
\usepackage{tikz}
\usepackage[dvipsnames]{xcolor}
\usepackage[hidelinks]{hyperref}
\hypersetup{
    colorlinks=true,
    linkcolor=blue!70!black,
    citecolor=blue!70!black,
    urlcolor=blue!70!black
}
\usetikzlibrary{calc,shapes.geometric}
\definecolor{lightblue}{rgb}{0.88,1,1}
\definecolor{cleanblue}{rgb}{0.1,0.7,1}

\newtheorem{theorem}{Theorem}[section]
\newtheorem{lemma}{Lemma}[section]
\newtheorem{proposition}{Proposition}[section]
\newtheorem{corollary}{Corollary}[section]
\theoremstyle{definition}
\newtheorem{definition}{Definition}[section]
\newtheorem{assumption}{Assumption}[section]
\newtheorem{example}{Example}[section]
\theoremstyle{remark}
\newtheorem{remark}{Remark}[section]

\newcommand{\R}{\mathbb{R}}

\newcommand{\Rp}{\mathbb{R}_+}
\newcommand{\E}{\mathbb{E}}
\newcommand{\Pbb}{\mathbb{P}}
\newcommand{\one}{\mathbf{1}}
\newcommand{\ip}[2]{\langle #1,#2\rangle}
\newcommand{\norm}[1]{\left\lVert #1 \right\rVert}
\newcommand{\normone}[1]{\norm{#1}_1}
\newcommand{\normtwo}[1]{\norm{#1}_2}
\newcommand{\eps}{\varepsilon}
\newcommand{\F}{\mathcal{F}}
\newcommand{\Sset}{\mathcal{S}}
\newcommand{\PiC}{\Pi}
\newcommand{\Hpi}{H_{\PiC}}
\newcommand{\alphapi}{\alpha(\PiC)}
\newcommand{\alphatwopi}{\alpha_2(\PiC)}
\newcommand{\rcirc}{r^\circ}
\newcommand{\Hface}{\mathcal{H}^\star}
\newcommand{\conv}{\operatorname{conv}}

\newcommand{\argmax}{\operatorname*{arg\,max}}

\newcommand{\dd}{\mathrm{d}}

\newcommand{\floor}[1]{\left\lfloor #1 \right\rfloor}
\newcommand{\Peak}[1]{\operatorname{Peak}_{#1}}
\newcommand{\PeakT}{\Peak{T}}
\newcommand{\safeincludegraphics}[2][]{%
  \IfFileExists{figures/#2}{\includegraphics[#1]{figures/#2}}{%
    \IfFileExists{#2}{\includegraphics[#1]{#2}}{\includegraphics[#1]{figures/#2}}%
  }%
}

\title{Finite-Time Queue Peak Laws in Stochastic Networks:\\ Logarithmic Scaling After Geometric Thresholds}
\author{Hao Liang\footnotemark[1] \and Cheng Tang\footnotemark[1] \and Yunzong Xu\thanks{Authors are listed alphabetically.}}
\date{University of Illinois Urbana--Champaign\\
\texttt{\small\{hliang17,chengt6,xyz\}@illinois.edu}}

\begin{document}
\addtocontents{toc}{\protect\setcounter{tocdepth}{-1}}
\maketitle

\begin{abstract}
We study finite-horizon queue peaks in generalized switches, a standard stochastic-network model in which many queues share constrained service resources. Arrivals may be dependent, nonstationary, and responsive to the system history; the only load
condition is uniform interior slack, meaning the conditional mean arrival
vector stays in a fixed contraction of the capacity region. We show that this slack reshapes the finite-time peak law for drift-minimizing scheduling policies such as MaxWeight. The square-root envelope that is sharp without slack persists only up to a geometry-dependent threshold; beyond that threshold, the running maximum grows only logarithmically with the horizon, both with high probability and in expectation.

The mechanism is self-normalization: in the current queue direction, the projected fluctuation scale is normalized by the stabilizing drift scale. This removes capacity geometry from the logarithmic coefficient, while geometry remains in the threshold. Matching lower bounds show that both the logarithmic term and a geometric threshold are unavoidable. When finite-time state-space collapse is available, the threshold can be sharpened using local bottleneck geometry. For generalized input-queued switches, we obtain finite-time peak bounds with tight logarithmic coefficients. Simulations illustrate the two-phase envelope, local geometric refinements, and variance-sensitive improvements predicted by the theory.
\end{abstract}

\section{Introduction}

Stochastic networks \cite{Kelly11Lectures,Williams16,DaiHarrison20} model systems in
which many queues compete for shared service resources: cloud clusters allocate compute
capacity to jobs, AI inference platforms allocate accelerator capacity to requests,
packet switches match input ports to output ports, and wireless systems allocate
time-frequency resources to users. Classical performance measures---throughput,
stability, steady-state backlog, and time-average cost---describe what a network can
sustain in the long run. Yet many modern systems operate over finite planning windows,
under bursty and highly uncertain demand, with scarce shared capacity and tight
service-level requirements. In this regime, even a brief congestion episode can
constitute an operational failure: a GPU queue spike can cause missed training or
inference deadlines; a burst of inference requests can violate latency guarantees; and
a congested port or wireless bottleneck can trigger packet drops, delay spikes, or
costly intervention, even when long-run averages remain benign. This motivates a
foundational finite-horizon question about extremes: over the next $T$ slots, how large
can the backlog become?

We study this question in the discrete-time generalized switch of \citet*{Stolyar04}, a
canonical stochastic-network model that represents service constraints through a general
capacity geometry. We retain this geometric framework but embed it in more flexible
dynamics. Arrivals follow a nonstationary \emph{adapted-arrival} model: the
next-slot arrival law may depend arbitrarily on the observed history, as in
feedback-driven traffic where demand or routing reacts to recent congestion. Service may
also be random after the scheduling decision, so the realized service in a slot may
fluctuate around the chosen feasible rate. Starting from $Q_0 = 0$, we seek explicit
guarantees for
\[
  \PeakT := \max_{0 \le t \le T} \normtwo{Q_t},
\]
the largest backlog observed up to time $T$. The load condition requires only that the
one-step conditional mean arrival vector remain uniformly $\eps$ inside the capacity
region---the set of arrival rates that the service resources can sustain on average.
The problem is therefore genuinely nonstationary and nonasymptotic: a steady state
need not exist, the arrival law may change with the history, and the guarantee must
control queue peaks for every finite $T$.

Finite-time control of queue peaks has important precursors in the stochastic-network literature. \citet*{ShahTsitsiklisZhong10,shah2014qualitative} proved maximal-excursion inequalities for Bernoulli switched networks and bandwidth-sharing networks under time-homogeneous arrivals. These analyses, however, are restricted to fixed stochastic input laws and specialized capacity geometries, rather than the adapted-arrival setting with general capacity geometry considered here.

The closest starting point for the present paper is \citet*{LTX25}, which placed the
peak question on a minimax, online-learning footing. Their framework accommodates general capacity geometry, nonstationary arrivals, and
post-decision random service, and yields a $\sqrt{T}$ upper bound on
$\PeakT$.%
\footnote{The statements in \cite{LTX25} impose independence across time.
  Theorem~\ref{thm:sqrt-peak} in Section~\ref{sec:model} records the direct extension
  of their $\sqrt{T}$ guarantee to the adapted-arrival setting considered here.}
That rate is sharp at the boundary of the capacity region: when $\eps = 0$ and there is
no interior slack to exploit, it is the correct robust law. But when the system
operates strictly inside capacity, a persistent $\sqrt{T}$ description can be overly
pessimistic, since it still permits the backlog to grow polynomially with the horizon
despite sustained slack.

Classical queueing, scheduling, and stochastic-network theory point toward a much
sharper picture, but they address a different question. The MaxWeight policy of
\citet*{TassiulasEphremides92} showed that interior slack $\eps > 0$ is sufficient for
stability, and subsequent heavy-traffic and diffusion theory for generalized switches
and processing networks identified how bottleneck geometry governs steady-state
performance
\cite{Stolyar04,MaguluriSrikant16,MaguluriBurleSrikant18,Neely10,Williams16}. These
results form the standard asymptotic picture: positive recurrence, steady-state or
time-average bounds, and diffusion limits, typically derived under time-homogeneous or
Markovian input models. They do not, however, answer the finite-time question that
motivates this paper: under adapted, nonstationary arrivals, how large can the backlog
grow before time $T$, with high probability?

This gap is not cosmetic. A bounded steady-state mean controls an average, not the
largest excursion along a finite sample path. The operational horizon may be shorter
than the time needed to approach steady state, and in a time-varying environment
there may be no relevant steady state at all. The peak question therefore does not
merely call for a new interpretation of classical theory; it requires a genuinely new
finite-time theory. Such a theory should retain the generality and strengths of the recent finite-time viewpoint of \citet*{LTX25}---distribution-free over adapted arrivals, explicit in the horizon, and stated directly for the running maximum---while sharply capturing the slack- and geometry-sensitive behavior that classical queueing theory identifies as central.

\subsection{Our contributions}

{We develop a finite-time theory of queue peak laws under interior slack. Our upper bounds are achieved using the MaxWeight policy \cite{TassiulasEphremides92},\footnote{More broadly, the main upper bound extends to any scheduling policy satisfying the directional certificate in Remark~\ref{rem:lyapopt-main}. This includes the LyapOpt policy \cite{LTX25}.} while our lower bounds hold for every scheduling policy. Together, they identify the structure of optimal finite-time peak behavior, summarized below.}

{\paragraph{A new two-phase characterization.} The first message is a two-phase law enabled by interior slack. The robust $\sqrt{T}$ law of \cite{LTX25} remains the right slack-free characterization, but under interior
slack it becomes only a burn-in upper envelope: once the queue length exceeds a threshold, the queue peak grows only
logarithmically. For MaxWeight, our characterization implies\footnote{\cite{LTX25} shows that MaxWeight's dependence on $S_{\rm max}$ is fundamental when $\eps=0$, whereas its LyapOpt policy avoids this dependence in the $\sqrt{T}$ law. Our paper focuses on optimizing the dependence on $T$, $\eps$, and $\alpha(\Pi)$. Whether another policy can improve the $A_{\rm max}$ and $S_{\rm max}$-dependence over MaxWeight is an interesting future direction.}
\begin{equation}\label{eq:main}
\E[\PeakT]
\;\lesssim\;
\min\left\{
D\sqrt{T},
\;
\frac{D}{\eps}\log(T+1)+\frac{D^2}{\eps\alphapi}
\right\},
\qquad D:=A_{\max}+S_{\max}.
\end{equation}
Here \(A_{\max}\) and \(S_{\max}\) bound the one-slot realized arrival and service norms, \(\eps\) is the interior slack, and \(\alphapi\) is the global service margin---the smallest projected service capacity over
nonnegative unit directions---associated with the capacity region $\Pi$. The term
\(D^2/(\eps\alphapi)\) is the threshold. Above it, running the process longer costs only \((D/\eps)\log T\). See Figure~\ref{fig:envelope-sketch} for an illustration.}

The upper bound is complemented by policy-independent lower bounds. A one-dimensional reflected walk shows that both the initial $\sqrt{T}$ regime and the later $\log T$ regime are unavoidable. A separate deterministic construction shows that no theorem at this level of generality can eliminate geometric dependence, such as $\alpha(\Pi)$, from the threshold. Thus, the structure of \eqref{eq:main} is essentially sharp: any general finite-time peak law must include a burn-in envelope, a logarithmic envelope, and a geometric threshold. Conversely, \eqref{eq:main} should be interpreted as a universal upper envelope, not as a claim that every instance undergoes an exact $\sqrt{T}$-to-$\log T$ transition.

{
\paragraph{Geometry-free logarithms via self-normalization.}
A notable feature of \eqref{eq:main} is that the logarithmic coefficient \(D/\eps\) is \emph{geometry-free}: it is universal over capacity geometries and, in particular, independent of \(\alpha(\Pi)\). This is not a generic consequence of classical stability. A direct Lyapunov drift argument would typically carry global geometry into the logarithmic term. MaxWeight avoids this loss through \emph{self-normalization}. In the current queue direction \(u_t=Q_t/\|Q_t\|_2\), it attains the projected service capacity \(H(u_t)\): the stabilizing drift is of order \(\eps H(u_t)\), while the projected fluctuation scale is controlled by \(D H(u_t)\). Since the same factor \(H(u_t)\) controls both drift and fluctuation, a directional exponential-supermartingale calculation cancels it, yielding the universal coefficient \(D/\eps\). Geometry governs the threshold at which self-normalization becomes effective, not the logarithmic growth after that threshold.
}

{
A classical queueing reader may expect logarithmic growth of running maximum from an exponential steady-state tail \cite{ShahTsitsiklisZhong10,shah2014qualitative}. The point here is that this intuition extends to a finite-time, distribution-free, geometry-agnostic form. The coefficient \(D/\eps\) holds without asymptotics, heavy traffic, stationarity, or even the existence of a steady state, and it holds uniformly over all capacity geometries. Thus the logarithmic law is not merely a steady-state consequence; it is a new finite-time principle.
}

\paragraph{Improved geometric thresholds via state-space collapse.}
In \eqref{eq:main}, self-normalization has already removed geometry from the logarithmic coefficient; the remaining geometric cost appears only through the entrance threshold $D^2/(\eps\alphapi)$. This raises a natural question: is the global margin $\alphapi$ the right complexity measure for a given network? For fully uniform guarantees it is, but it can be overly conservative: it protects against all nonnegative directions, including directions that the queue process may never visit. State-space collapse (SSC) \cite{Reiman84SSC} is the classical phenomenon whereby a high-dimensional queueing process remains close to a lower-dimensional bottleneck set. In this sense, SSC is an adaptive form of dimension reduction induced by network geometry. Under a quantitative finite-time complete resource pooling (CRP) condition \cite{Stolyar04,EryilmazSrikant12}, we show that this collapse can be converted directly into a sharper peak law: the global worst-direction margin \(\alphapi\) in the entrance threshold is replaced by the local geometry of the active bottleneck face, together with the finite-time collapse error needed to keep the queue near that face.

Thus, SSC provides a principled way to replace global worst-case geometry by the geometry actually seen by the queue, without relying on asymptotics or i.i.d./Markovian input laws. This \emph{finite-time}, \emph{nonstationary} extension of SSC is a methodological contribution of independent interest.

{\paragraph{Application to input-queued switches and beyond.}
We instantiate the framework on input-queued switches (IQS), a canonical model for wireless scheduling and internet routing. In a generalized IQS model allowing adversarial, dependent, and fractional arrivals, we prove matching upper and lower bounds for the logarithmic coefficient of the total-backlog peak. The bounds become sharper under the quantitative CRP geometry, and also under independent Bernoulli arrivals. Simulations for IQS and for more general parallel-server networks closely track the theory, illustrating the two-phase envelope, the gains from local geometry, and the variance-sensitive improvements.}

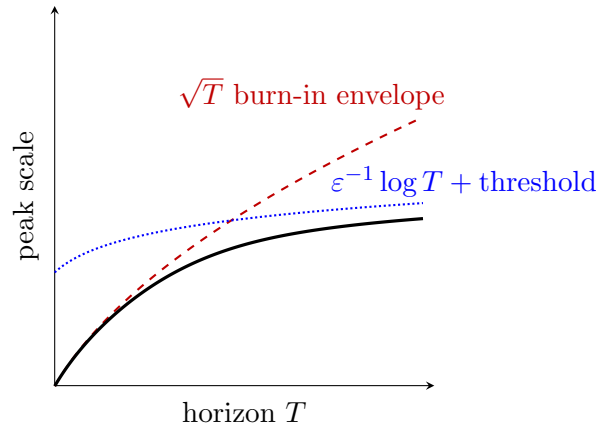
\begin{figure}[htbp]
\centering
\begin{tikzpicture}
\begin{axis}[
  clip=false,
  axis lines=left,
  xlabel={horizon $T$},
  ylabel={peak scale},
  xtick=\empty,
  ytick=\empty,
  xmin=0.5, xmax=7.2,
  ymin=0.6, ymax=3.2,
  samples=300,
  domain=0.5:7,
  width=0.4\textwidth,
  height=0.4\textwidth,
  no markers,
  legend style={draw=none, fill=none},
]

 \addplot[
  red!75!black,
  thick,
  dashed
]
{0.95*sqrt(x)-0.075};

\node[red!75!black, anchor=west] at (axis cs:2.5,2.6)
{$\sqrt{T}$ burn-in envelope};

\addplot[
  blue,
  thick,
  densely dotted
]
{0.18*ln(x) + 1.50};

\node[blue, anchor=west] at (axis cs:5.2,2)
{$\varepsilon^{-1}\log T + \text{threshold}$};

\addplot[
  black,
  very thick
]
{(1-(1-exp(-1.5*(x-0.5))))*(0.95*sqrt(x)-0.075)
+
(1-exp(-1.5*(x-0.5)))*
(
  -0.2*ln(
    exp(-(0.95*sqrt(x)-0.075)/0.2)
    +
    exp(-(0.18*ln(x)+1.40)/0.2)
  )
)};

\end{axis}
\end{tikzpicture}

\caption{An illustration of the two-phase upper envelope given by \eqref{eq:main}. The red dashed curve represents the slack-free $\sqrt{T}$ burn-in envelope from Theorem~\ref{thm:sqrt-peak}. The blue dotted curve represents the slack-sensitive logarithmic envelope from Corollary~\ref{cor:two-phase-main}, shifted by the queue-length threshold $D^2/(\eps\alphapi)$. The black curve illustrates a possible realized or typical growth pattern of the queue peak as the horizon $T$ increases; we do not claim an exact $\sqrt{T}$-to-$\log T$ transition for every instance.}
\label{fig:envelope-sketch}
\end{figure}

\subsection{Relation to prior work}\label{sec:related_work}

This paper sits at the intersection of two complementary viewpoints. Classical stochastic-network theory captures the roles of slack, bottlenecks, and capacity geometry, but mainly through asymptotic, stationary, or steady-state guarantees: the MaxWeight and heavy-traffic literatures study stability, time averages, steady-state backlog, tails, and diffusion limits, typically under fixed or Markovian input laws
\cite{TassiulasEphremides92,Stolyar04,MaguluriSrikant16,
MaguluriBurleSrikant18,Neely10,Williams16,DaiHarrison20}. Recent finite-time theory turns instead to sample-path questions over a fixed horizon, but existing results either rely on fixed input laws and specialized geometries
\cite{ShahTsitsiklisZhong10,shah2014qualitative}, or apply in much greater generality while giving a robust \(\sqrt T\) envelope rather than a slack- and geometry-sensitive logarithmic law \cite{LTX25}. We show that the two perspectives can be combined: under adapted arrivals, interior slack and bottleneck geometry yield finite-time peak bounds that cross over from a \(\sqrt T\) burn-in to \(\log T\) growth after a geometric entrance threshold.

The logarithmic scale has a natural queueing interpretation. For stable queues,
exponential steady-state tail bounds suggest that the maximum over a window of
length $T$ should grow on the order of $\log T$. Our theorem reaches the same
extreme-value scale without assuming that a steady-state distribution exists:
the logarithm is generated directly along the finite-horizon sample path by a
self-normalized drift argument. Remark~\ref{rem:ss-analogue} makes this
comparison precise. The result also identifies the sharp coefficient of the
$\log T$ term, and therefore captures not only the scale but also the
multiplicative prefactor.

The concentration tool behind our finite-horizon drift analysis is a Hajek-type scalar drift-to-peak lemma; see Appendix~\ref{app:scalar}. \citet*{Hajek82Drift} proved exponential first-hitting and occupation-time bounds for real-valued processes with negative drift above a threshold and exponentially controlled increments. Since the peak event $\{\max_{0\le t\le T} X_t\ge x\}$ can be written as a hitting event $\{\tau_x\le T\}$, where $\tau_x:=\inf\{t:X_t\ge x\}$, our scalar lemma should be viewed as a finite-horizon peak formulation of \cite{Hajek82Drift}, with the stopping-time bookkeeping tailored to the peak bounds needed here. The new technical challenges addressed in this paper lie not in the scalar concentration argument itself, but in how it is coupled to high-dimensional, controlled stochastic networks. Sections~\ref{sec:selfnorm} and~\ref{sec:geometry} explain these contributions in detail.

A distinct finite-time perspective was recently developed by \cite{nguyen2025finite} for the Erlang-C, or \(M/M/n\), queue. They establish quantitative convergence to steady state in \(\chi^2\)-distance, and use these mixing estimates to obtain finite-time bounds on time-average queue lengths and fixed-time tails. Their results address when steady-state approximations become valid for a reversible birth-death Markov chain. Our question is complementary: rather than establishing distributional convergence to equilibrium, we directly bound the running maximum of a finite-horizon sample path. Moreover, in our adapted-arrival model, a stationary distribution need not exist, and our results apply to more general controlled stochastic networks.\footnote{Technically, our direct peak analysis also avoids the large cold-start
factor that arises when \(\chi^2\)-mixing bounds are used to convert
fixed-time tail estimates into running-maximum bounds.}

Worst-case approaches such as adversarial queueing and network calculus
\cite{BorodinEtAl01,LeBoudecThiran01} provide another form of finite-time robustness.
They deliberately avoid slack conditions, and therefore give robust envelopes rather
than the slack-sensitive logarithmic law proved here.

Our input-queued-switch applications are informed by a line of work on IQS
geometry, state-space collapse, and heavy-traffic behavior
\citep{ShahTsitsiklisZhong11,ShahTsitsiklisZhong14,ShahWaltonZhong14,
MaguluriSrikant16,MaguluriBurleSrikant18,XuZhong20}. We defer a more detailed
comparison with this literature to Section~\ref{sec:iqs}.

Finally, the paper also points to a useful bridge between online learning \cite{cesa2006prediction,hazan2016introduction} and queueing \cite{walton2021learning}.
Online learning supplies the finite-time standard: guarantees should hold along a single
path, with minimal probabilistic structure, and with rates governed by
the complexity of the instance or instance family. Queueing supplies a different set of ideas: persistent
state, negative drift, bottleneck geometry, and collapse. The generalized switch is a
natural meeting point. Actions create future state, as in reinforcement
learning and control, but the feasible actions retain enough geometric structure to permit sharp
pathwise guarantees. From this viewpoint, \cite{LTX25} gives the minimax theory of queue peaks.
The present paper gives the sharper instance-dependent theory: under interior
slack and the right geometry, the worst-case \(\sqrt T\) scale gives way to \(\log T\).

\subsection{Organization and notation}

Section~\ref{sec:model} states the main upper and lower bounds and gives the two-phase envelope for finite-time queue peaks. The resulting characterization isolates three fundamental features: square-root burn-in, geometry-free logarithmic growth, and a geometry-dependent entrance threshold. Section~\ref{sec:selfnorm} explains why the logarithmic coefficient is geometry-free through self-normalization, and identifies queue-Bernstein conditions as a natural unbounded-arrival class for which this mechanism survives. Section~\ref{sec:geometry} shows how state-space collapse can sharpen the geometric threshold: a two-queue example illustrates how favorable local geometry removes the global margin, and the CRP theorem converts finite-horizon collapse into a local peak law. Section~\ref{sec:iqs} specializes the framework to generalized input-queued switches, proving sharp logarithmic peak bounds in the robust model, CRP refinements, and a variance-sensitive Bernoulli bound. Section~\ref{sec:experiments} reports simulations of the predicted two-phase envelope, the gap between global and local geometry, and the effects of bottleneck structure and arrival variance. Section~\ref{sec:discussion} closes with the main open directions.

We use the following global conventions. Vector inequalities are coordinatewise, $x^+$ denotes coordinatewise projection onto $\Rp^d$, $\ip{\cdot}{\cdot}$ is the Euclidean inner product, and $\norm{\cdot}_p$ is the $\ell_p$ norm. When $A_t$ is a vector, $A_i(t)$ denotes its $i$th coordinate. In Section~\ref{sec:iqs}, where $A_t$ is a matrix, $A_{ij}(t)$ denotes its $(i,j)$ entry. The same convention applies to the mean vector or matrix $\lambda_t$ when its entries are used. Throughout, $c$ and $C$ denote positive universal numerical constants that may change from line to line. We write $f\lesssim g$, $f\gtrsim g$, and $f\asymp g$ to mean $f\le Cg$, $f\ge cg$, and $cg\le f\le Cg$, respectively. Model-specific notation is introduced in Section~\ref{sec:model}.

\section{Model and main results}\label{sec:model}

We work with a discrete-time generalized switch \cite{Stolyar04}, initialized at $Q_0=0$. Building on the general formulation of \cite{LTX25}, we allow arrivals to be nonstationary and history-dependent, and we allow service randomness to be resolved only after the scheduling decision in each slot. This pathwise formulation differs from the classical generalized-switch setting, which typically assumes i.i.d. arrivals and service uncertainty resolved before the decision \cite{Stolyar04,HurtadoLangeVarmaMaguluri22}. In this sense, the model is also a stateful online learning problem; see Section~\ref{sec:related_work}.

\paragraph{Dynamics.}
Let $\Sset$ be a finite set of feasible mean-service vectors (induced by the shared service resources). At slot $t=0,1,2,\dots$, the scheduler observes the queue vector $Q_t\in\Rp^d$, chooses a mean-service vector $s_t\in\Sset$, and then the arrival vector $A_t\in\Rp^d$ and realized service vector $S_t\in\Rp^d$ are revealed. The queue evolves as
\begin{equation}\label{eq:queue-recursion}
Q_{t+1}=(Q_t+A_t-S_t)^+,
\end{equation}
where $(\cdot)^+$ denotes coordinatewise projection onto $\Rp^d$. Let $\F_t$ be the pre-decision history at slot $t$, i.e.,
\[
\F_t:=\sigma\bigl(Q_0,(s_\tau,A_\tau,S_\tau):0\le \tau<t\bigr).
\]
Thus $Q_t$ is $\F_t$-measurable, the decision $s_t$ is $\F_t$-measurable, and $A_t$ and $S_t$ are revealed after the decision. The realized service satisfies
\begin{equation}\label{eq:service-mean}
\E[S_t\mid\F_t]=s_t
\qquad\text{a.s.}
\end{equation}
We write
\[
\lambda_t:=\E[A_t\mid \F_t]
\]
for the one-step conditional mean arrival vector.

\paragraph{Policy.}
Throughout the main upper bounds the scheduler uses the MaxWeight policy \cite{TassiulasEphremides92}:
\begin{equation*}\label{eq:MW}
s_t\in\argmax_{s\in\Sset}\ip{Q_t}{s}.
\end{equation*}
The scheduler need not know $\lambda_t$. Extensions beyond MaxWeight are discussed in Remark~\ref{rem:lyapopt-main}.

\paragraph{Capacity geometry.}
The capacity region is the downward closure of the convexified service set,
\begin{equation}\label{eq:capacity-downward}
\PiC
:=
\bigl\{x\in\Rp^d:\text{there exists }r\in\conv(\Sset)\text{ such that }0\le x\le r\text{ coordinatewise}\bigr\}.
\end{equation}
This is the usual stability region: by classical stochastic-network theory \cite{Dai99,DaiHarrison20}, arrival-rate
vectors in the interior of \(\PiC\) are stabilizable, while rates outside \(\PiC\)
cannot be supported. We use this convention throughout.

For nonnegative directions, define the support function
\begin{equation*}\label{eq:Hdef}
H_\Pi(q):=\max_{s\in\mathcal S}\langle q,s\rangle
=
\max_{x\in\PiC}\langle q,x\rangle,
\qquad q\in\mathbb R_+^d .
\end{equation*}
When $u$ is a unit queue direction, $H_\Pi(u)$ is the projected service capacity in direction $u$. Since the region is fixed within the model, we write $H(q)=\Hpi(q)$. MaxWeight gives the identity $\ip{Q_t}{s_t}=H(Q_t)$, which is the directional fact used throughout the proofs. The global Euclidean service margin is
\begin{equation*}\label{eq:alphaPi-def}
\alphapi:=\alphatwopi
:=\inf\{\Hpi(w): w\in\Rp^d,\ \normtwo{w}=1\}.
\end{equation*}
Thus $\alphapi$ measures the capacity in the worst-served nonnegative Euclidean unit direction. We will show that it plays a fundamental role in connecting geometry with queue peaks. Whenever $\alphapi$ appears, we work on the served subspace and assume $\alphapi>0$.

\paragraph{Interior slack and bounded primitives.}
The next assumption is the finite-horizon analogue of operating strictly inside capacity. It constrains the arrivals only in conditional mean: individual realizations may still lie outside the capacity region.

\begin{assumption}[Uniform conditional interior slack]\label{ass:interior}
There exists $\eps\in(0,1)$ such that
\begin{equation}\label{eq:interior}
\lambda_t\in(1-\eps)\PiC
\qquad\text{a.s. for every }t\ge0.
\end{equation}
Equivalently, for every $q\in\Rp^d$,
\begin{equation*}\label{eq:support-slack}
\ip{q}{\lambda_t}\le (1-\eps)H(q)
\qquad\text{a.s.}
\end{equation*}
\end{assumption}
The main upper bounds are derived based on the following boundedness assumptions on arrivals and service; we extend our results to the unbounded case in Section~\ref{sec:qb-unbounded}.
\begin{assumption}[Bounded post-decision primitives]\label{ass:bounded-primitives}
There are finite constants $A_{\max}$ and $S_{\max}$ such that, for every $t$,
\begin{equation*}\label{eq:bounded-primitives}
\normtwo{A_t}\le A_{\max},
\qquad
\normtwo{S_t}\le S_{\max}
\qquad\text{a.s.}
\end{equation*}
and $\sup_{s\in\mathcal{S}}\normtwo{s}\le S_{\max}$. Moreover, conditional on $\F_t$, the arrival randomness $A_t$ and the post-decision service randomness $S_t$ are independent.
\end{assumption}

\subsection{The self-normalized peak law}

The first main bound is the logarithmic law. 

\begin{theorem}[Self-normalized logarithmic peak law]\label{thm:selfnorm-main}
Consider MaxWeight scheduling under Assumptions~\ref{ass:interior} and~\ref{ass:bounded-primitives}. Let $D:=A_{\max}+S_{\max}$. For every $T\ge1$ and $\delta\in(0,1)$, with probability at least $1-\delta$,
\[
\PeakT
\;\lesssim\;
\frac{D}{\eps}\log\frac{T+1}{\delta}
+
\frac{D^2}{\eps\alphapi}.
\]
Consequently,
\begin{align*}
  \E[\PeakT]
\;\lesssim\;
\frac{D}{\eps}\log(T+1)
+
\frac{D^2}{\eps\alphapi}.
\end{align*}
If service is deterministic, the logarithmic coefficient $D/\eps$ can be replaced by $A_{\max}/\eps$.
\end{theorem}

In Theorem~\ref{thm:selfnorm-main}, geometry enters only through the entrance level $D^2/(\varepsilon\alpha(\Pi))$. Once the process is above that level, the coefficient of $\log T$, $D/\varepsilon$, is independent of $\alpha(\Pi)$. The main mechanism behind Theorem~\ref{thm:selfnorm-main} will be explained in detail in Section~\ref{sec:selfnorm}.

We make two remarks below.

\begin{remark}[An exponential steady-state tail analogue without steady state]\label{rem:ss-analogue}
Exponential steady-state tails for positive recurrent queues \cite{ShahTsitsiklisZhong10,shah2014qualitative} imply that the maximum over $T$ approximately steady-state observations should scale as $\log T$. Theorem~\ref{thm:selfnorm-main} gives a finite-time counterpart to this heuristic: it requires neither stationarity, positive recurrence, nor a fixed arrival law, and it applies under general network geometry. Thus, the logarithmic term in our bound should not be viewed as a finite-horizon shadow of an equilibrium tail estimate. It is produced directly by the self-normalization argument in Section~\ref{sec:selfnorm}, which identifies a foundational finite-time mechanism for logarithmic peak growth.
\end{remark}

\begin{remark}[Extension beyond MaxWeight]\label{rem:lyapopt-main}

We extend the proof to any $\F_t$-measurable policy satisfying a directional certificate given in Appendix~\ref{app:lyapopt-certificate}. As an example, Appendix~\ref{app:lyapopt-certificate} verifies this certificate for the LyapOpt policy  \cite{LTX25}, so the same logarithmic law applies to LyapOpt.
\end{remark}

\subsection{The finite-time two-phase envelope}

The logarithmic law is a finite-time theorem on its own, but its entrance threshold matters. When \(T\) is small relative to a polynomial scale in \(1/\eps\), or when \(\eps\) is small, or when the global service margin \(\alphapi\) is small, the minimax \(\sqrt T\) envelope of \citet*{LTX25} can be smaller than the logarithmic bound because the threshold term \(D^2/(\eps\alphapi)\) dominates. It is therefore natural to combine the robust \(\sqrt T\) bound with our logarithmic bound, taking the better of the two, to obtain a sharper cross-regime characterization of finite-time queue peaks.

\begin{theorem}[A slack-free $\sqrt{T}$ envelope for the peak]\label{thm:sqrt-peak}
Consider MaxWeight scheduling under Assumption~\ref{ass:bounded-primitives}. Assume only the boundary-load condition $\lambda_t\in\PiC$ a.s. for every $t$; no strict interior slack is used. Let $D:=A_{\max}+S_{\max}$. Then for every $T\ge1$ and every $r>0$,
\[
\Pbb\!\left(\PeakT\ge r\right)
\le \frac{D^2T}{r^2}.
\]
In particular, $\E[\PeakT]\le 2D\sqrt{T}$.
\end{theorem}

Appendix~\ref{app:upper-burnin} proves the estimate by telescoping the one-step quadratic drift and applying Doob's maximal inequality. This is the bounded-primitive version of the finite-time viewpoint of \citet*{LTX25}: before one uses slack, $\sqrt{T}$ growth is the natural robust envelope.

Combining Theorem~\ref{thm:sqrt-peak} with Theorem~\ref{thm:selfnorm-main} yields the two-phase upper envelope.

\begin{corollary}[Two-phase upper envelope]\label{cor:two-phase-main}
Consider MaxWeight scheduling under Assumptions~\ref{ass:interior} and~\ref{ass:bounded-primitives}. Assume $\alphapi>0$, and let $D:=A_{\max}+S_{\max}$. Then
\[
\E[\PeakT]
\;\lesssim\;
\min\left\{
D\sqrt{T},
\;
\frac{D}{\eps}\log(T+1)
+
\frac{D^2}{\eps\alphapi}
\right\}.
\]
With
$
x_{\mathrm{th}}:={D^2}/(\eps\alphapi)  
$ denoting the geometric threshold,
the transition between the two phases occurs near $T\asymp x_{\mathrm{th}}^2/D^2$, which in normalized bounded models is $T\asymp(\eps\alphapi)^{-2}$.
\end{corollary}

We emphasize that this corollary is an upper envelope, not a claim that every instance has a sharp deterministic transition time. See Figure~\ref{fig:envelope-sketch} for an illustration, and see Section~\ref{sec:experiments} for empirical queue peak scaling simulations.

\subsection{Why the two phases and the threshold are necessary}

The upper envelope has two independent obstructions. The first is probabilistic: even a one-dimensional queue has an initial diffusive regime and, after enough time, a logarithmic maximum. The second is geometric: even at constant slack, a scheduler may need to serve many weak directions before the backlog can be kept below the global threshold scale.

\begin{theorem}[One-dimensional lower bound at arbitrary arrival scale]\label{thm:1d-lower-main}
There exist universal constants $c,C>0$ such that the following holds. For every $A\ge1$,  $\eps\in(0,1/8]$, and $T\ge C A$, there is a one-dimensional queue in the present model, with deterministic unit service $S_t=s_t=1$, arrivals taking values in $\{0,A+1\}$, such that
\[
\mathbb E\Big[\max_{0\le t\le T}\|Q_t\|_1\Big]
\;\ge\;
c\,
\min\!\left\{
\sqrt {AT},\;
\frac{A}{\varepsilon}\log\!\Big(1+\frac{\varepsilon^2}{A}T\Big)
\right\}.
\]
\end{theorem}

Since $A_{\max}=A+1$, the logarithmic term may be written with $A_{\max}$ in place of $A$, up to universal constants. In particular, for horizons above $A_{\max}/\eps^2$, the deterministic-service coefficient $(A_{\max}/\eps)\log T$ is unavoidable. The logarithmic coefficient matches the deterministic-service upper bound in Theorem~\ref{thm:selfnorm-main}. A complementary $(S_{\max}/\eps)\log T$ lower bound can be obtained by a random-service construction, which we omit here.

\paragraph{Proof idea.}
The mechanism is clearest in the regenerative analysis of Appendix~\ref{app:1d}.
A typical excursion lasts \(O(1/\eps)\) steps, while its maximum has an exponential
tail with decay rate \(\Theta(\eps/A)\). Maximizing over the excursions seen by time
\(T\) gives the logarithmic term, and the same construction has a diffusive initial
regime, giving the square-root term.

\begin{theorem}[Geometric threshold lower bound under constant slack]\label{thm:geometry-lower-main}
There exists a universal constant $c>0$ such that, for every $d\ge2$, $\eps\in[1/4,1/2]$, there is a switched network with deterministic arrivals and services, satisfying Assumptions~\ref{ass:interior} and~\ref{ass:bounded-primitives}, with $A_{\max}\le1$, $S_{\max}=1$, and $\alphapi=1/\sqrt d$, such that the following holds. For any scheduling policy,
\[
\max_{0\le t\le T_0} \normtwo{Q_t}
\;\ge\;
\frac{c}{\eps\alphapi},
\]
where the horizon $T_0:=\floor{d/2}$ is polynomial in $1/\alphapi$.
\end{theorem}

\paragraph{Proof idea.}
The proof in Appendix~\ref{app:geometry-lower-bound} uses the simplex schedule set. In each slot a policy can serve at most one coordinate. By time $T_0=\floor{d/2}$, at least $d-T_0$ coordinates have never been served, and these untouched coordinates alone carry Euclidean backlog of order $\sqrt d=1/\alphapi$. Since $\eps$ is bounded away from zero, this is the same scale as $1/(\eps\alphapi)$.

\vspace{1em}
Therefore, the structure of Corollary~\ref{cor:two-phase-main} is essentially sharp: any general finite-time peak law must include a burn-in envelope, a logarithmic envelope, and a geometric threshold. Section~\ref{sec:selfnorm} explains why the logarithmic coefficient is already geometry-free, and Section~\ref{sec:geometry} shows how state-space collapse can replace the global service margin by local bottleneck geometry.

\section{Self-normalization and geometry-free logarithms}\label{sec:selfnorm}

We now explain the mechanism behind Theorem~\ref{thm:selfnorm-main}. The full proof appears in Appendix~\ref{app:upper-selfnorm}.

\subsection{Key mechanism: self-normalization}\label{sec:proof-idea}

\paragraph{Scalar drift-to-peak lemma.}
The one-queue case is the cleanest setting in which to see where the logarithm comes from. In a single queue, there is no network-level scheduling decision, and the peak question reduces to bounding excursions of a scalar adapted process with negative drift, as in a bounded-increment $G/G/1$ workload process whose traffic intensity is strictly below one. This scalar question is already covered, in essence, by Hajek's classical drift-to-hitting machinery~\cite{Hajek82Drift}, which applies to general real-valued processes. We use a finite-horizon drift-to-peak formulation of this mechanism, with stopping-time bookkeeping tailored to peak bounds. Appendix~\ref{app:scalar} gives the formal statement; the following simplified version captures the intuition.

Let $(X_t)_{t\ge0}$ be a nonnegative adapted process. Suppose that, above a level $x_0$, its conditional drift satisfies
\[
\E[X_{t+1}-X_t\mid \F_t]\le -\Delta,
\]
its upward jumps are bounded by $L$, and its centered increments have conditional Bernstein variance proxy $\sigma^2$. Then, with probability at least $1-\delta$,
\[
\max_{0\le t\le T}X_t
\;\lesssim\;
x_0
+L
+\frac{\sigma^2}{\Delta}
\log\frac{T+1}{\delta}.
\]
Thus the logarithmic coefficient is governed by a one-step fluctuation-to-drift ratio, $\sigma^2/\Delta$. For a single queue, this scalar estimate applies directly. In a stochastic network, however, the correct use is tricky: one must first identify a scalar coordinate of the high-dimensional queue whose drift and fluctuation can be controlled on the right scale. This choice of coordinate, and the cancellation it creates, is the self-normalization step.

\paragraph{The queue direction is the right coordinate.}
A direct quadratic-Lyapunov argument gives the right stability mechanism but loses the sharp logarithmic coefficient when inserted into the scalar drift-to-peak lemma. It bounds fluctuations by the worst one-step scale $D^2$ and drift by the worst directional margin $\eps\alphapi$. The resulting ratio $D^2/(\eps\alphapi)$ combines two worst cases that need not occur in the same queue direction. This mismatch, rather than any single bad direction by itself, is what injects unfavorable geometric dependence into the logarithmic coefficient. Appendix~\ref{app:upper-warmup} gives this calculation.

The self-normalized proof avoids this mismatch by keeping the current queue direction inside the estimate. Let
\[
X_t:=\|Q_t\|_2,
\qquad
u_t:=Q_t/X_t
\quad\text{when }X_t>0.
\]
In one slot the realized net input is $A_t-S_t$. When the queue is large, the Euclidean norm is well approximated by its directional derivative:
\begin{equation}\label{eq:selfnorm-norm-expansion-main}
X_{t+1}-X_t
\le
\ip{u_t}{A_t-S_t}
+\frac{\|A_t-S_t\|_2^2}{2X_t}.
\end{equation}
The first term is the net input projected onto the current queue direction. By Assumption~\ref{ass:interior} and the MaxWeight choice of $S_t$,
\begin{equation}\label{eq:neg_drift}
\E\!\left[\ip{u_t}{A_t-S_t}\mid \F_t\right]
\le
(1-\eps)H(u_t)-H(u_t)
=
-\eps H(u_t).
\end{equation}
The second term in \eqref{eq:selfnorm-norm-expansion-main} is the curvature cost of replacing the norm increment by its linear projection. It is quadratic in the one-slot movement and inverse in the current queue size. Since
\[
\|A_t\|_2\le A_{\max},
\qquad
\|S_t\|_2\le S_{\max},
\qquad
D=A_{\max}+S_{\max},
\]
this curvature term is at most $D^2/(2X_t)$. Hence, once
\[
X_t\gtrsim \frac{D^2}{\eps\alphapi},
\]
the curvature cost is dominated by the projected MaxWeight drift. Beyond this entrance threshold, the norm increment is governed by the projected net input
$
\ip{u_t}{A_t-S_t}.
$

The decisive point is that the same support value $H(u_t)$ controls both the conditional mean and the conditional variance of $
\ip{u_t}{A_t-S_t}
$. These two quantities determine, respectively, the drift scale and the fluctuation scale of $X_t$ once the process is above the threshold. The drift scale is $\eps H(u_t)$, as shown in \eqref{eq:neg_drift}. The fluctuation scale is also proportional to $H(u_t)$: for $\theta\lesssim 1/D$,
\begin{equation}\label{eq:selfnorm-mgf-main}
\log \E\!\left[
\exp\!\left(\theta\Bigl(
\ip{u_t}{A_t-S_t}
-
\E[\ip{u_t}{A_t-S_t}\mid\F_t]
\Bigr)\right)
\middle|\F_t\right]
\lesssim
\theta^2 DH(u_t).
\end{equation}
This is the self-normalization mechanism. The factor $H(u_t)$ measures how much service MaxWeight can extract in direction $u_t$; the same factor also measures the size of the projected randomness in that direction.

Thus a poorly served direction has weaker drift, but it also has proportionally weaker projected fluctuation. The scalar ratio is not formed by pairing the largest possible fluctuation with the smallest possible drift. Along the realized path, the relevant ratio is directional:
\[
\frac{\text{projected fluctuation scale at time }t}
{\text{directional drift scale at time }t}
\;\asymp\;
\frac{DH(u_t)}{\eps H(u_t)}
=
\frac{D}{\eps}.
\]
The support value cancels. Geometry still matters, but only through the entrance level needed for the norm expansion to be drift dominated. Once the process is beyond that threshold, the logarithmic coefficient is the local fluctuation-to-drift ratio $D/\eps$, uniformly over the queue directions $u_t$ exposed along the sample path. If service is deterministic after the scheduling decision, the service-noise contribution in \eqref{eq:selfnorm-mgf-main} disappears, and the coefficient improves from $D/\eps$ to $A_{\max}/\eps$.

\subsection{Beyond bounded primitives: queue-Bernstein conditions}\label{sec:qb-unbounded}

Boundedness (Assumption \ref{ass:bounded-primitives}) is only one route to \eqref{eq:selfnorm-mgf-main}. The real
structural requirement is mass sensitivity: a direction that carries little mean traffic should also carry little stochastic variation. Motivated by this requirement, we identify a new, general distribution class of random vectors that enable self-normalization, leading to an unbounded, variance-sensitive extension of Theorem~\ref{thm:selfnorm-main}.

\begin{definition}[Queue-Bernstein condition]\label{def:queue-bernstein}
Let \(Z\in\Rp^d\) be nonnegative and let \(\mathcal G\) be a sigma-field. Write \(m:=\E[Z\mid\mathcal G]\). We say that \(Z\) is \emph{queue-Bernstein} with parameters \((\nu,b)\) conditional on \(\mathcal G\) if, for every \(w\in[0,1]^d\) and every \(|\theta|<1/b\),
\begin{equation}\label{eq:QB}
\log \E\!\left[
\exp\!\left(\theta\ip{w}{Z-m}\right)
\middle| \mathcal G\right]
\le
\frac{\theta^2\nu\ip{w}{m}}{2(1-b|\theta|)}
\qquad\text{a.s.}
\end{equation}
Here \(1/0=\infty\). A deterministic primitive has parameters \((0,0)\).
\end{definition}

Queue-Bernstein can be understood as a strengthened version of the sub-exponential condition. The point of \eqref{eq:QB} is that every nonnegative projection has a Bernstein MGF whose variance proxy is proportional to its own mean. We provide many natural examples satisfying the queue-Bernstein condition below.

\begin{example}[Standard queue-Bernstein examples]\label{ex:qb-primitives}
The following primitives are queue-Bernstein conditionally on any past with respect to which their means are fixed. Appendix~\ref{app:qb-calculus} gives the calculations.
\begin{enumerate}[label=\textup{(\roman*)},leftmargin=2.2em,itemsep=0.25em]
\item \emph{Bernoulli and binomial coordinates.} Independent unit Bernoulli coordinates, and sums of independent unit Bernoulli variables, satisfy \eqref{eq:QB} with \((\nu,b)=(1,1)\).
\item \emph{Poisson coordinates.} Independent Poisson coordinates with arbitrary predictable means satisfy \eqref{eq:QB} with \((\nu,b)=(1,1)\).
\item \emph{Bounded total mass.} If \(\|Z\|_1\le L\) a.s., then \(Z\) is queue-Bernstein with \((\nu,b)=(L,L)\).
\item \emph{Bounded-batch compound Poisson input.} If \(Z=\sum_{k=1}^N B_k\), where \(N\) is Poisson and \(\|B_k\|_1\le L\) a.s., then \(Z\) is queue-Bernstein with \((\nu,b)=(L,L)\).
\item \emph{Gamma and exponential coordinates.} Independent Gamma coordinates with scales bounded by \(\beta\), including exponential coordinates, satisfy \eqref{eq:QB} with \((\nu,b)=(\beta,\beta)\).
\end{enumerate}
\end{example}

The following theorem
extends Theorem~\ref{thm:selfnorm-main} to general queue-Bernstein arrival and service classes.

\begin{theorem}[Peak law under queue-Bernstein conditions]\label{thm:qb-unbounded-main}
Consider MaxWeight scheduling under Assumption~\ref{ass:interior}. Conditional on $\F_t$, assume $A_t$ and $S_t$ are independent, $A_t$ is queue-Bernstein with parameters $(\nu_A,b_A)$ and mean $\lambda_t$, and $S_t$ is queue-Bernstein with parameters $(\nu_S,b_S)$ and mean $s_t$, uniformly over $t$. Let $x_{\rm QB}$ be the geometric threshold formula defined in \eqref{eq:xQB-app} (deferred to Appendix~\ref{app:qb-entrance-scale}). Then, for every $T\ge1$ and $\delta\in(0,1)$, with probability at least $1-\delta$,
\[
\PeakT
\lesssim\left[
  x_{\rm QB}
  +
  \left(b_A+b_S+\frac{\nu_A+\nu_S}{\eps}\right)
  \log\frac{T+1}{\delta}
\right].
\]
Moreover, if $\|A_t-S_t\|_2\le D$ a.s., then
\[
  x_{\rm QB}\lesssim\left(1+D+\frac{D^2}{\eps\alphapi}\right).
\]
Consequently, if the queue-Bernstein constants also satisfy $b_A+b_S+\nu_A+\nu_S\lesssim D$, the guarantee recovers the bounded self-normalized law, up to universal constants.
\end{theorem}


The bound is \emph{variance-sensitive}. Even in a bounded model, the Euclidean bound may see only a worst-case jump of size $D$, while \eqref{eq:QB} sees the projected conditional mean. When many small independent coordinates replace one large batch, $\nu_A$ can stay order one although $D$ grows with dimension. This is the improvement exploited for Bernoulli IQS in Section~\ref{sec:iqs} and illustrated in the experiments.

\paragraph{One-dimensional special case: recovering  Kingman's scale.} To further illustrate Theorem~\ref{thm:qb-unbounded-main}, we consider the one-dimensional queue as a special case, where the queue-Bernstein parameter $\nu$ can be put in familiar queueing language. Write the scalar workload recursion as
\begin{equation*}\label{eq:scalar-gg1-queue}
  Q_{t+1}=\bigl(Q_t+A_t-S_t\bigr)^+,
\end{equation*}
where $A_t$ is work added and $S_t$ is the service opportunity in the slot. This is the $G/G/1$ Lindley recursion.\footnote{Kingman's approximation is usually stated for the steady-state waiting time in the $G/G/1$ Lindley recursion. We write the same recursion as a scalar workload queue to keep notation aligned with the rest of the paper; the mathematics is the same.}

In one dimension there is no Euclidean curvature term, so the geometric entrance threshold disappears. Appendix~\ref{app:one-dimensional-qb} proves the following consequence of the queue-Bernstein condition: if $\lambda_t:=\E[A_t\mid\F_t]$, $s_t:=\E[S_t\mid\F_t]$, and $\lambda_t\le(1-\eps)s_t$, then, up to terms independent of $1/\eps$,
\[
  \PeakT
  \;\lesssim\;
  \frac{\nu_A+\nu_S}{\eps}\log\frac{T+1}{\delta}
\]
with probability at least $1-\delta$. For i.i.d. primitives, $\nu$ is the variance-to-mean scale. Thus, with $\tau:=\E[S]$, $\rho:=1-\eps\in[0.5,1)$, and squared coefficients of variation $c_s^2,c_a^2$, the dominant logarithmic coefficient has the same order as
\begin{equation*}
\frac{\nu_A +\nu_S}{\eps}
\asymp
\frac{\rho}{1-\rho}\frac{c_s^2+c_a^2}{2}\,\tau
\end{equation*}
where $\asymp$ means equality up to universal numerical constants. This is the Kingman scale \cite{Kingman61}. Its role here is different: it is not an asymptotic approximation to the steady-state mean, which is strictly smaller, but the nonasymptotic logarithmic coefficient governing the finite-horizon peak.

\section{Geometric thresholds and state-space collapse}\label{sec:geometry}

Section~\ref{sec:selfnorm} showed that the logarithmic term need not carry the global geometry of the capacity region. What remains is the entrance threshold. The global theorem, Theorem~\ref{thm:selfnorm-main}, waits until the queue is large enough to certify drift uniformly over all directions; this costs the global margin $\alphapi$. That cost is unavoidable without further structure, as Theorem~\ref{thm:geometry-lower-main} shows. In many networks, however, the queue does not explore all directions. It accumulates near the bottleneck set selected by the active capacity constraints. Then the final peak bound should pay for entering a neighborhood of that local set, not for the worst margin of the whole $\Pi$.

State-space collapse (SSC) is the geometric statement that makes this refinement precise. It says that, over the relevant time horizon, the queue does not explore the whole state space; instead, it remains close to a lower-dimensional bottleneck set determined by the active face of the capacity region. This principle is classical in heavy-traffic theory; see \citet*{Reiman84SSC} and \citet*{Williams16} for background, and \citet*{Stolyar04} and \citet*{EryilmazSrikant12} for MaxWeight-type collapse arguments. Here we extend the machinery to a finite-time, nonstationary form. On a high-probability collapse event, the queue stays within a controlled neighborhood of the local bottleneck set throughout the horizon. We then analyze the queue peak accumulated inside the neighborhood, where the global margin can be replaced by the local bottleneck geometry.

\subsection{Two queues: collapse removes the global margin}

The following two-queue example is the smallest setting in which the global and local geometries separate. The parameter $a$ makes the capacity triangle thin, so the global margin satisfies $\alphapi\asymp a$. A geometry-free theorem that must protect against every direction therefore predicts an entrance cost of order $1/a$. The peak bound proved in this subsection does not.

\begin{example}[A two-queue geometry]\label{ex:two-queue}
Fix $a\in(0,1]$ and let
\[
\Sset=\{(1,0),(0,a)\}\subset\Rp^2.
\]
Throughout this example service is deterministic. With our capacity-region convention,
\[
\PiC=\{r\in\Rp^2:r_1+r_2/a\le1\},
\qquad
H(q)=\max\{q_1,aq_2\},
\qquad
\alphapi=\frac{a}{\sqrt{1+a^2}}\asymp a.
\]
Recall that $A_i(t)$ denotes the $i$th component of $A_t$. Assume
\[
A_1(t)\in\{0,1\},
\qquad
0\le A_2(t)\le1
\qquad\text{a.s. for every }t.
\]
Assume also that the conditional means satisfy
\begin{equation}\label{eq:twoqueue-slack}
\E[A_1(t)\mid\F_t]+\frac{1}{a}\E[A_2(t)\mid\F_t]\le1-\eps
\qquad\text{a.s. for every }t.
\end{equation}
\end{example}

Figure~\ref{fig:twoqueue-ssc-shadow} illustrates the geometry. The green triangle is the capacity region $\PiC$. Its upper green segment lies on
\(
r_1+r_2/a=1\); this is the exposed/active face associated with the load constraint in \eqref{eq:twoqueue-slack}. Its normal is $v:=(a,1)$. Define
\[
Y_t:=\ip{v}{Q_t}=aQ_{1,t}+Q_{2,t}.
\]
Thus $Y_t$ is the queue length projected onto the bottleneck normal. A bound on $Y_t$ alone is not enough to tightly control $\PeakT$: from $Y_t\le M$ we can only get $Q_{1,t}\le M/a$, which still pays $1/a$. The \emph{collapse strip} in Figure~\ref{fig:twoqueue-ssc-shadow} is the mechanism that removes this global-margin loss: state-space collapse confines \(Q_t\) to the strip, so the analysis no longer has to protect directions that the process never reaches.

\begin{figure}[H]
\centering
\begin{tikzpicture}[scale=1.05, line cap=round, line join=round]

  \draw[->] (0,0) -- (5.2,0) node[below] {$q_1$};
  \draw[->] (0,0) -- (0,4.4) node[left] {$q_2$};
  \fill[black!10] (0,0) -- (0,4.2) -- (2.88, 4.2) -- (2.08,0) -- cycle;
    \coordinate (A) at (0.0,0.4);
    \coordinate (B) at (2,0.0);
  \fill[green!16] (0,0) -- (A) -- (B) -- cycle;
  \draw[blue!70!black,thick,dashed] (2.08,0) -- (2.88,4);
  \node[blue!70!black,anchor=west] at (2.58,2.3) {$q_1-aq_2 =1+a^2$};
      \node[black] at (-0.2,0.17) {$a$};
    \node[black] at (1,-0.25) {$1$};
  \node[black!70] at (1.25,4) {collapse strip};

  \draw[->,very thick,red!70!black] (1,0.2) -- (1.36,2);
  \node[red!70!black,anchor=west] at (1.1,2.35) {$v$};
    \foreach \p in {A,B} {
  \fill[red!85!black] (\p) circle (1.5pt);
}
\draw[green!40!black, thick] (A)--(B);

\end{tikzpicture}
\caption{The local geometry behind the two-queue bound. The green triangle is the capacity region, and the green segment is the bottleneck face associated with the load constraint. The red normal $v=(a,1)$ defines the projected queue length $Y=\langle v,q\rangle$. Without the strip, $Y\le M$ permits $q_1$ as large as $M/a$; inside the strip $q_1\le aq_2+O(1)$, the same projection controls $q_1+q_2$ at constant cost.}
\label{fig:twoqueue-ssc-shadow}
\end{figure}
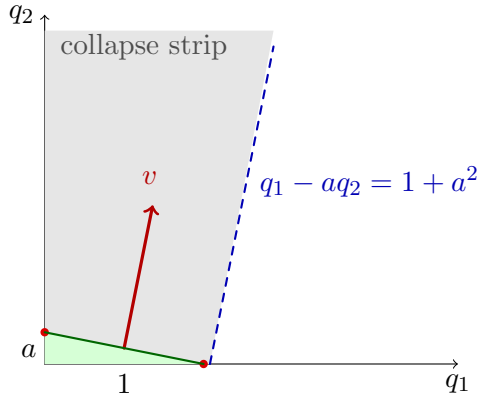

\begin{lemma}[Two-queue finite-time collapse and local drift]\label{lem:twoqueue-ssc}
In the deterministic-service two-queue instance of Example~\ref{ex:two-queue}, the following two facts hold for every sample path.
\begin{enumerate}[leftmargin=2.2em,itemsep=0.25em,topsep=0.25em]
\item \emph{Collapse strip.} The orthogonal overshoot is uniformly bounded:
\[
0\le (Q_{1,t}-aQ_{2,t})^+\le1+a^2
\qquad\text{for every }t\ge0.
\]
\item \emph{No-waste projection drift.} If $Y_t\ge2$, then the chosen service is fully used in the bottleneck direction and
\[
Y_{t+1}=Y_t+aA_1(t)+A_2(t)-a.
\]
Consequently,
\[
\E[Y_{t+1}-Y_t\mid\F_t]\le -a\eps
\qquad\text{on }\{Y_t\ge2\}.
\]
\end{enumerate}
\end{lemma}

The lemma separates the two ingredients. The no-waste identity gives negative drift for the scalar projection once $Y_t$ is above a constant. The collapse strip gives the geometric conversion. Indeed, on the strip,
\[
Q_{1,t}+Q_{2,t}\le (1+a)Y_t+O(1),
\]
whereas without the strip the same projected bound would lose a factor $1/a$. This is why the final peak bound below has no global-margin term.

\begin{proposition}[A local threshold in the two-queue example]\label{prop:twoqueue-peak}
For the deterministic-service two-queue instance in Example~\ref{ex:two-queue}, for every $T\ge1$ and $\delta\in(0,1)$, with probability at least $1-\delta$
\[
\PeakT
\;\lesssim\;
1+\frac{1}{\eps}\log\frac{T+1}{\delta}.
\]
In particular, the bound is uniform in the anisotropy parameter $a$, even though $\alphapi\asymp a$.
\end{proposition}

\paragraph{Proof sketch.}
The scalar recursion for $Y_t$ has drift of order $a\eps$ above a constant level. Its upward fluctuations have variance proxy $O(a)$ because the projection sees $aA_1(t)+A_2(t)$. The factors of $a$ cancel in the Bernstein peak estimate, giving
\[
\max_{t\le T}Y_t\lesssim 1+\frac1\eps\log\frac{T+1}{\delta}
\]
with high probability. The collapse strip then converts this projected bound to the stated total-backlog bound. Appendix~\ref{app:prep2queue} gives the details.

\subsection{Single-bottleneck collapse and complete resource pooling}
We now pass from the two-queue example to a generalized switch. The global theorem pays the margin \(\alphapi\) because, without further structure, a large queue can point in any nonnegative direction. State-space collapse changes the geometry. If the dynamics keep the queue close to its active bottleneck set, then the peak is governed by the local bottleneck geometry, together with the finite-horizon cost of keeping the queue near that set.

Complete resource pooling (CRP) is the cleanest setting in which this reduction becomes one-dimensional. It is the canonical single-bottleneck condition in the heavy-traffic theory of state-space collapse \cite{Reiman84SSC,Stolyar04,EryilmazSrikant12,Williams16}. Under CRP, the saturated capacity face has a normal cone generated by a single ray. The queue may fluctuate, but its large component is pulled toward this bottleneck ray. Thus the high-dimensional peak problem separates into two tasks: prove that the queue remains in a tube around the ray, and, within that tube, control the one-dimensional bottleneck excursion while paying the tube radius to translate this control back to the full queue peak.

We answer these questions in that order. First, CRP gives finite-horizon SSC: with high probability, the queue remains in a controlled tube around the bottleneck ray. Second, on this SSC event, MaxWeight reduces the remaining peak analysis to a one-dimensional motion along the bottleneck ray. The tube radius is then paid when translating this one-dimensional bound back into a bound on \(\PeakT\). The reusable conversion from a single-bottleneck collapse bound to a queue-peak bound is proved in Appendix~\ref{app:collapse-to-peak}; the main text states the resulting CRP peak law directly.

Under Assumption~\ref{ass:interior}, define
\begin{equation}\label{eq:nominal-load-def}
\rcirc_t:=\frac{\lambda_t}{1-\eps}\in\PiC .
\end{equation}
as the \emph{nominal load point}. An exposed face is \emph{active} for the load process if it contains the nominal points $\rcirc_t$. CRP says that the same active facet is selected at every time and that the nominal points stay a positive distance from its relative boundary. The normal to this facet is the bottleneck direction.

\begin{assumption}[Nonstationary complete resource pooling (CRP)]
\label{ass:crp}
There is a vector $v\in\Rp^d$ with $\normtwo{v}=1$ and $\mu:=H(v)>0$ such that
\[
\Hface:=\{r\in\R^d:\ip{v}{r}=\mu\},
\qquad
F^\star:=\PiC\cap\Hface
\]
is the unique exposed bottleneck facet.\footnote{Here ``facet'' is used in the relative sense: $F^\star$ is full-dimensional inside its supporting hyperplane, equivalently $\operatorname{aff}(F^\star)=\Hface$, and has nonempty relative interior in $\Hface$.} The nominal points in \eqref{eq:nominal-load-def} satisfy $\rcirc_t\in F^\star$ a.s. for every $t\ge0$. Moreover, there is a deterministic constant $\delta_{\mathrm{face}}>0$ such that, almost surely, for every integer $t\ge0$,
\begin{equation}\label{eq:face-margin-def}
\bigl(B_2(\rcirc_t,\delta_{\mathrm{face}})\cap\Hface\bigr)
\subseteq F^\star .
\end{equation}
That is, the distance from $\rcirc_t$ to the relative boundary of $F^\star$ is uniformly bounded below by $\delta_{\mathrm{face}}$.
\end{assumption}

The slack $\eps$ and the face margin $\delta_{\mathrm{face}}$ have different geometric meanings. The slack measures distance below the bottleneck face, in the normal direction $v$. The face margin measures room to move \emph{within} the active face. That transverse room is what lets MaxWeight compare its chosen schedule with nearby points of $F^\star$ and push the perpendicular component back toward the ray.

Figure~\ref{fig:crp_geometry} illustrates the assumption. In a stationary system, $\rcirc_t\equiv\rcirc$ and CRP is a positive relative-inradius condition around one nominal load point. Here $\rcirc_t$ may move with time and may depend on the past, but it must remain on the same face. The bottleneck direction is fixed even though the input law is not.

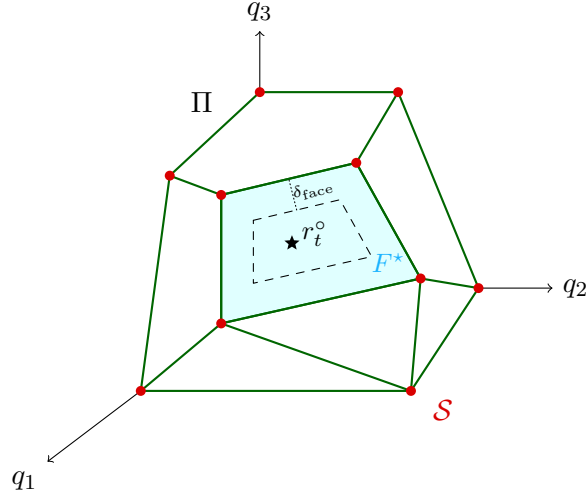
\begin{figure}[htbp]
\centering
\begin{tikzpicture}[scale=0.85, line cap=round, line join=round]

\coordinate (A) at (0.0,0.0);
\coordinate (B) at (4.2,0.0);
\coordinate (C) at (5.25,1.6);
\coordinate (D) at (4.00,4.65);
\coordinate (E) at (1.85,4.65);
\coordinate (F) at (0.45,3.35);
\coordinate (H) at (1.25,3.05); %
\coordinate (G) at (3.35,3.55); %
\coordinate (I) at (1.25,1.05); %
\coordinate (J) at (4.35,1.75); %

\coordinate (H2) at (1.7500,2.6551);
\coordinate (G2) at (3.0970,2.9758);
\coordinate (J2) at (3.5887,2.0907);
\coordinate (I2) at (1.7500,1.6755);
\draw[->, thin] (C) -- ++(1.15,0) node[right] {$q_2$};
\draw[->, thin] (E) -- ++(0,0.95) node[above] {$q_3$};
\draw[->, thin] (A) -- ++(-1.45,-1.10) node[below left] {$q_1$};

\fill[lightblue] (H) -- (G) -- (J) -- (I) -- cycle;
\colorlet{edgec}{green!40!black}
\draw[edgec, thick] (A)--(B)--(C)--(D)--(E)--(F)--(A);
\draw[edgec, thick] (J)--(C);
\draw[edgec, thick] (F)--(H);
\draw[edgec, thick] (I)--(B);
\draw[edgec, thick] (H)--(G);
\draw[edgec, thick] (G)--(D);
\draw[edgec, thick] (G)--(J);
\draw[edgec, thick] (I)--(H);
\draw[edgec, thick] (I)--(J);
\draw[edgec, thick] (I)--(A);
\draw[edgec, thick] (J)--(B);
\draw[edgec, thick] (H)--(G)--(J)--(I)--cycle;
\draw[black, dashed] (H2)--(G2)--(J2)--(I2)--cycle;
\foreach \p in {A,B,C,D,E,F,G,H,I,J} {
  \fill[red!85!black] (\p) circle (2.2pt);
}

\node at (0.95,4.5) {\color{black}$\PiC$};
\node at (3.85,2.0) {\color{cleanblue}\small$F^\star$};
\node at (2.70, 2.45) {\color{black}\small$\rcirc_t$};
\node[rotate=13.5] at (2.68,3.13) {\color{black}\tiny$\delta_{\mathrm{face}}$};
\node at (4.7,-0.3) {\color{red!85!black}$\mathcal S$};

\coordinate (Rcirc) at (2.35,2.3);
\draw[black, densely dotted] ($(2.4235,2.8155)$) -- ($(2.3000,3.3000)$);
\node[
    star,
    star points=5,
    star point ratio=2.25,
    fill=black,
    draw=none,
    minimum size=6pt,
    inner sep=0pt,
    rotate=0
] at (Rcirc){};
\end{tikzpicture}
\caption{Nonstationary CRP geometry. The nominal point $\rcirc_t$ may move
with time and depend on the past, but it always lies in the
$\delta_{\mathrm{face}}$-interior of the same exposed face $F^\star$,
illustrated by the dashed inner polygon. The normal to this face is the
bottleneck direction $v$.}
\label{fig:crp_geometry}
\end{figure}

For the CRP normal $v$, define
\[
Y_t:=\ip{v}{Q_t},
\qquad
Q_t^\parallel:=Y_tv,
\qquad
Q_t^\perp:=Q_t-Q_t^\parallel .
\]
We call $Y_t$ the \emph{bottleneck projection}. Since $v,Q_t\ge0$ and $\normtwo{v}=1$, $Q_t^\parallel$ is the orthogonal projection of $Q_t$ onto the ray spanned by $v$, while $Q_t^\perp$ is the transverse residual.

The first result is the promised tube estimate. It is a finite-time,
nonstationary analogue of the perpendicular-drift argument behind
complete-resource-pooling SSC \citep{Stolyar04,EryilmazSrikant12}.
If \(Q_t^\perp\) is large, the face-margin condition
\eqref{eq:face-margin-def} lets us move a fixed fraction of
\(\delta_{\mathrm{face}}\) from \(\rcirc_t\) in the direction
\(Q_t^\perp/\normtwo{Q_t^\perp}\) while staying on the same active face.
MaxWeight must score at least as well as this feasible comparison point,
and the resulting transverse gain gives negative drift for \(Q_t^\perp\).

\begin{theorem}[Finite-time state-space collapse under CRP]\label{thm:finite_horizon_ssc_crp}
Consider MaxWeight scheduling under Assumptions~\ref{ass:interior}, \ref{ass:bounded-primitives}, and~\ref{ass:crp}. Let $D:=A_{\max}+S_{\max}$. Suppose $\eps\le \delta_{\mathrm{face}}/(2S_{\max})$. Let $v$ be the CRP normal and set $Q_t^\perp:=Q_t-\ip{v}{Q_t}v$. For every $T\ge1$ and $\delta\in(0,1)$, with probability at least $1-\delta$,
\[
\max_{0\le t\le T}\normtwo{Q_t^\perp}
\lesssim
\frac{D^2}{\delta_{\mathrm{face}}}\log\frac{e(T+1)}{\delta}.
\]
Equivalently, the network satisfies finite-horizon ray collapse with radius
\begin{equation}\label{eq:ssc-Rperp-def}
R_\perp(T,\delta)
:=
C\,\frac{D^2}{\delta_{\mathrm{face}}}\log\frac{e(T+1)}{\delta}
\end{equation}
for a universal constant $C$.
\end{theorem}

The restriction $\eps\le\delta_{\mathrm{face}}/(2S_{\max})$ is automatic in heavy traffic for a fixed CRP geometry. It appears only because the statement is nonasymptotic: the transverse comparison point must remain inside the face after the load is scaled by $1-\eps$. Appendix~\ref{app:ssc} gives the proof.

On the collapse event, the network is essentially one-dimensional:
\[
Q_t=Y_tv+Q_t^\perp,
\qquad
\normtwo{Q_t^\perp}\le R_\perp(T,\delta).
\]
It remains to identify the entrance level above which the scalar projection has clean negative drift. Two local obstructions must be cleared. First, the transverse error can tilt a MaxWeight comparison between a $v$-optimal schedule and a nearly optimal one; the schedule gap $\Delta_{\mathrm{gap}}$ determines how large $Y_t$ must be before this tilt is irrelevant. Second, even a $v$-optimal schedule contributes its full bottleneck service only when the queues in the positive coordinates of $v$ are not empty; this is the no-waste threshold governed by $v_{\min}$.

Under CRP, the scalar arrival rate in the bottleneck direction is exactly
\[
\ip{v}{\lambda_t}=(1-\eps)\mu
\qquad\text{a.s. for every }t,
\]
because $\lambda_t=(1-\eps)\rcirc_t$ and $\rcirc_t\in F^\star$. Thus, once the two local obstructions are absent, the projection $Y_t$ has drift $-\eps\mu$ and can be controlled by the same one-dimensional self-normalized peak argument used earlier.

\begin{theorem}[Local peak law under CRP]\label{thm:crp-main}
Consider MaxWeight scheduling under Assumptions~\ref{ass:interior}, \ref{ass:bounded-primitives}, and~\ref{ass:crp}. Let $D:=A_{\max}+S_{\max}$ and suppose $\eps\le\delta_{\mathrm{face}}/(2S_{\max})$. Let $v$ be the CRP normal, set $\mu:=H(v)>0$, and let $R_\perp$ be the CRP collapse radius in \eqref{eq:ssc-Rperp-def}. Define
\[
\Delta_{\mathrm{gap}}
:=
\min\{\mu-\ip{v}{s}:s\in\Sset,\ \ip{v}{s}<\mu\},
\]
with the convention $\Delta_{\mathrm{gap}}=\infty$ if every schedule is $v$-optimal, and interpret $a/\infty:=0$ for finite $a$. Let
\[
v_{\min}:=\min\{v_i:v_i>0\}.
\]
For $T\ge1$ and $\delta\in(0,1)$, write
\[
R_\delta:=R_\perp(T,\delta/2)
\]
and set
\[
x_{\mathrm{loc}}(T,\delta)
:=
\max\left\{
\frac{4S_{\max}R_\delta}{\Delta_{\mathrm{gap}}},
\frac{S_{\max}+R_\delta}{v_{\min}}
\right\}.
\]
Then, with probability at least $1-\delta$, both bounds hold:
\[
\PeakT
\;\lesssim\;
\frac{D}{\eps}\log\frac{T+1}{\delta}
+
x_{\mathrm{loc}}(T,\delta)
+
R_\delta,
\]
and
\[
\max_{0\le t\le T}\normone{Q_t}
\;\lesssim\;
\frac{\normone{v}\,D}{\eps}\log\frac{T+1}{\delta}
+
\normone{v}\,x_{\mathrm{loc}}(T,\delta)
+
\sqrt d\,R_\delta.
\]
If service is deterministic, $S_t=s_t$ a.s., then $D$ can be replaced by $A_{\max}$ in the displayed logarithmic coefficients. The collapse radius remains the one supplied by \eqref{eq:ssc-Rperp-def}.
\end{theorem}

\paragraph{Proof sketch.}
Apply Theorem~\ref{thm:finite_horizon_ssc_crp} with error probability $\delta/2$ and stop the process when it leaves the resulting tube. Inside the tube, the first term in $x_{\mathrm{loc}}$ makes the bottleneck component dominate the transverse error in the MaxWeight comparison, so every selected schedule is $v$-optimal. The second term ensures that service in the positive coordinates of $v$ is actually used. Therefore the stopped projection has drift $-\eps\mu$ above $x_{\mathrm{loc}}$ and upward jumps bounded by $D$. Appendix~\ref{app:scalar} gives the scalar peak bound; the identities $Q_t=Y_tv+Q_t^\perp$ and $\normtwo{Q_t^\perp}\le R_\delta$ convert it back to Euclidean peak and total backlog. Appendix~\ref{app:collapse-to-peak} proves this conversion in reusable form, and Appendix~\ref{app:locallaw} applies it here.

\vspace{1em}

Theorem~\ref{thm:crp-main} is the local counterpart of the global peak theorem. The global argument waits until the queue is large enough to certify drift uniformly over the whole nonnegative orthant; the price is the global margin $\alphapi$. CRP first rules out most of those directions. The entrance threshold is then governed by local quantities: the schedule gap $\Delta_{\mathrm{gap}}$, the face margin through $R_\perp$, and the no-waste scale $v_{\min}^{-1}$. None of these carries the global $1/\alphapi$ cost. After localization, the only remaining growth is scalar, with logarithmic coefficient $D/\eps$, or $A_{\max}/\eps$ under deterministic service.

The theorem still contains the finite-horizon cost of proving localization. The generic SSC estimate above gives a logarithmic radius $R_\perp(T,\delta)$, and that logarithm appears in the final bound. Thus the result should be read as a principled replacement of global geometry by local bottleneck geometry, rather than an identification of the optimal local geometric threshold independent of \(T\). The two-queue bound in Proposition~\ref{prop:twoqueue-peak} illustrates what is possible when a sharper, model-specific collapse estimate yields a tube radius that is independent of the horizon.

\section{Applications to input-queued switches}\label{sec:iqs}

We now apply the finite-time framework to input-queued switches \citep{ShahTsitsiklisZhong11}. This is the canonical matching network behind packet switches, crossbar routers, and related data-center and internet-routing architectures. The model is simple enough for the geometry to be explicit and rich enough to contain one of the central queue-size scaling problems in stochastic networks.

\subsection{Summary of results for the three IQS models}

An $n\times n$ input-queued switch has $d=n^2$ queues indexed by input-output pairs $(i,j)$. A schedule $s\in\Sset_{\mathrm{IQS}}$ is a full permutation matrix. Thus every schedule serves exactly one queue from each row and one queue from each column, and service is deterministic: $S_t=s_t$. With the downward-closed convention \eqref{eq:capacity-downward}, the capacity region is the substochastic bipartite matching polytope,
\[
\PiC_{\rm IQS}
=
\left\{x\in\Rp^{n\times n}:\sum_jx_{ij}\le1\ \forall i,
\quad
\sum_ix_{ij}\le1\ \forall j\right\}.
\]
For nonnegative queue weights, maximizing over full permutation matrices has the same support function as maximizing over this downward-closed polytope, since any partial matching can be completed to a full permutation without decreasing a nonnegative weight. We identify matrices with their vectorizations and use the Frobenius norm as the Euclidean norm.

We consider three IQS settings. We call the first setting the \emph{generalized IQS} model (to be defined shortly in the next subsection): arrivals may be adapted, dependent, and fractional, subject to a pathwise Frobenius envelope and conditional row/column slack. The second is the same generalized model under CRP, where a single row or column is the active bottleneck. Collapse then exposes a smaller logarithmic coefficient for the bottleneck projection, while the total-backlog bound also contains the finite-time collapse-radius contribution. The third is a Bernoulli IQS model with predictable, possibly
time-varying arrival probabilities.  This contains the classical
time-homogeneous Bernoulli IQS as a special case, but is nonstationary in
nature.  It is not a
pathwise submodel of the generalized IQS, but its entrywise independence  
structure gives stronger variance-sensitive concentration.

The upper bound results of three IQS models are summarized in Table~\ref{tab:IQS three regimes}. The results are shown for the \emph{total backlog} peak $\max_{0\le t\le T}\normone{Q_t}$ rather than $\PeakT$ due to convention in IQS literature.
\begin{table}[t]
    \centering
    \begin{tabular}{ccc}
\toprule
Setting & Logarithmic coefficient & Geometric threshold \\
\midrule
Generalized IQS (Theorem~\ref{thm:iqs-upper-main}) &$n^{3/2}/\eps$ & $n^{5/2}/\eps$\\
Generalized IQS under CRP (Corollary~\ref{cor:iqs-crp})& $n/\eps+n^2/\delta_{\rm face}$ & $n/\eps+n^2/\delta_{\rm face}+n^{3/2}$ \\
Bernoulli IQS (Theorem~\ref{thm:bernoulli-iqs-upper})& $n/\eps$ & $n^{5/2}/\eps$ \\
\bottomrule
    \end{tabular}
    \caption[Upper bounds for the three IQS regimes]{Upper bounds of $\max_{0\le t\le T}\normone{Q_t}$ for three IQS regimes. The CRP row reports the bound obtained by combining the local theorem with the generic finite-time SSC radius: $n/\eps$ is the bottleneck-projection coefficient, while $n^2/\delta_{\rm face}$ comes from the collapse radius.}
    \label{tab:IQS three regimes}
\end{table}

We note that the classical independent Bernoulli IQS model is usually studied in steady state, where the optimal queue-size scaling problem remains open~\citep{ShahTsitsiklisZhong11,ShahTsitsiklisZhong14,ShahWaltonZhong14,MaguluriSrikant16,
MaguluriBurleSrikant18,XuZhong20}. Our bounds address finite-time peaks, whose geometric thresholds can be larger than the steady-state scales, and we do not intend to improve existing steady-state bounds. The contributions here are different: we identify the \emph{tight} coefficient multiplying $\log T$ under robust nonstationary input models and show how geometry and variance change that coefficient.

Before explaining the detailed guarantees of each IQS model, we state some elementary properties shared by the three IQS models. The proof is deferred to Appendix~\ref{app:iqs-general-upper}.

\begin{proposition}[Euclidean geometry of IQS models]\label{prop:iqs-alpha2}
For the $n\times n$ input-queued switch,
\[
S_{\max}=\sqrt n,
\qquad
\alphapi=\frac1{\sqrt n},
\qquad
\normone{Q}\le n\normtwo{Q}
\quad\text{for all }Q\in\Rp^{n\times n}.
\]
\end{proposition}

\subsection{Generalized IQS: achieving the tight logarithmic coefficient}

The generalized IQS model is defined by the following assumption.

\begin{assumption}[Generalized IQS arrival envelope]\label{ass:gIQS}
The arrival matrix $A_t\in\Rp^{n\times n}$ is $\F_{t+1}$-measurable and satisfies
\[
\normtwo{A_t}=\norm{A_t}_F\le\sqrt n
\qquad\text{a.s. for every }t.
\]
Its conditional mean $\lambda_t:=\E[A_t\mid\F_t]$ satisfies
\begin{equation}\label{eq:iqs-slack}
\sum_{j=1}^n \lambda_{ij}(t)\le1-\eps,
\qquad
\sum_{i=1}^n \lambda_{ij}(t)\le1-\eps
\qquad\text{a.s. for all }i,j,t.
\end{equation}
\end{assumption}

The Frobenius envelope is normalized to the natural service scale: every
permutation schedule has Frobenius norm \(\sqrt n\). The row and column
inequalities in \eqref{eq:iqs-slack} are exactly the condition
\(\lambda_t\in(1-\eps)\PiC_{\rm IQS}\). Thus Assumption~\ref{ass:gIQS}
preserves the usual IQS capacity region and slack condition while discarding
the classical independence structure: arrivals may be correlated across queues
and time, and may even be chosen adaptively, subject only to the aggregate
slack condition and the pathwise Frobenius envelope.\footnote{The classical
independent Bernoulli IQS is not literally a submodel of
Assumption~\ref{ass:gIQS} with the pathwise \(\sqrt n\) envelope, since one slot
may contain more than \(n\) arrivals. Nevertheless, generalized IQS captures a
natural deterministic-envelope analogue of the classical model. Let
\(N_t=\|A_t\|_F^2\). In the independent Bernoulli model,
\(\mathbb E[N_t\mid\mathcal F_t]\le n\). Bernstein's inequality and a union
bound give, with \(L=\log(T/\delta)\),
\[
\mathbb P\left(
\max_{0\le t<T}N_t>
n+\sqrt{2nL}+\frac{2}{3}L
\right)\le\delta .
\]
When \(n\gtrsim L\), this recovers the generalized-IQS envelope in
Assumption~\ref{ass:gIQS}, up to universal constants. For smaller \(n\), any
pathwise envelope of this form must depend on the horizon. The
queue-Bernstein theorem (Theorem~\ref{thm:qb-unbounded-main}) later treats
independent Bernoulli IQS directly, without this localization, and yields
Theorem~\ref{thm:bernoulli-iqs-upper}.}

Our global theorem, stated below, immediately gives a finite-time total-backlog peak bound. The important term is the logarithmic coefficient; the threshold has higher dependence on $n$, which is why this estimate is not a steady-state-mean improvement.

\begin{theorem}[Generalized IQS upper bound]\label{thm:iqs-upper-main}
Consider the $n\times n$ input-queued switch with full-permutation deterministic service $S_t=s_t$ and  MaxWeight scheduling. Suppose the arrivals satisfy Assumption~\ref{ass:gIQS} for some $\eps\in(0,1)$. Then, for every $T\ge1$ and $\delta\in(0,1)$, with probability at least $1-\delta$, we have
\[
\max_{0\le t\le T}\normone{Q_t}
\;\lesssim\;
\frac{n^{3/2}}{\eps}\log\frac{T+1}{\delta}
+
\frac{n^{5/2}}{\eps}.
\]
\end{theorem}

This is a corollary of Theorem~\ref{thm:selfnorm-main} (deterministic service part) plus Proposition~\ref{prop:iqs-alpha2}; Appendix~\ref{app:iqs-general-upper} records the calculation. We establish a matching lower bound (in terms of the logarithmic coefficient) by constructing a hard instance satisfying the all-ports-loaded slack form, meaning that every row and every column of the mean arrival matrix has sum $1-\eps$.

\begin{theorem}[Generalized IQS lower bound]\label{thm:iqs-lower-main}
There exist universal constants $c,C>0$ such that, for every $n\ge4$, $\eps\in(0,1/16]$, there is a generalized $n\times n$ IQS arrival instance satisfying the all-ports-loaded slack form with the following property. Under any nonanticipative full-permutation deterministic-service scheduling rule, in particular under MaxWeight, for every $T\ge C\sqrt n/\eps^2$,
\begin{equation}\label{inequ:IQSlower1}
 \E\Big[\max_{0\le t\le T}\normone{Q_t}\Big]
\;\ge\;
c\,\frac{n^{3/2}}{\eps}
\log\!\Big(1+\frac{\eps^2}{\sqrt n}T\Big).
\end{equation}
\end{theorem}

The upper and lower bounds match in the coefficient of the $\log T$ term. This is the finite-time message of the generalized IQS: under adapted, dependent, fractional arrivals, the $n^{3/2}\eps^{-1}\log T$ contribution to the total-backlog peak is fundamental. The lower bound does not rule out better $n$ scaling in the classical independent Bernoulli model, because the generalized IQS construction uses dependent and fractional arrivals that are excluded there.

\subsection{Generalized IQS under CRP: local refinement via state-space collapse}

The global generalized IQS bound treats the switch as an $n^2$-dimensional object. In a CRP regime with a single tight row or column, the switch is approximately described by the corresponding bottleneck projection plus a transverse collapse error. The bottleneck vector is the normalized indicator of the tight row or column, and hence $\normone{v}=\sqrt n$. Feeding this local direction into the deterministic-service form of Theorem~\ref{thm:crp-main} sharpens the projected-coordinate logarithmic coefficient from $n^{3/2}/\eps$ to $n/\eps$. If the CRP-based SSC estimate of Theorem~\ref{thm:finite_horizon_ssc_crp} is used to supply the collapse radius, the total-backlog upper bound also includes the collapse logarithmic coefficient of order $n^2/\delta_{\rm face}$. See Corollary~\ref{cor:iqs-crp} for the formal guarantee.

\begin{corollary}[Generalized IQS under CRP: upper bound]\label{cor:iqs-crp}
Consider the generalized $n\times n$ IQS with full-permutation deterministic service, MaxWeight scheduling, and nonstationary bounded arrivals satisfying Assumption~\ref{ass:gIQS}. Write $\rcirc_t:=\lambda_t/(1-\eps)$. Suppose Assumption~\ref{ass:crp} holds and its bottleneck face is one row or one column. Assume $\eps\le\delta_{\mathrm{face}}/(2S_{\max})$ and let
\[
R_{\perp,\mathrm{IQS}}(T,\delta)
:=
\frac{C_0 n}{\delta_{\mathrm{face}}}
\log\frac{e(T+1)}{\delta}
\]
for a universal constant $C_0$. Then, with probability at least $1-\delta$,
\[
\max_{0\le t\le T}\normone{Q_t}
\;\lesssim\;
\frac{n}{\eps}\log\frac{T+1}{\delta}
+
n^{3/2}+n\,R_{\perp,\mathrm{IQS}}(T,\delta/2).
\]
\end{corollary}

Appendix~\ref{app:iqs-local-upper} records the calculation. The corollary is not a replacement for the global theorem; it is the payoff when collapse information is available. It separates the bottleneck-projection coefficient, $n/\eps$, from the additional finite-time price of proving collapse. With the generic SSC radius displayed above, this additional logarithmic coefficient is of order $n^2/\delta_{\mathrm{face}}$.

We also establish a lower bound, stated below.

\begin{theorem}[Generalized IQS under CRP: lower bound]\label{thm:iqs-local-lower-main}
There exist universal constants $c',C'>0$ such that, for every $n\ge4$, $\eps\in(0,1/16]$, there is a generalized $n\times n$ IQS arrival instance satisfying Assumption~\ref{ass:crp} with the following property. Under any nonanticipative full-permutation deterministic-service scheduling rule, in particular under MaxWeight, started from $Q_0=0$, for every $T\ge C'n/\eps^2$,
\begin{equation}\label{inequ:IQSlower2}
\E\Big[\max_{0\le t\le T}\normone{Q_t}\Big]
\;\ge\;
c'\,\frac{n}{\eps}
\log\!\Big(1+\frac{\eps^2}{n}T\Big).
\end{equation}
\end{theorem}

Thus the SSC refinement has the right bottleneck-projection logarithmic coefficient in the generalized model. With the generic finite-time SSC estimate inserted, the total-backlog coefficient is the sum of this projected-coordinate coefficient and the collapse coefficient shown in Table~\ref{tab:IQS three regimes}. 

\subsection{Bernoulli IQS: variance-sensitive improvement}

The global generalized-IQS theorem pays for the crude Frobenius scale
\(A_{\max}=\sqrt n\) in Assumption~\ref{ass:gIQS}.  Independent Bernoulli input
has more structure, even when its rates are nonstationary.  Throughout this
subsection, ``Bernoulli IQS'' means that, conditional on the past, the entries
of \(A_t\) are independent Bernoulli variables with predictable means
\(\lambda_{ij}(t)\) satisfying the row and column slack inequalities
\eqref{eq:iqs-slack}.  The classical time-homogeneous Bernoulli IQS corresponds to
the special case in which these means are constant in \(t\).

In every nonnegative queue direction, the projected arrival
\(\ip{u_t}{A_t}\) is queue-Bernstein with a dimension-free
variance-to-mean constant.  Applying the variance-sensitive queue-Bernstein
theorem of Section~\ref{sec:qb-unbounded}, we find that the logarithmic
coefficient for \(\PeakT\) is \(1/\eps\), corresponding to \(n/\eps\) for
\(\max_{0\le t\le T}\normone{Q_t}\), rather than \(\sqrt n/\eps\),
corresponding to \(n^{3/2}/\eps\) for
\(\max_{0\le t\le T}\normone{Q_t}\).

\begin{theorem}[Bernoulli IQS upper bound]\label{thm:bernoulli-iqs-upper}
Consider the
$n\times n$ input-queued switch with full-permutation deterministic service and
MaxWeight scheduling. Let $\eps\in(0,1)$. For every $t$, assume
that, conditional on $\F_t$, the entries of $A_t$ are independent Bernoulli
random variables with predictable conditional means $\lambda_{ij}(t)\in[0,1]$:
\[
\Pbb\bigl(A_{ij}(t)=1\mid\F_t\bigr)=\lambda_{ij}(t),
\qquad 1\le i,j\le n,
\]
and assume that the predictable mean matrix $\lambda_t$ satisfies the row and
column slack inequalities \eqref{eq:iqs-slack}. Then, for every $T\ge1$ and
$\delta\in(0,1)$, with probability at least $1-\delta$,
\[
\max_{0\le t\le T}\normtwo{Q_t}
\lesssim
\frac{1}{\eps}\log\frac{T+1}{\delta}
+
\frac{n^{3/2}}{\eps}.
\]
Consequently, with the same probability,
\[
\max_{0\le t\le T}\normone{Q_t}
\lesssim
\frac{n}{\eps}\log\frac{T+1}{\delta}
+
\frac{n^{5/2}}{\eps}.
\]
\end{theorem}
From Table~\ref{tab:IQS three regimes}, we can see that compared to generalized IQS, the independent Bernoulli structure sharpens the logarithmic term.

\section{Experiments}\label{sec:experiments}

The preceding results establish three quantitative, testable finite-time mechanisms: a two-phase peak envelope transitioning from $\sqrt{T}$ to $\log T$, sharper geometric thresholds under state-space collapse, and variance-sensitive peak scaling under the queue-Bernstein condition. The experiments below confirm each of them.

In all experiments we simulate MaxWeight scheduling, start from $Q_0=0$, and record the running peak $\max_{0\le s\le t}\norm{Q_s}$, using total backlog or Euclidean norm according to the theorem being illustrated. The main control parameters were the horizon $T$, slack $\eps$, network size, and bottleneck geometry. Curves report empirical means over independent replications; shaded bands are $\pm$ one standard error. Experiment~1 averages over 500 runs, while Experiments~2 and~3 average over 200 independent runs. All experiments were run on standard CPU resources. No GPU or accelerator was used. The reported experiments require at most a few CPU-hours in total.

Appendix~\ref{app:experiments} contains the full experimental details, additional diagnostics, and instructions for accessing the GitHub replication package.

\paragraph{Experiment 1: the two-phase law.} The first experiment tests the two regimes in Corollary~\ref{cor:two-phase-main}: $\sqrt{T}$ burn-in followed by $\log T$ growth after a geometric entrance threshold. A direct consequence is that the transition becomes visible at shorter horizons when the slack $\varepsilon$ is larger. 

We examine two networks: the first network is an $n\times n$ input-queued switch with $n=20$, independent Bernoulli arrivals $A_{ij}(t)\sim\mathrm{Bernoulli}((1-\eps)/n)$ and deterministic full-permutation service, under MaxWeight scheduling. We plot the Frobenius-norm peak; the second is a discrete-time parallel-server system with $L=K=40$ customer classes and servers, connected by a ring compatibility graph: each class $\ell$ can be served by servers $\ell$ and $(\ell{+}1)\bmod L$. Arrivals per class are Bernoulli $A_\ell(t) \sim \operatorname{Bernoulli}((1-\varepsilon) \mu K / L)$, and each server has geometric service times with parameter $\mu=0.05$. Idle servers are assigned to compatible waiting classes by MaxWeight. We plot the Euclidean-norm peak. In both networks, we sweep $\varepsilon \in\{0.005,0.01,0.02,0.05\}$.
 
The top curves plot the empirical mean running peak $\mathbb{E}\left[\mathrm{Peak}_T\right]$, and the bottom curves plot the local $\log$-log slope $\beta(T):=\frac{\mathrm{d}}{\mathrm{d} \log T} \log \mathbb{E}\left[\mathrm{Peak}_T\right]$. A pure $T^x$ growth has $\beta(T)=x$, while a logarithmic growth has $\beta(T) \asymp 1 / \log T$. In both the IQS and parallel-server examples, larger slack makes $\beta(T)$ drop away from the $\sqrt{T}$ reference line much earlier, illustrating the transition from the burn-in envelope to the logarithmic-growth regime.

\begin{figure}[h]
    \centering
    \begin{subfigure}[t]{0.4\linewidth}
        \centering
        \safeincludegraphics[width=\linewidth]{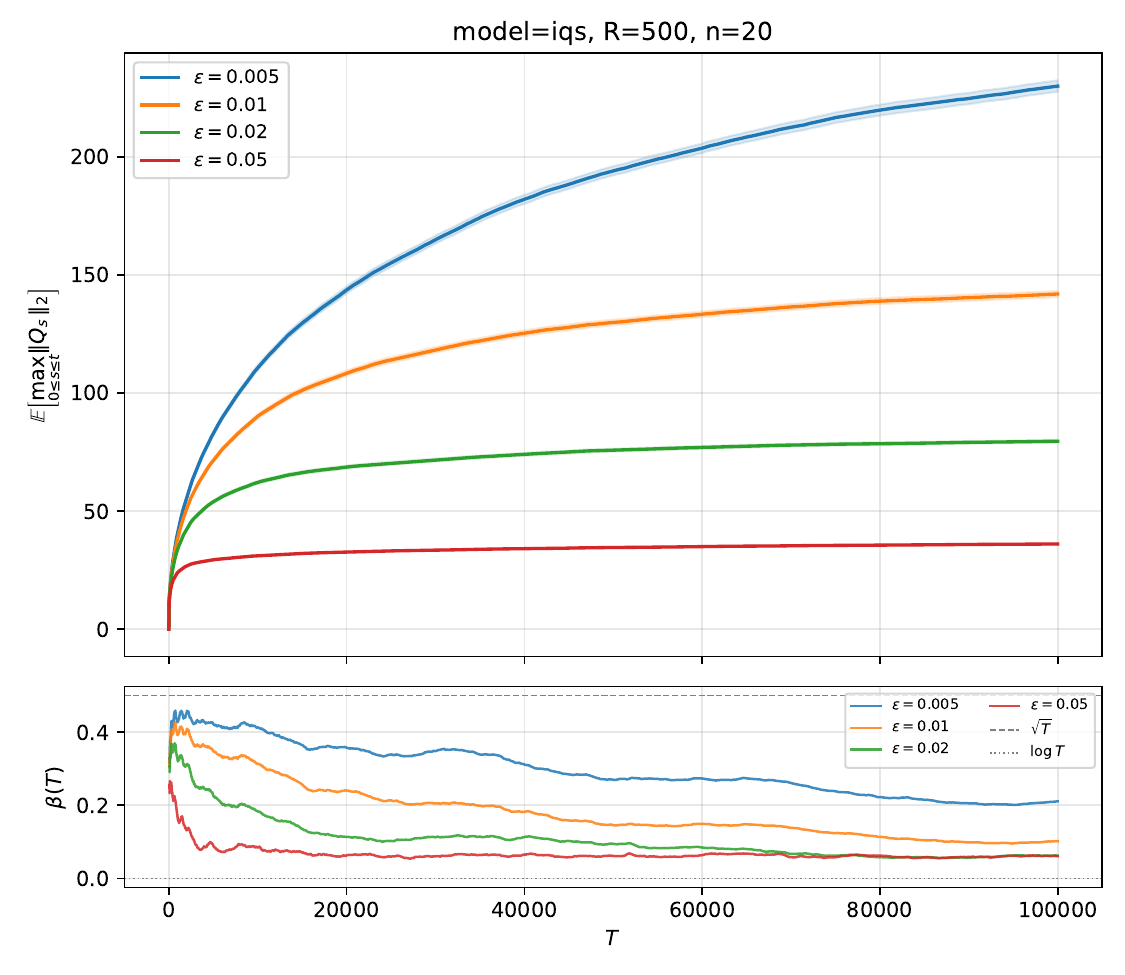}
        \caption{{Input-queued switch, $n=20$.}}
        \label{fig:iqs-two-phase}
    \end{subfigure}
    \hspace{3em}
    \begin{subfigure}[t]{0.4\linewidth}
        \centering
        \safeincludegraphics[width=\linewidth]{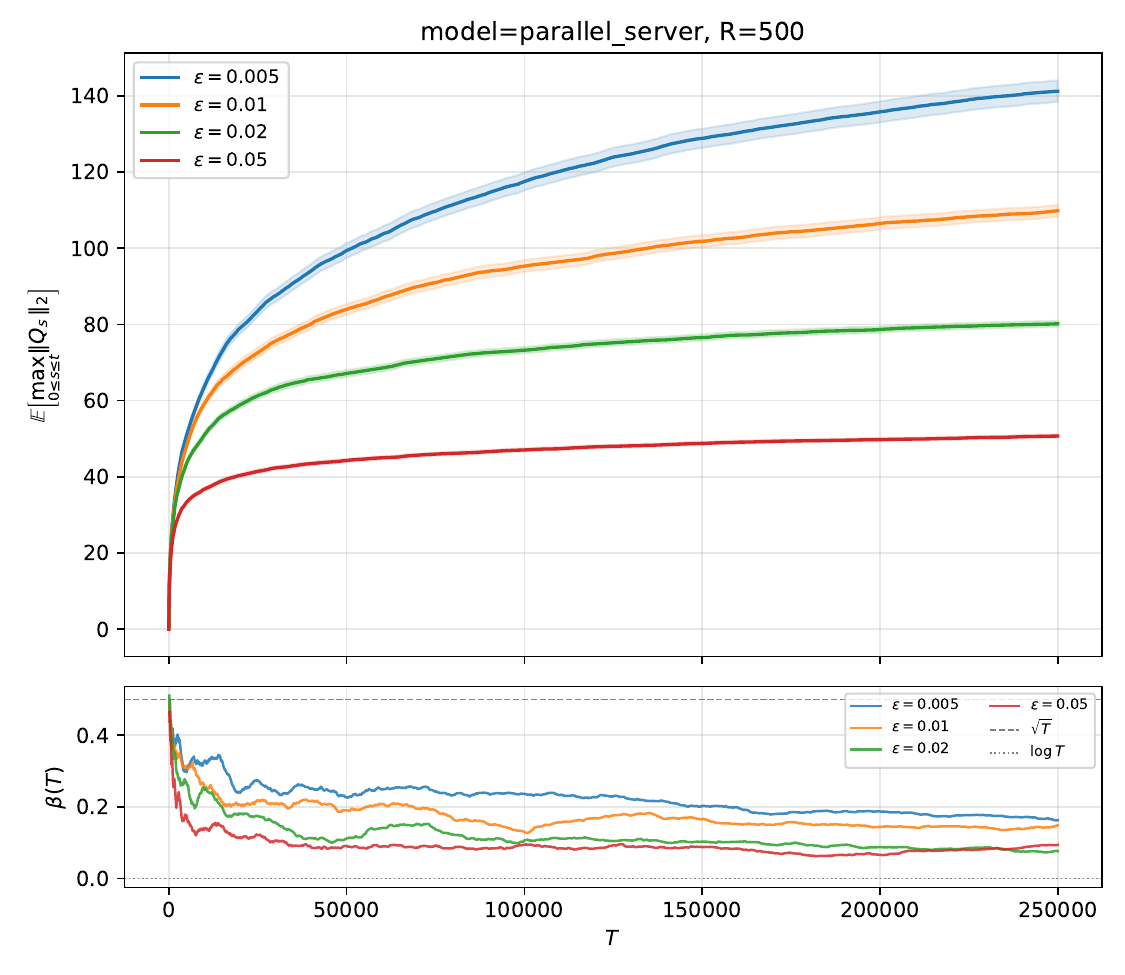}
        \caption{{Parallel-server, 40 servers \& 40 classes}.}
        \label{fig:ps-two-phase}
    \end{subfigure}
    \caption{{Two-phase behavior of the mean running peak. Top: mean peak with $\pm$SE bands. Bottom: local log-log slope $\beta(T)$.}}
    \label{fig:two-phase}
\end{figure}

\paragraph{Experiment 2: local versus global geometry.} 
This experiment uses the two-queue system of Section~\ref{sec:geometry} with service set $\Sset=\{(1,0),(0,a)\}$ and global margin $\alphapi\asymp a$, under MaxWeight scheduling.
Figure~\ref{fig:twoqueue-local-global} varies the anisotropy parameter $a\in\{0.5,0.25,0.1\}$ and compares three curves: the simulated peak, the global prediction from Theorem~\ref{thm:selfnorm-main} (which pays the worst-direction margin $\alphapi$), and the local prediction from Proposition~\ref{prop:twoqueue-peak} (which exploits collapse). The simulations track the local prediction closely, and the gap with the global prediction grows as $a$ decreases. This is the finite-time signature of state-space collapse: the process pays for the geometry it actually visits, not for the worst direction of the capacity region.

\begin{figure}
    \centering
    \safeincludegraphics[width=1\linewidth]{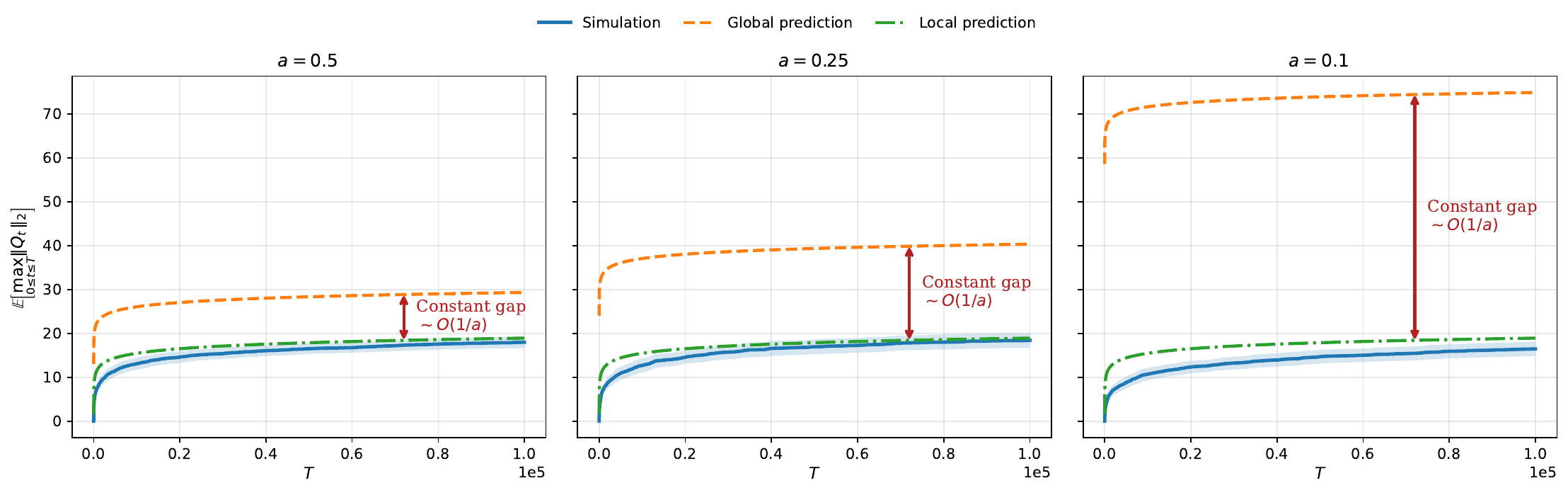}
    \caption{Two-queue example. Global prediction uses Theorem~\ref{thm:selfnorm-main}; local prediction uses Proposition~\ref{prop:twoqueue-peak}.}
    \label{fig:twoqueue-local-global}
\end{figure}

\paragraph{Experiment 3: CRP and variance.}
The third experiment separates two effects that are invisible from Theorem \ref{thm:selfnorm-main} alone: bottleneck geometry and arrival variability. Both comparisons use a generalized $n\times n$ IQS with $n=20$, $\eps=0.01$. Service is deterministic and schedules are chosen by MaxWeight.
 
\emph{CRP versus non-CRP geometry} (Figure~\ref{fig:crp-variance-comparison}(a)). We compare two instances whose arrivals equal a fixed matrix $M$ with probability $1-\eps$ and zero otherwise, 
\[
    A_t
    :=
    \begin{cases}
    0, & \text{with probability } \eps,\\
    M, & \text{with probability } 1-\eps,
    \end{cases}
\]
so the mean load is $\lambda = (1-\eps)M$ in both cases.
The non-CRP instance loads the first row \emph{and} the first column at level $1/n$, with all other entries at $1/n^2$:
\[
M^{\mathrm{nonCRP}}_{ij}
:=
\begin{cases}
1/n, & i=1 \text{ or } j=1,\\
1/n^2, & i\ne1 \text{ and } j\ne1.
\end{cases}
\]
This creates two nearly active bottleneck constraints, since both the first row and the first
column carry elevated traffic.
In the CRP instance, the elevated traffic is concentrated on the first
row and on the diagonal:
\[
M^{\mathrm{CRP}}_{ij}
:=
\begin{cases}
1/n, & i=1 \text{ or } i=j,\\
1/n^2, & \text{otherwise}.
\end{cases}
\]
Here the first row is the sole dominant bottleneck. For each instance we simulate the running total-backlog peak \(\max_{s\le t}\|Q_s\|_1\) over independent replications. Figure~\ref{fig:crp-variance-comparison}(a) shows that the CRP instance has a visibly smaller peak, consistent with the local refinement: collapse near a single bottleneck face reduces the projected-coordinate contribution from the global \(n^{3/2}/\eps\) scale toward the local \(n/\eps\) scale, while a finite-time proof must still account for the collapse radius itself.
 
\emph{Independent versus synchronous arrivals} (Figure~\ref{fig:crp-variance-comparison}(b)). We compare two arrival processes with the same entrywise means but different dependence across entries. In the independent model, $A_{ij}(t)\sim\mathrm{Bernoulli}((1-\eps)/n)$ independently across entries and slots. In the synchronous model, arrivals remain independent across slots but are dependent across entries: for every slot, with probability $(1-\eps)/n$ every entry receives one arrival simultaneously, and otherwise every entry is zero. Thus the two models have the same entrywise mean load, but very different dependence and total-mass variance. The synchronous model produces larger peaks, illustrating that finite-time peaks are controlled not only by mean arrivals but also by the directional variance structure captured by the queue-Bernstein condition.

\begin{figure}[h]
    \centering
    \begin{subfigure}[b]{0.49\textwidth}
        \centering
        \safeincludegraphics[width=\textwidth]{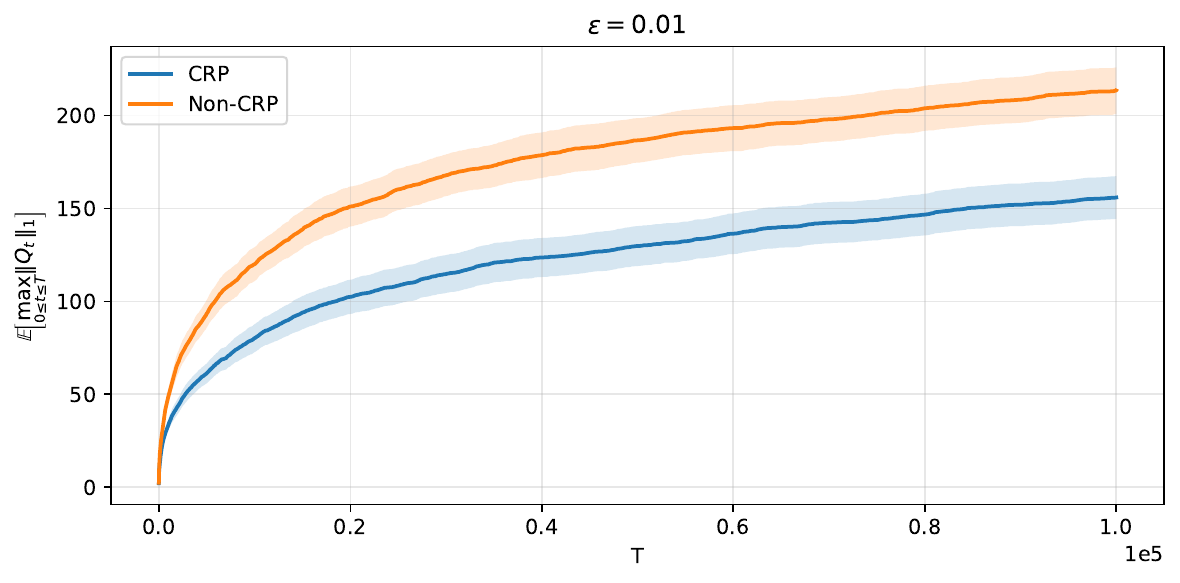}
        \caption{CRP vs. non-CRP geometry.}
        \label{fig:crp-vs-noncrp}
    \end{subfigure}
    \begin{subfigure}[b]{0.49\textwidth}
        \centering
        \safeincludegraphics[width=\textwidth]{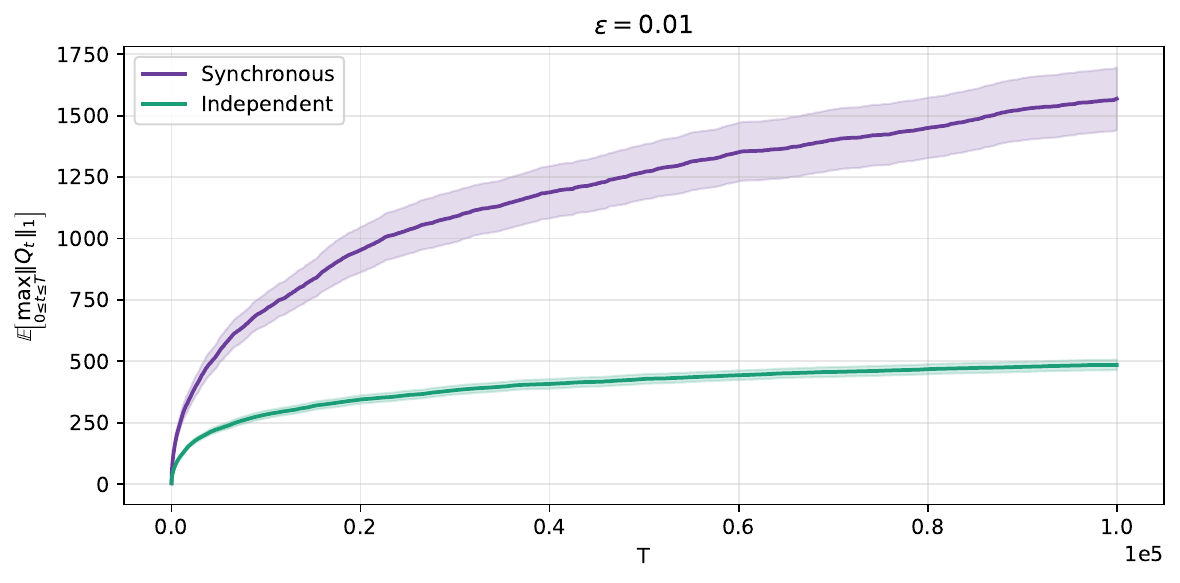}
        \caption{Independent vs. synchronous arrivals.}
        \label{fig:independent-vs-synchronous}
    \end{subfigure}
    \caption{Geometry and variance effects in generalized IQS.}
    \label{fig:crp-variance-comparison}
\end{figure}

\section{Discussion and open directions}\label{sec:discussion}

This paper identifies the finite-time outer shape of queue peaks. Before the queue has had time to exploit the network's global or local geometry, the natural envelope is the minimax $\sqrt{T}$ law. After the queue has paid the geometric entrance fee, additional growth is only logarithmic. The logarithmic coefficient is  geometry-free, while the entrance scale is geometric. This separation, together with its origin and possible refinements, is the main conceptual outcome of the paper and points to many future research directions. 

We discuss two open directions below.

\paragraph{Beyond CRP-based state-space collapse.}

Theorem~\ref{thm:geometry-lower-main} shows that the global margin $\alphapi$ is necessary for fully uniform peak laws. Section~\ref{sec:geometry} shows why it can nevertheless be overly pessimistic for localized dynamics: once the queue is known to stay near its active bottleneck set, the relevant margin is determined by that local geometry. Our local theorem treats the cleanest case, quantitative CRP, where a unique exposed bottleneck facet yields the classical single-bottleneck collapse picture \cite{Reiman84SSC,Stolyar04,EryilmazSrikant12,Williams16}.

The remaining challenge is to move beyond CRP. In non-CRP networks, several bottleneck constraints may bind at once, so the local state space is no longer a ray, but a cone, or more generally a stratified family of bottleneck faces. The natural extension would attach to each such face a collapse radius and a local service margin, extending Proposition~\ref{prop:collapse-to-peak}. This would turn the richer heavy-traffic geometry known for switches with multiple saturated resources \cite{MaguluriSrikant16,MaguluriBurleSrikant18,DaiHarrison20} into nonasymptotic peak laws.

\paragraph{Beyond uniform slack.}

A second limitation is the uniform slack condition in Assumption~\ref{ass:interior}. Weakening this assumption is delicate. A time-averaged slack condition
\[
\frac{1}{T}\sum_{t=0}^{T-1}\lambda_t\in(1-\eps)\Pi
\]
by itself is too weak: one can construct environments whose average interior margin over the horizon is of order $\eps$, yet whose queue length still grows linearly in $T$, simply because the slack arrives too late to undo earlier backlog accumulation. Any useful replacement for uniform slack must therefore control not only how much slack exists on average, but also how it is distributed in time relative to the queue dynamics.

Existing nonstationary queueing frameworks, such as Markov-modulated service or channel-fading models \cite{Neely10}, suggest one possible route. These models impose enough temporal structure to recover meaningful long-run guarantees without full stationarity. Developing a finite-time peak law under weaker, temporally structured slack conditions remains an appealing direction for future work.

\bibliographystyle{plainnat}
\bibliography{ref}

\clearpage
\appendix
\renewcommand{\contentsname}{Appendix Contents}
\addtocontents{toc}{\protect\setcounter{tocdepth}{2}}
\tableofcontents

\clearpage

\section{Experimental details}\label{app:experiments}

This appendix gives the simulation details for the experiments in
Section~\ref{sec:experiments}.

\subsection{Experiment 1: two-phase behavior}

This experiment visualizes the two regimes in
Corollary~\ref{cor:two-phase-main}. The qualitative prediction is that smaller slack delays the transition, while
larger slack moves the logarithmic regime into smaller, practically visible
horizons.

\paragraph{Local log-log slope.}
For the two-phase experiment we also plot the local log-log slope \[\beta(T):=\frac{d\log \PeakT}{d\log T},\]
estimated from the downsampled log-log curve by finite differences.  This is a
diagnostic for the local growth exponent.  If \( \PeakT \asymp T^x\), then
\(\beta(T)=x\).  If \( \PeakT \asymp \log T\), then $\beta(T)=\frac{d\log\log T}{d\log T}=\frac1{\log T}$.
Thus \(\beta(T)\approx 1/2\) signals \(\sqrt T\)-type growth, while a small
positive value at large \(T\) indicates logarithmic-scale growth.  The
slope estimate is more sensitive to downsampling than the raw peak curve,
because it is derivative-like.  We therefore choose the downsampling resolution $10^4$ grid points to stabilize the displayed finite-difference estimate; nearby choices give the
similar qualitative conclusions.

\paragraph{Input-queued switch.}
The IQS experiment uses an \(n\times n\) input-queued switch with \(n=20\).
Arrivals are independent Bernoulli across entries and time slots:
\[
  A_{ij}(t)\sim
  \mathrm{Bernoulli}\!\left(\frac{1-\eps}{n}\right).
\]
The scheduling rule is MaxWeight. We sweep
\[
  \eps\in\{0.005,0.01,0.02,0.05\}.
\]
The plotted metric is $\max_{0\le t\le T}\normtwo{Q(t)}$. The horizon is $T=10^5$.

\paragraph{Parallel-server model.}
We consider a discrete-time parallel-server queueing system with $L$ customer
classes and $K$ servers. Compatibility between classes and servers is described
by a bipartite graph $G=(\mathcal{L},\mathcal{K},E)$, $\mathcal{L}=\{0,1,\dots,L-1\}$ and $\mathcal{K}=\{0,1,\dots,K-1\}$, 
where $(\ell,k)\in E$ means that server $k$ is able to serve class $\ell$. Figure~\ref{fig:parallel-server-diagram} illustrates the compatibility graph
used in our experiments. In the two-phase parallel-server setup we use a sparse ring (``ring zigzag'')
connectivity with \(L=K=40\): each class \(\ell\) is compatible with
two servers, namely \(k=\ell\) and \(k=(\ell+1)\bmod L\). Equivalently,
\[
E=\{(\ell,\ell),(\ell,(\ell+1)\bmod L): \ell=0,1,\dots,L-1\},
\]
so \(|E|=2L\) and both sides have degree \(2\). For example, class \(0\)
connects to servers \(\{0,1\}\), while class \(39\) connects to servers
\(\{39,0\}\).

\begin{figure}[t]
    \centering
    \includegraphics[width=0.6\linewidth]{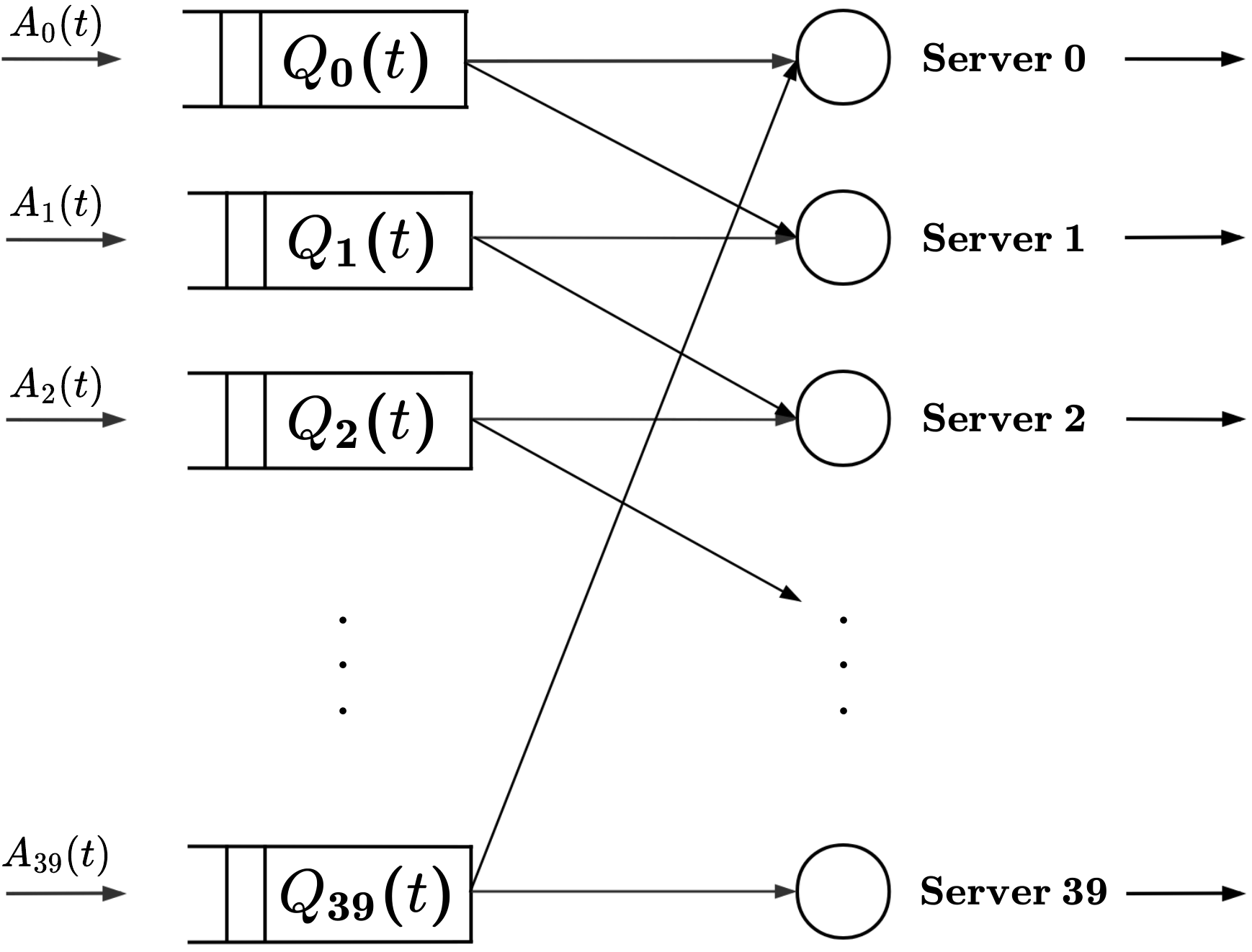}
    \caption{
    Illustrative compatibility graph for the parallel-server model with
    \(L=K=40\). The left column represents customer classes and the right
    column represents servers.
    }
    \label{fig:parallel-server-diagram}
\end{figure}

Let $Q_\ell(t)$ denote the number of waiting class-$\ell$ customers at the
beginning of slot $t$, and write $Q(t)=(Q_0(t),\ldots,Q_{L-1}(t))$. Arrivals are independent Bernoulli across classes and time slots:
\[
    A_\ell(t)\sim \mathrm{Bernoulli}\left(\frac{(1-\eps)\mu K}{L} \right).
\]
Each server has i.i.d. service probability $p_{\mathrm{svc}}=0.05$ (geometric service times with parameter $\mu$). A busy server remains unavailable until its
current service is completed; upon completion, it becomes available for a new
assignment. 

At each slot, after service completions are revealed, the scheduling policy
assigns idle servers to waiting compatible customers. Let $x_{\ell k}(t)\in
\{0,1\}$ indicate whether idle server $k$ is assigned to class $\ell$ in slot
$t$. Feasibility ensures $x_{\ell k}(t)=0$ if $(\ell,k)\notin E$, $\sum_{\ell:(\ell,k)\in E} x_{\ell k}(t)\le 1$, and $\sum_{k:(\ell,k)\in E} x_{\ell k}(t)\le Q_\ell(t)$,
so that no class is assigned more servers than its current queue length. In our
experiments we use MaxWeight. The waiting queues then evolve as
\[
    Q_\ell(t+1)
    =
    Q_\ell(t)
    - \sum_{k:(\ell,k)\in E} x_{\ell k}(t)
    + A_\ell(t),
    \qquad \ell\in[L].
\]
Thus $Q(t)$ tracks waiting customers only; customers that have been assigned to
servers leave the waiting queue and occupy their servers for a random service
duration. We sweep
\[
  \eps\in\{0.005,0.01,0.02,0.05\}.
\]
The plotted metric is $\max_{0\le t\le T}\normtwo{Q(t)}$. The horizon is $T=2.5\times10^5$.

\subsection{Experiment 2: local versus global geometry}

This experiment uses the two-queue system from Section~\ref{sec:geometry} with
service set
\[
  \mathcal S=\{(1,0),(0,a)\}.
\]
The support function is
\[
  H(q)=\max\{q_1,aq_2\},
\]
and the global Euclidean margin satisfies \(\alpha(\Pi)\asymp a\).  We vary
the anisotropy parameter \(a\) and compare the simulated expected peak with two
predictions:
\begin{itemize}
  \item the global prediction from Theorem~\ref{thm:selfnorm-main},
  which pays the worst-direction margin \(\alpha(\Pi)\);
  \item the local prediction from Proposition~\ref{prop:twoqueue-peak}, exploiting state-space collapse.
\end{itemize}
The simulations track the local prediction much more closely than the global
one, illustrating that state-space collapse can improve the entrance threshold
by replacing worst-case geometry with the geometry actually visited by the
process.

\subsection{Experiment 3: CRP and variance in IQS}

This experiment compares two effects that are invisible from the scalar slack alone: bottleneck geometry and arrival variability. Both comparisons use a generalized $n\times n$ IQS with $n=20$ and $\eps=0.01$. The horizon is $T=10^5$ and the plotted metric is the total-backlog peak $\max_{0\le t\le T}\normone{Q_t}$.

\paragraph{CRP versus non-CRP geometry.}

The third experiment first isolates the effect of bottleneck geometry in a generalized IQS. We fix a switch size \(n\), a fractional arrival matrix $M$, a slack parameter \(\eps\), we let the arrival process $A_t$ satisfy:
\begin{align*}
    A_{t}
:=
\begin{cases}
0, &  \text{ with probability } \varepsilon,\\
M, &  \text{ with probability } 1-\varepsilon.
\end{cases}
\end{align*}
Hence the mean matrix is
\[
\lambda=(1-\eps)M.
\]
Thus \(M\) is the displayed normalized load matrix, and the factor \(1-\eps\) enforces the same row/column slack in both instances. Service is deterministic and schedules are chosen by MaxWeight.
We compare two choices of \(M\). The two load matrices $M^{\mathrm{nonCRP}}$ and $M^{\mathrm{CRP}}$ are defined in Section~\ref{sec:experiments}, which makes the dominant bottleneck effectively single-face: the first row is the main active constraint, while the diagonal entries distribute the remaining traffic without creating the same row/column intersection bottleneck as in the non-CRP construction. In both the actual arrival mean is \((1-\eps)M\), so the comparison keeps the global load scale fixed and changes only the active-face geometry.

\paragraph{Independent versus synchronous arrivals.}
The two arrival processes are defined in Section~\ref{sec:experiments}. Both have mean load $(1-\eps)/n$ per entry. 

\subsection{Code and data availability}
The replication package is provided at \href{https://github.com/learning-decision/finite-time-queue-peak-laws}{https://github.com/learning-decision/finite-time-queue-peak-laws}. The experiments use synthetic data
generated by the code; no proprietary data is required.

\section{Scalar drift-to-peak lemma}\label{app:scalar}

The queueing arguments in this paper use a Hajek-type scalar drift-to-peak lemma
to produce logarithmic finite-horizon bounds. \citet*{Hajek82Drift} proved exponential bounds
for first-hitting and occupation times of real-valued processes with
uniform negative drift above a level and exponentially controlled
increments, with queueing applications including the G/G/1 queue. Since
\[
    \left\{\max_{0\le t\le T}X_t\ge x\right\}
    =
    \{\tau_x\le T\},
    \qquad
    \tau_x:=\inf\{t:X_t\ge x\},
\]
the scalar peak lemma developed below should be understood as a tailored
finite-horizon formulation of this hitting-time mechanism.

There is also a quantitative sharpening in the way we use this mechanism. A finite-horizon peak bound can be obtained from \cite{Hajek82Drift} by first deriving the associated fixed-time exponential tail estimate and then union bounding over $0,\ldots,T$. In our notation, this route gives a bound similar to \eqref{eq:exponential-peak}, but with an additional term coming from the uniform fixed-time exponential-moment estimate. After specialization to queueing, that term typically contributes an additive term of order $\eps^{-1}\log(1/\eps)$ to the peak bound, even for a single queue. This term is not intrinsic to the finite-horizon peak and can be unfavorable in the heavy traffic regime where $\eps\downarrow0$. Our scalar lemma avoids this loss. By restarting the stopped exponential supermartingale at each upcrossing of the threshold, it pays only the one-sided upward MGF envelope $M_+(\eta)$ defined in \eqref{eq:UP}, rather than a uniform fixed-time normalization.

\subsection{Exponential drift above a threshold}\label{app:scalar-expdrift}

Let $(X_t)_{t\ge 0}$ be a nonnegative process adapted to $(\F_t)$ and write $\Delta X_t:=X_{t+1}-X_t$.

\begin{definition}[Exponential drift above a threshold]\label{def:ED}
Fix $\theta>0$. We say that $X_t$ satisfies \emph{exponential drift above a threshold} with parameters $(\theta,x_0,\gamma)$ if for every $t\ge 0$,
\begin{equation}\label{eq:ED}
\E\big[e^{\theta \Delta X_t}\mid \F_t\big]
\le e^{-\gamma}
\qquad\text{a.s. on }\{X_t\ge x_0\}.
\end{equation}
\end{definition}

\begin{definition}[One-sided upward MGF envelope]\label{def:UP}
Fix $\theta>0$. We say that $X_t$ admits a \emph{one-sided upward MGF envelope} at level $\theta$ if there exists $M_+(\theta)<\infty$ such that
\begin{equation}\label{eq:UP}
\sup_{t\ge 0}\ \E\big[e^{\theta (\Delta X_t)^+}\mid \F_t\big]
\le M_+(\theta)
\qquad\text{a.s.}
\end{equation}
\end{definition}

\begin{lemma}[Exponential drift-to-peak]\label{lm:exponential-peak}
Assume that $X_t$ satisfies \eqref{eq:ED} and \eqref{eq:UP} for some $\theta>0$, $x_0\ge 0$, $\gamma>0$, and $M_+(\theta)<\infty$, and assume also that $\E[e^{\theta X_0}]<\infty$. Then for every $T\ge 0$ and every $u\ge 0$,
\begin{equation}\label{eq:exponential-tail}
\Pbb\!\left(\max_{0\le t\le T} X_t \ge x_0+u\right)
\le
e^{-\theta(x_0+u)}\E[e^{\theta X_0}]
+
T M_+(\theta)e^{-\theta u}.
\end{equation}
If $X_0\le x_0$ a.s., then
\[
\Pbb\!\left(\max_{0\le t\le T} X_t \ge x_0+u\right)
\le (T+1) M_+(\theta)e^{-\theta u}.
\]
In particular, if $X_0\le x_0$ a.s., then with probability at least $1-\delta$,
\begin{equation}\label{eq:exponential-peak}
\max_{0\le t\le T} X_t
\le
x_0 + \frac{1}{\theta}\log M_+(\theta)
+\frac{1}{\theta}\log\frac{T+1}{\delta}.
\end{equation}
\end{lemma}

\begin{proof}
The path is naturally decomposed into excursions above the threshold $x_0$.

For the initial excursion, define
\[
\tau_0:=\inf\{t\ge 0: X_t<x_0\ \text{or}\ X_t\ge x_0+u\}\wedge (T+1).
\]
On $\{t<\tau_0\}$ we have $X_t\ge x_0$, so \eqref{eq:ED} implies
\[
\E[e^{\theta X_{t+1}}\mid \F_t] \le e^{\theta X_t}.
\]
Hence $(e^{\theta X_{t\wedge\tau_0}})_{t\ge 0}$ is a nonnegative supermartingale. Optional stopping gives
\[
e^{\theta(x_0+u)}\Pbb(X_{\tau_0}\ge x_0+u)
\le \E[e^{\theta X_{\tau_0}}]
\le \E[e^{\theta X_0}],
\]
that is,
\begin{equation}\label{eq:exponential-initial}
\Pbb(X_{\tau_0}\ge x_0+u)\le e^{-\theta(x_0+u)}\E[e^{\theta X_0}].
\end{equation}

Now fix $s\in\{0,\dots,T-1\}$ and consider an excursion that starts between times $s$ and $s+1$. Let
\[
E_s:=\Bigl\{X_s<x_0,\ X_{s+1}\ge x_0,\ \exists\,t\in\{s+1,\dots,T\}: X_t\ge x_0+u
\ \text{before the next return below }x_0\Bigr\}.
\]
On $E_s$, define
\[
\tau_s:=\inf\{t\ge s+1: X_t<x_0\ \text{or}\ X_t\ge x_0+u\}\wedge (T+1).
\]
Conditionally on $\F_{s+1}$, the same supermartingale argument on the interval $[s+1,\tau_s]$ yields
\[
\Pbb(E_s\mid \F_{s+1})
\le
\one_{\{X_s<x_0,\ X_{s+1}\ge x_0\}}
e^{-\theta(x_0+u)}e^{\theta X_{s+1}}.
\]
Taking conditional expectation with respect to $\F_s$, and using $X_{s+1}\le x_0 + (\Delta X_s)^+$ on $\{X_s<x_0,\ X_{s+1}\ge x_0\}$, we obtain
\[
\Pbb(E_s\mid \F_s)
\le
\one_{\{X_s<x_0\}} e^{-\theta u}\E[e^{\theta(\Delta X_s)^+}\mid \F_s]
\le
M_+(\theta)e^{-\theta u}.
\]
Therefore
\begin{equation}\label{eq:exponential-excursion}
\Pbb(E_s)\le M_+(\theta)e^{-\theta u}
\qquad\text{for every }s=0,\dots,T-1.
\end{equation}

If $\max_{0\le t\le T}X_t\ge x_0+u$, then either the initial excursion hits $x_0+u$, or there exists $s\in\{0,\dots,T-1\}$ such that $E_s$ occurs. Combining \eqref{eq:exponential-initial} and \eqref{eq:exponential-excursion} with a union bound proves \eqref{eq:exponential-tail}. If $X_0\le x_0$ a.s., then the first term in \eqref{eq:exponential-tail} is at most $e^{-\theta u}$, and Jensen's inequality gives $M_+(\theta)\ge 1$, so
\[
\Pbb\!\left(\max_{0\le t\le T}X_t\ge x_0+u\right)\le (T+1)M_+(\theta)e^{-\theta u}.
\]
The high-probability bound is the inversion of this inequality.
\end{proof}

\subsection{From negative drift and Bernstein fluctuations to exponential drift}\label{app:scalar-bernstein}

The previous lemma becomes useful once the exponential drift is verified from a more primitive drift condition and a one-step fluctuation bound.

\begin{definition}[Negative drift above a threshold]\label{def:ND}
We say that $X_t$ satisfies \emph{negative drift above a threshold} with parameters $(x_0,\Delta)$ if for every $t\ge 0$,
\begin{equation}\label{eq:ND}
\E[\Delta X_t\mid \F_t]\le -\Delta
\qquad\text{a.s. on }\{X_t\ge x_0\}.
\end{equation}
\end{definition}

\begin{definition}[Conditional Bernstein increments]\label{def:CB}
We say that $X_t$ satisfies \emph{conditional Bernstein increments} with parameters $(\sigma^2,b)$ if for every $t$ and every $\theta\in[0,1/b)$,
\begin{equation}\label{eq:CB}
\log \E\!\left[
\exp\Big(\theta(\Delta X_t - \E[\Delta X_t\mid \F_t])\Big)
\middle| \F_t\right]
\le \frac{\theta^2\sigma^2}{2(1-\theta b)}
\qquad\text{a.s.}
\end{equation}
When $b=0$ we interpret $1/b=\infty$.
\end{definition}

\begin{lemma}[Bernstein drift-to-peak]\label{lm:Bernstein-peak}
Assume \eqref{eq:ND} and \eqref{eq:CB}. Fix \(\theta>0\) with \(\theta<1/b\), interpreting \(1/0=\infty\), and assume that the upward MGF envelope \eqref{eq:UP} holds at this \(\theta\). Define
\begin{equation*}
\gamma(\theta)
:=
\theta\Delta - \frac{\theta^2\sigma^2}{2(1-\theta b)}.
\end{equation*}
If \(\gamma(\theta)>0\), then Lemma~\ref{lm:exponential-peak} applies with this choice of \(\theta\) and \(\gamma(\theta)\). In particular, if \(X_0\le x_0\) a.s., then
\[
\Pbb\!\left(\max_{0\le t\le T} X_t \ge x_0+u\right)
\le (T+1)M_+(\theta)e^{-\theta u}.
\]
If \(\sigma^2+b>0\), choose
\begin{equation}\label{eq:theta-star}
\theta_* := \min\left\{\frac{\Delta}{2\sigma^2},\,\frac{1}{2b}\right\},
\end{equation}
where \(\Delta/(2\sigma^2):=\infty\) when \(\sigma^2=0\), and \(1/(2b):=\infty\) when \(b=0\). Then \(\theta_*<\infty\) and \(\gamma(\theta_*)\ge\theta_*\Delta/4\). If \(X_0\le x_0\) a.s. and \eqref{eq:UP} holds at \(\theta_*\), then
\[
\max_{0\le t\le T} X_t
\;\le\;
 x_0 + \frac{1}{\theta_*}\log M_+(\theta_*)
+ \frac{1}{\theta_*}\log\frac{T+1}{\delta}
\]
with probability at least \(1-\delta\). Moreover \(1/\theta_*\asymp \max\{\sigma^2/\Delta,b\}\).

If \(\sigma^2=b=0\), then the Bernstein fluctuation term vanishes and \(\gamma(\theta)=\theta\Delta\) for every \(\theta>0\). In this deterministic-centered case the same tail bound holds for every \(\theta\) at which \eqref{eq:UP} holds; no finite Bernstein scale restricts the tilt.
\end{lemma}

\begin{proof}
Under \eqref{eq:ND} and \eqref{eq:CB}, on \(\{X_t\ge x_0\}\) we have
\begin{align*}
\log \E[e^{\theta \Delta X_t}\mid \F_t]
&=
\theta\E[\Delta X_t\mid \F_t]
+
\log \E\!\left[
\exp\Big(\theta(\Delta X_t-\E[\Delta X_t\mid \F_t])\Big)
\middle| \F_t\right] \\
&\le -\theta\Delta + \frac{\theta^2\sigma^2}{2(1-\theta b)}
= -\gamma(\theta).
\end{align*}
Thus \eqref{eq:ED} holds, and Lemma~\ref{lm:exponential-peak} gives the asserted tail bound.

Assume next that \(\sigma^2+b>0\). The conventions in \eqref{eq:theta-star} make \(\theta_*\) finite and place it in \((0,1/b)\). Since \(\theta_*b\le1/2\) and \(\theta_*\sigma^2\le\Delta/2\),
\[
\frac{\theta_*^2\sigma^2}{2(1-\theta_*b)}
\le
\theta_*^2\sigma^2
\le
\frac{\theta_*\Delta}{2},
\]
so \(\gamma(\theta_*)\ge\theta_*\Delta/2\), and in particular \(\gamma(\theta_*)\ge\theta_*\Delta/4\). The displayed high-probability bound follows by taking
\(u=\theta_*^{-1}\log((T+1)M_+(\theta_*)/\delta)\). The relation \(1/\theta_*\asymp\max\{\sigma^2/\Delta,b\}\) is immediate from the definition of \(\theta_*\).

When \(\sigma^2=b=0\), the curvature term in \(\gamma\) is identically zero, so \(\gamma(\theta)=\theta\Delta\) for every admissible \(\theta\). The first part of the proof therefore applies at any \(\theta>0\) for which the upward MGF envelope is available.
\end{proof}

\section[Proofs for Section 2]{Proofs for Section~\ref{sec:model}}\label{app:upper}

This section proves finite-time peak bounds for generalized switches under MaxWeight.

\subsection[A warm-up via global geometry]{Theorem~\ref{thm:warmup-main}: a warm-up via global geometry}\label{app:upper-warmup}

Let
\[
\Phi(q):=\frac12 \normtwo{q}^2.
\]
The starting point is the standard quadratic drift inequality.

\begin{lemma}[Quadratic drift under conditional interior slack]\label{lem:quad-drift}
Consider MaxWeight scheduling under Assumptions~\ref{ass:interior} and~\ref{ass:bounded-primitives}. Set $D:=A_{\max}+S_{\max}$ and $B:=D^2/2$. Then
\begin{equation}\label{eq:quad-drift}
\E\big[\Phi(Q_{t+1})-\Phi(Q_t)\mid \F_t\big]
\le B - \eps H(Q_t)
\qquad\text{a.s. for every }t.
\end{equation}
\end{lemma}

\begin{proof}
From \eqref{eq:queue-recursion} and $(x^+)^2\le x^2$ componentwise,
\[
\normtwo{Q_{t+1}}^2
\le
\normtwo{Q_t + A_t - S_t}^2
=
\normtwo{Q_t}^2 + \normtwo{A_t-S_t}^2 + 2\ip{Q_t}{A_t-S_t}.
\]
Divide by $2$ and take conditional expectation. Since $\normtwo{A_t-S_t}\le D$ a.s., the quadratic term contributes at most $B$. By \eqref{eq:service-mean} and MaxWeight,
\[
\E[\ip{Q_t}{S_t}\mid \F_t]=\ip{Q_t}{s_t}=H(Q_t).
\]
Finally, \eqref{eq:interior} implies
\[
\ip{Q_t}{\lambda_t}\le (1-\eps) H(Q_t).
\]
Substituting these bounds proves \eqref{eq:quad-drift}.
\end{proof}

\begin{lemma}[Geometry comparison]\label{lem:H-alpha2}
For every $q\in\Rp^d$,
\[
H(q)\ge \alphapi \normtwo{q}.
\]
\end{lemma}

\begin{proof}
If $q=0$ there is nothing to prove. Otherwise write $w=q/\normtwo{q}$. By definition of $\alphapi$, $H(w)\ge \alphapi$. Since $H$ is positively homogeneous, $H(q)=\normtwo{q}H(w)\ge \alphapi \normtwo{q}$.
\end{proof}

\begin{theorem}[Baseline peak law from global geometry]\label{thm:warmup-main}
Consider MaxWeight scheduling under Assumptions~\ref{ass:interior} and~\ref{ass:bounded-primitives}. Let $D:=A_{\max}+S_{\max}$ and assume $\alphapi>0$. Then for every $T\ge1$ and $\delta\in(0,1)$, with probability at least $1-\delta$,
\[
\PeakT
\;\lesssim\;
\frac{D^2}{\eps\alphapi}\log\frac{T+1}{\delta}
+
\frac{D^2}{\eps\alphapi}.
\]
\end{theorem}

This appendix benchmark is the direction-blind quadratic-drift proof. It is useful for comparison because it places the global margin $\alphapi$ in both the threshold and the logarithmic coefficient.

\begin{proof}[Proof of Theorem~\ref{thm:warmup-main}]
The proof turns the quadratic drift of Lemma~\ref{lem:quad-drift} into a linear drift for $X_t:=\normtwo{Q_t}$. The concavity of the square root gives
\[
X_{t+1}-X_t
\le \frac{\Phi(Q_{t+1})-\Phi(Q_t)}{X_t}
\qquad\text{on }\{X_t>0\}.
\]
Taking conditional expectation and using Lemmas~\ref{lem:quad-drift} and \ref{lem:H-alpha2} yields
\[
\E[\Delta X_t\mid \F_t]
\le \frac{B}{X_t} - \eps \alphapi
\qquad\text{on }\{X_t>0\}.
\]
Hence if $X_t\ge x_0:=2B/(\eps\alphapi)$, then
\[
\E[\Delta X_t\mid \F_t] \le -\frac12\eps\alphapi.
\]
Also $|\Delta X_t|\le D$ a.s. by the triangle inequality, so Hoeffding's lemma gives \eqref{eq:CB} with $\sigma^2=D^2$ and $b=0$, while \eqref{eq:UP} holds with $M_+(\theta)\le e^{\theta D}$. Lemma~\ref{lm:Bernstein-peak} therefore yields
\[
\max_{0\le t\le T}X_t
\;\lesssim\;
x_0 + D + \frac{D^2}{\eps\alphapi}\log\frac{T+1}{\delta}
\qquad\text{with probability at least }1-\delta.
\]
Since $\alphapi\le S_{\max}\le D$, we have $x_0=D^2/(\eps\alphapi)\ge D$, so the additive $D$ is absorbed into the threshold. This proves the stated bound.
\end{proof}

The theorem is already informative: it is the nonasymptotic analogue of the classical mean bound $\E_\pi\normtwo{Q}\lesssim D^2/(\eps\alphapi)$. But it still attributes the logarithmic term to the worst global direction. The next subsection shows that this is not the right fluctuation scale.

\subsection[The self-normalized theorem for MaxWeight]{Theorem~\ref{thm:selfnorm-main}: the self-normalized theorem for MaxWeight}\label{app:upper-selfnorm}

The key estimate is a one-step bound for the queue norm itself.

\begin{lemma}[Norm increment bound]\label{lem:norm-inc}
For every $q\in\R^d\setminus\{0\}$, every $y\in\R^d$, and $u=q/\normtwo{q}$,
\begin{equation}\label{eq:norm-inc}
\normtwo{q+y}-\normtwo{q}
\le \ip{u}{y} + \frac{\normtwo{y}^2}{2\normtwo{q}}.
\end{equation}
\end{lemma}

\begin{proof}
Write
\[
\normtwo{q+y}
=
\sqrt{\normtwo{q}^2 + 2\ip{q}{y} + \normtwo{y}^2}
=
\normtwo{q}\sqrt{1 + \frac{2\ip{u}{y}}{\normtwo{q}} + \frac{\normtwo{y}^2}{\normtwo{q}^2}},
\]
and apply $\sqrt{1+z}\le 1+z/2$ for $z\ge -1$.
\end{proof}

\begin{lemma}[MGF of a bounded nonnegative variable]\label{lem:bounded-nonneg-mgf}
Let $Z$ be a nonnegative random variable with $0\le Z\le A$ a.s., and write $m:=\E[Z]$. Then for every $\theta\in[0,1/A)$,
\[
\log \E[e^{\theta Z}]
\le
\theta m + \frac{\theta^2 A m}{2(1-\theta A)}.
\]
The same bound holds conditionally: if $Z$ is $\F_{t+1}$-measurable and $0\le Z\le A$ a.s., then
\[
\log \E[e^{\theta Z}\mid \F_t]
\le
\theta \E[Z\mid \F_t]
+
\frac{\theta^2 A\,\E[Z\mid \F_t]}{2(1-\theta A)}
\qquad\text{a.s.}
\]
\end{lemma}

\begin{proof}
For $x\in[0,1)$ we have
\[
e^x \le 1+x+\frac{x^2}{2(1-x)}.
\]
Applying this with $x=\theta Z$ and using $Z^2\le AZ$ gives
\[
e^{\theta Z}
\le
1+\theta Z+\frac{\theta^2 A Z}{2(1-\theta A)}.
\]
Taking expectations and using $\log(1+y)\le y$ proves the unconditional bound; the conditional version is identical.
\end{proof}

\begin{lemma}[Directional net-input MGF]\label{lem:netinput-mgf}
Assume Assumption~\ref{ass:bounded-primitives}. Fix $t$ and an $\F_t$-measurable vector $u\in\Rp^d$ with $\normtwo u=1$. Let
\[
 m_t(u):=\ip{u}{\lambda_t},
 \qquad
 g_t(u):=\ip{u}{s_t}.
\]
For $D:=A_{\max}+S_{\max}$ and every $\theta\in[0,1/D)$,
\begin{equation}\label{eq:netinput-mgf}
\log \E\!\left[
\exp\!\left(\theta\bigl(\ip{u}{A_t}+\ip{u}{s_t-S_t}\bigr)\right)
\middle|\F_t\right]
\le
\theta m_t(u)+\frac{\theta^2 D\,g_t(u)}{2(1-\theta D)}
\end{equation}
whenever $m_t(u)\le g_t(u)$. If service is deterministic, $S_t=s_t$ a.s., then the sharper bound
\begin{equation}\label{eq:arrival-only-mgf}
\log \E\!\left[e^{\theta\ip{u}{A_t}}\middle|\F_t\right]
\le
\theta m_t(u)+\frac{\theta^2 A_{\max}g_t(u)}{2(1-\theta A_{\max})}
\end{equation}
holds for every $\theta\in[0,1/A_{\max})$.
\end{lemma}

\begin{proof}
We first control the arrival term. Since $u\in\Rp^d$ and $\normtwo u=1$,
$0\le \ip{u}{A_t}\le A_{\max}$ a.s. Lemma~\ref{lem:bounded-nonneg-mgf} gives
\[
\log \E\!\left[e^{\theta\ip{u}{A_t}}\middle|\F_t\right]
\le
\theta m_t(u)+\frac{\theta^2 A_{\max}m_t(u)}{2(1-\theta A_{\max})}.
\]
For the service-shortfall term, set $Y:=\ip{u}{S_t}$. Then $0\le Y\le S_{\max}$ and
$\E[Y\mid\F_t]=g_t(u)$. The elementary inequality $e^{-x}\le 1-x+x^2/2$ for $x\ge0$ yields
\[
\E[e^{-\theta Y}\mid\F_t]
\le 1-\theta g_t(u)+\frac{\theta^2S_{\max}g_t(u)}{2}.
\]
Multiplying by $e^{\theta g_t(u)}$ and using $\log(1+z)\le z$ gives
\[
\log \E[e^{\theta(g_t(u)-Y)}\mid\F_t]
\le \frac{\theta^2S_{\max}g_t(u)}{2}.
\]
Conditional independence of $A_t$ and $S_t$ given $\F_t$ factorizes the two conditional MGFs. Since $m_t(u)\le g_t(u)$ and $\theta D<1$,
\[
\frac{A_{\max}m_t(u)}{1-\theta A_{\max}}+S_{\max}g_t(u)
\le
\frac{(A_{\max}+S_{\max})g_t(u)}{1-\theta D}.
\]
Combining the displayed bounds proves \eqref{eq:netinput-mgf}. If $S_t=s_t$ a.s., the service-shortfall term vanishes and the arrival estimate, together with $m_t(u)\le g_t(u)$, gives \eqref{eq:arrival-only-mgf}.
\end{proof}

\begin{proof}[Proof of Theorem~\ref{thm:selfnorm-main}]
Set $X_t:=\normtwo{Q_t}$ and let $u_t:=Q_t/X_t$ on $\{X_t>0\}$. By nonexpansiveness of projection onto $\Rp^d$ and Lemma~\ref{lem:norm-inc},
\begin{equation}\label{eq:deltaX-basic}
\Delta X_t
\le
\ip{u_t}{A_t}+\ip{u_t}{s_t-S_t}-H(u_t)+\frac{D^2}{2X_t}
\qquad\text{on }\{X_t>0\}.
\end{equation}
Here $D=A_{\max}+S_{\max}$ and we used $\ip{u_t}{s_t}=H(u_t)$ under MaxWeight.

Choose
\[
x_0 := \frac{D^2}{\eps\alphapi}.
\]
If $X_t\ge x_0$, then Lemma~\ref{lem:H-alpha2} gives
\[
\frac{D^2}{2X_t}
\le \frac{\eps\alphapi}{2}
\le \frac{\eps}{2}H(u_t),
\]
and hence
\begin{equation}\label{eq:deltaX-reduced-random-service}
\Delta X_t
\le
\ip{u_t}{A_t}+\ip{u_t}{s_t-S_t}
-
\Bigl(1-\frac\eps2\Bigr)H(u_t)
\qquad\text{on }\{X_t\ge x_0\}.
\end{equation}
For this direction, set $m_t:=\ip{u_t}{\lambda_t}$ and $g_t:=\ip{u_t}{s_t}=H(u_t)$. The slack condition gives $m_t\le(1-\eps)g_t$, so Lemma~\ref{lem:netinput-mgf} yields, for every $\theta\in[0,1/D)$,
\begin{align*}
\log \E[e^{\theta\Delta X_t}\mid \F_t]
&\le
-\theta\Bigl(1-\frac\eps2\Bigr)g_t
+\theta m_t
+\frac{\theta^2Dg_t}{2(1-\theta D)} \\
&\le
-\frac{\theta\eps}{2}g_t
+\frac{\theta^2Dg_t}{2(1-\theta D)}.
\end{align*}
Take
\[
\theta:=\frac{\eps}{4D}.
\]
Then $\theta D\le 1/4$, and the last display is at most
\[
-\frac{\theta\eps}{3}g_t
\le
-\frac{\theta\eps\alphapi}{3}
\qquad\text{on }\{X_t\ge x_0\}.
\]
Thus the exponential-drift condition \eqref{eq:ED} holds. Moreover, a service realization can only decrease queues, so positive norm increments are caused by arrivals:
\[
(\Delta X_t)^+
\le \normtwo{(Q_t+A_t)^+}-\normtwo{Q_t}
\le \normtwo{A_t}
\le A_{\max}
\le D.
\]
Therefore \eqref{eq:UP} holds with $M_+(\theta)\le e^{\theta D}$. Since $Q_0=0$, Lemma~\ref{lm:exponential-peak} gives
\[
\max_{0\le t\le T}X_t
\lesssim
x_0+D+\frac1\theta\log\frac{T+1}{\delta}
\]
with probability at least $1-\delta$. Finally, $1/\theta=4D/\eps$ and $x_0\ge D$ because $\alphapi\le H(u)\le S_{\max}\le D$ for some unit $u\in\Rp^d$. This proves the high-probability bound.

\paragraph{Expectation bound.} More precisely, we have just shown that
\[
\Pbb\!\left(\PeakT\ge x_0 + D + s\right)\le (T+1)e^{-\theta s}
\qquad\text{for every }s\ge 0,
\]
where $x_0=D^2/(\eps\alphapi)$ and $\theta=\eps/(4D)$. Therefore
\begin{align*}
\E[\PeakT]
&\le x_0 + D + \int_0^\infty \min\{1,(T+1)e^{-\theta s}\}\,\dd s \\
&\le x_0 + D + \frac{1+\log(T+1)}{\theta}.
\end{align*}
Since $x_0\ge D$ and $1/\theta\asymp D/\eps$, this yields
\[
\E[\PeakT]\lesssim \frac{D}{\eps}\log(T+1) + \frac{D^2}{\eps\alphapi},
\]
as claimed in the expectation bound.

\paragraph{Deterministic service refinement.} When $S_t=s_t$ a.s., the term $\ip{u_t}{s_t-S_t}$ vanishes in \eqref{eq:deltaX-reduced-random-service}. Applying the deterministic-service part of Lemma~\ref{lem:netinput-mgf} gives, on $\{X_t\ge x_0\}$,
\[
\log \E[e^{\theta\Delta X_t}\mid\F_t]
\le
-\frac{\theta\eps}{2}H(u_t)+\frac{\theta^2A_{\max}H(u_t)}{2(1-\theta A_{\max})}.
\]
With $\theta:=\eps/(4A_{\max})$ when $A_{\max}>0$, the right-hand side is at most $-c\theta\eps H(u_t)$ for a universal $c>0$. Also $(\Delta X_t)^+\le A_{\max}$. The same exponential-drift peak lemma therefore replaces $1/\theta\asymp D/\eps$ by $A_{\max}/\eps$, while the entrance threshold remains $x_0=D^2/(\eps\alphapi)$. If $A_{\max}=0$, the queue is nonincreasing from $Q_0=0$ and the claim is immediate. Integrating the resulting tail bound proves the expectation statement.

This proves the entire theorem.
\end{proof}

\subsection{Comparison with the classical steady-state bound}\label{app:upper-expected}

The finite-time peak law should be compared with the classical steady-state mean bound derived from the same quadratic drift.

\begin{proposition}[Classical steady-state mean comparison]\label{prop:steady-state-mean}
Assume the time-homogeneous Markov-chain version of Lemma~\ref{lem:quad-drift}: the primitive conditional laws are fixed, MaxWeight is used, and the bounded drift inequality \eqref{eq:quad-drift} holds with the same constants at every state. Suppose the chain has an invariant distribution \(\pi\) such that, for \(\bar Q\sim\pi\), \(\E\|\bar Q\|_2^2<\infty\). For a continuous, nonnegative, positively homogeneous functional \(V:\Rp^d\to\Rp\) with $V(q)>0$ for all $q\ne0, q\ge0$, define
\[
\alpha_V(\PiC):=\inf\{\Hpi(w): w\in\Rp^d,\ V(w)=1\}.
\]
If \(\alpha_V(\PiC)>0\), then
\[
\E[V(\bar Q)]\le \frac{D^2}{2\eps\alpha_V(\PiC)}.
\]
In particular, taking \(V(q)=\|q\|_2\), for which \(\alpha_V(\PiC)=\alphapi\), gives
\[
\E\normtwo{\bar Q}\le \frac{D^2}{2\eps\alphapi}.
\]
\end{proposition}

\begin{proof}
The finite second-moment assumption justifies taking expectations of the unbounded quadratic Lyapunov drift in steady state. Since \(Q_t\sim\pi\) implies \(Q_{t+1}\sim\pi\),
\[
\E[\Phi(Q_{t+1})-\Phi(Q_t)]=0.
\]
Taking expectations in \eqref{eq:quad-drift} gives
\[
\E[H(\bar Q)]\le \frac{D^2}{2\eps}.
\]
By definition of \(\alpha_V(\PiC)\), positive homogeneity gives \(H(q)\ge \alpha_V(\PiC)V(q)\) for every \(q\in\Rp^d\). The result follows.
\end{proof}

The proposition controls a steady-state mean. Theorem~\ref{thm:selfnorm-main} controls the entire peak process with high probability. The two bounds have the same base scale. The extra $\log T$ is the price of looking at extremes over a horizon instead of at a single steady-state sample. This is why we call the logarithmic term the finite-time shadow of exponential steady-state tails.

\subsection{Directional certificate and LyapOpt}\label{app:lyapopt-certificate}

Here we provide the directional certificate in Remark~\ref{rem:lyapopt-main}: any scheduling policy satisfying this certificate will enjoy a self-normalized logarithmic queue peak law. We then verify this directional certificate for LyapOpt \cite{LTX25}.

\begin{proposition}[A general directional certificate]\label{prop:directional-certificate-peak}
Consider any $\F_t$-measurable policy $s_t\in\Sset$ under Assumptions~\ref{ass:interior} and~\ref{ass:bounded-primitives}, with $Q_0=0$. Let $X_t:=\normtwo{Q_t}$ and let $u_t:=Q_t/X_t$ on $\{X_t>0\}$. Suppose that there exist constants $x_{\rm cert}\ge0$, $\eps_\star\in(0,1)$, $c_G\ge1$, and $\underline\alpha>0$, and an $\F_t$-measurable random variable $G_t$, such that, on $\{X_t\ge x_{\rm cert}\}$,
\[
G_t\ge \underline\alpha,
\qquad
G_t\le\ip{u_t}{s_t}\le c_GG_t,
\qquad
\ip{u_t}{\lambda_t}\le(1-\eps_\star)G_t.
\]
Set $D:=A_{\max}+S_{\max}$. Then, for every $T\ge1$ and $\delta\in(0,1)$,
\[
\Pbb\!\left(
\PeakT
\lesssim
x_{\rm cert}
+
\frac{D^2}{\eps_\star\underline\alpha}
+
\frac{c_GD}{\eps_\star}\log\frac{T+1}{\delta}
\right)
\ge1-\delta.
\]
Moreover,
\[
\E[\PeakT]
\lesssim
x_{\rm cert}
+
\frac{D^2}{\eps_\star\underline\alpha}
+
\frac{c_GD}{\eps_\star}\log(T+1).
\]
If service is deterministic after the decision, the two conclusions hold with $c_GD/\eps_\star$ replaced by $c_GA_{\max}/\eps_\star$.
\end{proposition}

\begin{proof}
We repeat the MaxWeight proof with $G_t$ replacing $H(u_t)$. Set
\[
 x_0:=\max\left\{x_{\rm cert},\frac{D^2}{\eps_\star\underline\alpha}\right\}.
\]
On $\{X_t\ge x_0\}$, the norm increment bound gives
\[
\Delta X_t
\le
\ip{u_t}{A_t}+\ip{u_t}{s_t-S_t}-\ip{u_t}{s_t}
+\frac{D^2}{2X_t}
\le
\ip{u_t}{A_t}+\ip{u_t}{s_t-S_t}
-\left(1-\frac{\eps_\star}{2}\right)G_t .
\]
Indeed, the last step uses $\ip{u_t}{s_t}\ge G_t$ and $D^2/(2X_t)\le \eps_\star\underline\alpha/2\le\eps_\star G_t/2$. Put $m_t:=\ip{u_t}{\lambda_t}$ and $g_t:=\ip{u_t}{s_t}$. The certificate gives $m_t\le(1-\eps_\star)G_t\le g_t$, so Lemma~\ref{lem:netinput-mgf} yields, for $\theta\in[0,1/D)$,
\[
\log\E[e^{\theta\Delta X_t}\mid\F_t]
\le
-\theta\left(1-\frac{\eps_\star}{2}\right)G_t
+\theta m_t
+\frac{\theta^2Dg_t}{2(1-\theta D)}
\le
-\frac{\theta\eps_\star}{2}G_t
+\frac{\theta^2Dc_GG_t}{2(1-\theta D)} .
\]
Choose $\theta:=\eps_\star/(4c_GD)$. Since $c_G\ge1$ and $\eps_\star<1$, we have $\theta D\le1/4$, and the last display is at most $-\theta\eps_\star G_t/3\le -\theta\eps_\star\underline\alpha/3$. Thus the exponential-drift condition of Definition~\ref{def:ED} holds above $x_0$. As in the proof of Theorem~\ref{thm:selfnorm-main}, service can only decrease the queue relative to adding arrivals alone, so $(\Delta X_t)^+\le A_{\max}\le D$ and Definition~\ref{def:UP} holds with $M_+(\theta)\le e^{\theta D}$. Since $Q_0=0$, Lemma~\ref{lm:exponential-peak} gives the stated high-probability bound, and integrating its tail gives the expectation bound.

When $S_t=s_t$ a.s., the service-noise term vanishes. The arrival-only bound in Lemma~\ref{lem:netinput-mgf} gives the same argument with $D$ replaced by $A_{\max}$ in the logarithmic coefficient; if $A_{\max}=0$, then the queue is nonincreasing from $Q_0=0$ and the claim is immediate.
\end{proof}

\begin{definition}[LyapOpt in the present model]\label{def:lyapopt}
The LyapOpt policy chooses any $\F_t$-measurable minimizer
\begin{equation}\label{eq:LyapOpt}
 s_t^{\rm LO}
 \in
 \arg\min_{s\in\Sset}\normtwo{(Q_t-s)^+}^2 .
\end{equation}
This is the one-step quadratic Lyapunov lookahead of \cite{LTX25}, specialized to the fixed scheduling set and mean-service convention used here.
\end{definition}

\begin{lemma}[LyapOpt preserves directional service above the entrance scale]\label{lem:lyapopt-directional-service}
Let $q\in\Rp^d\setminus\{0\}$, set $u:=q/\normtwo{q}$, and let $s^{\rm LO}\in\arg\min_{s\in\Sset}\normtwo{(q-s)^+}^2$. Then
\[
\ip{u}{s^{\rm LO}}
\ge
H(u)-\frac{S_{\max}^2}{2\normtwo{q}}.
\]
\end{lemma}

\begin{proof}
For every $s\in\Sset$ and every coordinate $i$,
\[
\bigl(q_i-s_i\bigr)_+^2\ge q_i^2-2q_is_i .
\]
Hence
\[
\normtwo{(q-s)^+}^2
\ge
\normtwo{q}^2-2\ip{q}{s}.
\]
Let $\bar s\in\argmax_{s\in\Sset}\ip{u}{s}$, so $\ip{u}{\bar s}=H(u)$. By optimality of $s^{\rm LO}$ and by $\normtwo{\bar s}\le S_{\max}$,
\[
\normtwo{q}^2-2\ip{q}{s^{\rm LO}}
\le
\normtwo{(q-s^{\rm LO})^+}^2
\le
\normtwo{(q-\bar s)^+}^2
\le
\normtwo{q-\bar s}^2
\le
\normtwo{q}^2-2\normtwo{q}\,H(u)+S_{\max}^2 .
\]
Dividing by $2\normtwo{q}$ proves the claim.
\end{proof}

\begin{proposition}[Self-normalized peak law for LyapOpt]\label{prop:lyapopt-selfnorm}
Assume Assumptions~\ref{ass:interior} and~\ref{ass:bounded-primitives}, the LyapOpt policy \eqref{eq:LyapOpt}, and $Q_0=0$. Let $D:=A_{\max}+S_{\max}$ and assume $\alphapi>0$. Then, for every $T\ge1$ and $\delta\in(0,1)$, with probability at least $1-\delta$,
\[
\PeakT
\lesssim
\frac{D}{\eps}\log\frac{T+1}{\delta}
+
\frac{D^2}{\eps\alphapi}.
\]
Consequently,
\[
\E[\PeakT]
\lesssim
\frac{D}{\eps}\log(T+1)
+
\frac{D^2}{\eps\alphapi} .
\]
If service is deterministic, the logarithmic coefficient $D/\eps$ can be replaced by $A_{\max}/\eps$.
\end{proposition}

\begin{proof}
Set $X_t:=\normtwo{Q_t}$. Fix a time $t$ with $X_t>0$ and write $u_t=Q_t/X_t$. Lemma~\ref{lem:lyapopt-directional-service} gives
\[
\ip{u_t}{s_t^{\rm LO}}
\ge
H(u_t)-\frac{S_{\max}^2}{2X_t} .
\]
On the event
\[
X_t\ge x_{\rm cert}:=\frac{S_{\max}^2}{\eps\alphapi},
\]
Lemma~\ref{lem:H-alpha2} gives $H(u_t)\ge\alphapi$, and therefore
\[
\ip{u_t}{s_t^{\rm LO}}
\ge
\left(1-\frac\eps2\right)H(u_t).
\]
Set
\[
G_t:=\left(1-\frac\eps2\right)H(u_t),
\qquad
\eps_\star:=\frac\eps2,
\qquad
\underline\alpha:=\left(1-\frac\eps2\right)\alphapi,
\qquad
c_G:=\frac{1}{1-\eps/2}.
\]
Since $\eps\in(0,1)$, $\underline\alpha\ge\alphapi/2$ and $c_G\le2$. The preceding display gives $\ip{u_t}{s_t^{\rm LO}}\ge G_t$, while the support bound gives $\ip{u_t}{s_t^{\rm LO}}\le H(u_t)=c_GG_t$. The interior slack condition gives
\[
\ip{u_t}{\lambda_t}
\le
(1-\eps)H(u_t)
\le
\left(1-\frac\eps2\right)G_t
=
(1-\eps_\star)G_t .
\]
Thus LyapOpt satisfies the certificate of Proposition~\ref{prop:directional-certificate-peak}. Applying that proposition and using $S_{\max}\le D$, $\eps_\star=\eps/2$, $\underline\alpha\ge\alphapi/2$, and $c_G\le2$ gives the high-probability bound. The expectation and deterministic-service statements follow from the corresponding conclusions of Proposition~\ref{prop:directional-certificate-peak}.
\end{proof}

\subsection[One-dimensional lower bound]{Theorem~\ref{thm:1d-lower-main}: one-dimensional lower bound}\label{app:1d}

Fix $A\ge1$ and let
\[
p:=\frac{1-\eps}{A+1}.
\]
Let $(X_t)_{t\ge0}$ be i.i.d. with
\begin{equation}\label{eq:1d-X-dist}
X_t=
\begin{cases}
A,&\text{with probability }p,\\[1mm]
-1,&\text{with probability }1-p.
\end{cases}
\end{equation}
Then $\E X_t=pA-(1-p)=-\eps$. Define
\[
Q_{t+1}=(Q_t+X_t)^+,
\qquad
Q_0=0,
\qquad
M_T:=\max_{0\le t\le T}Q_t.
\]

\begin{lemma}\label{lem:1d-regen}
Let $\sigma_0:=0$ and $\sigma_{r+1}:=\inf\{t>\sigma_r:Q_t=0\}$. Set
\[
L_r:=\sigma_{r+1}-\sigma_r,
\qquad
P_r^{\mathrm{cyc}}:=\max_{\sigma_r\le t\le\sigma_{r+1}}Q_t.
\]
Then $(L_r,P_r^{\mathrm{cyc}})_{r\ge0}$ are i.i.d. Moreover, for every integer $M\ge1$, if $\Sigma_M:=\sum_{r=0}^{M-1}L_r=\sigma_M\le T$, then
\begin{equation*}
M_T\ge\max_{0\le r<M}P_r^{\mathrm{cyc}}.
\end{equation*}
\end{lemma}

\begin{proof}
The process regenerates whenever it returns to zero, so the cycle pairs are i.i.d. If $\sigma_M\le T$, then the first $M$ cycles are completed by time $T$, and their peaks are included in the horizon-$T$ maximum.
\end{proof}

\begin{lemma}\label{lem:1d-cycle-length}
There is a universal constant $C$ such that $\E L_0\le 3/\eps$. Consequently, if $M:=\lfloor\eps T/12\rfloor$, then
\begin{equation}\label{eq:1d-SigmaM-good}
\Pbb(\Sigma_M\le T)\ge \frac34.
\end{equation}
\end{lemma}

\begin{proof}
If the first increment of a cycle is $-1$, the cycle length is one. If it is $+A$, compare the subsequent evolution with the unreflected walk
\[
W_m=A+\sum_{r=1}^m X_r,
\qquad
\tau:=\inf\{m\ge0:W_m\le0\}.
\]
For $N\ge1$, the stopped time $\tau\wedge N$ is bounded, so optional stopping for the martingale $W_m+\eps m$ gives
\[
\E W_{\tau\wedge N}+\eps\E[\tau\wedge N]=A.
\]
Before time \(\tau\) the walk is positive, and at time \(\tau\) it can undershoot the origin by at most one, because the only downward jump is \(-1\). Thus \(W_{\tau\wedge N}\ge -1\). Hence
\[
  \eps\E[\tau\wedge N]
  = A-\E W_{\tau\wedge N}
  \le A+1.
\]
Monotone convergence gives \(\E\tau\le(A+1)/\eps\). Therefore
\[
\E L_0\le 1+p\left(1+\frac{A+1}{\eps}\right)
\le 2+\frac{1}{\eps}
\le \frac3\eps.
\]
Thus $\E\Sigma_M\le 3M/\eps\le T/4$, and Markov's inequality gives \eqref{eq:1d-SigmaM-good}.
\end{proof}

\begin{lemma}\label{lem:1d-beta}
For $k\ge A$, let $\beta_k:=\Pbb(P_0^{\mathrm{cyc}}\ge k)$. There exists $r>1$ solving
\begin{equation}\label{eq:1d-CL}
p r^A+(1-p)r^{-1}=1,
\end{equation}
and for this $r$,
\begin{equation}\label{eq:1d-beta-pre}
\beta_k\ge p\,(r^A-1)\,r^{-(k+A)}.
\end{equation}
\end{lemma}

\begin{proof}
First we justify the root. Put
\[
\phi(\theta):=p e^{A\theta}+(1-p)e^{-\theta}.
\]
Then $\phi(0)=1$, $\phi'(0)=-\eps<0$, $\phi$ is strictly convex, and $\phi(\theta)\to\infty$ as $\theta\to\infty$. Hence there is a unique positive root $\theta^*>0$ of $\phi(\theta)=1$; set $r=e^{\theta^*}>1$. This is exactly \eqref{eq:1d-CL}.

Condition on the first increment being \(+A\). Until the cycle returns to zero, the reflected and unreflected walks agree. Let the unreflected walk start from \(A\), and let \(\tau_0\) be its first time at or below zero and \(\tau_k\) its first time at or above \(k\). The process \(r^{W_t}\) is a martingale by \eqref{eq:1d-CL}.

First note that \(\tau_0\wedge\tau_k<\infty\) a.s. For \(N\ge1\), optional stopping for the martingale \(W_t+\eps t\) at \((\tau_0\wedge\tau_k)\wedge N\) gives
\[
  \E W_{(\tau_0\wedge\tau_k)\wedge N}
  +\eps\E[(\tau_0\wedge\tau_k)\wedge N]
  =A.
\]
The stopped walk is always at least \(-1\): it is positive before \(\tau_0\), and the first crossing of the nonpositive half-line has overshoot at most one. Hence
\[
  \E[(\tau_0\wedge\tau_k)\wedge N]\le\frac{A+1}{\eps}.
\]
Letting \(N\to\infty\) shows that \(\tau_0\wedge\tau_k\) is finite a.s.

Now stop the martingale \(r^{W_t}\) at \((\tau_0\wedge\tau_k)\wedge N\). The stopped values are bounded by \(r^{k+A}\): before stopping the walk lies below \(k\), at \(\tau_k\) it overshoots by at most \(A\), and at \(\tau_0\) it lies in \([-1,0]\). Dominated convergence gives
\[
  r^A=\E_A r^{W_{\tau_0\wedge\tau_k}}.
\]
On \(\{\tau_0<\tau_k\}\), the terminal value is at most \(1\); on \(\{\tau_k<\tau_0\}\), it is at most \(r^{k+A}\). Therefore
\[
r^A
\le
1\cdot\Pbb_A(\tau_0<\tau_k)+r^{k+A}\Pbb_A(\tau_k<\tau_0),
\]
which gives \(\Pbb_A(\tau_k<\tau_0)\ge (r^A-1)r^{-(k+A)}\). Multiplying by the probability \(p\) of the first upward jump proves the claim.
\end{proof}

\begin{lemma}\label{lem:1d-theta}
Let $\theta^*:=\log r$, where $r$ is defined by \eqref{eq:1d-CL}. For $\eps\le1/8$,
\begin{equation*}\label{eq:1d-theta-bounds}
\frac{\eps}{16A}\le \theta^*\le \frac{8\eps}{A}.
\end{equation*}
Consequently, for all $k\ge A$,
\begin{equation}\label{eq:1d-beta-exp}
\beta_k\ge c\,\frac{\eps}{A}\exp\!\left(-C\frac{\eps}{A}(k+A)\right)
\end{equation}
for universal constants $c,C>0$.
\end{lemma}

\begin{proof}
Let $\phi(\theta):=\E e^{\theta X_0}$. Then $\phi(0)=1$, $\phi'(0)=-\eps$, and $\theta^*$ is the positive root of $\phi(\theta)=1$.
For $\theta_0:=\eps/(16A)$, using $e^x\le1+x+x^2$ for $0\le x\le1$ and $e^{-x}\le1-x+x^2/2$ gives
\[
\phi(\theta_0)-1
\le -\eps\theta_0+\theta_0^2\bigl(pA^2+(1-p)/2\bigr)
\le -\frac{\eps^2}{16A}+\frac{\eps^2}{128A}<0.
\]
Thus $\theta^*\ge\theta_0$. For $\theta_1:=8\eps/A$, using $e^x\ge1+x+x^2/2$ and $e^{-x}\ge1-x$ gives
\[
\phi(\theta_1)-1
\ge -\eps\theta_1+\frac12pA^2\theta_1^2.
\]
Since $pA^2=(1-\eps)A^2/(A+1)\ge 7A/16$, the right-hand side is positive; hence $\theta^*\le\theta_1$. Also $r^A-1=e^{A\theta^*}-1\ge A\theta^*/2\ge\eps/32$. Combining this with $p\asymp1/A$ and \eqref{eq:1d-beta-pre} proves \eqref{eq:1d-beta-exp}.
\end{proof}

\begin{lemma}\label{lem:1d-long}
There exist universal constants $c,C>0$ such that, for $T\ge C A/\eps^2$,
\[
\E[M_T]
\ge
c\frac{A}{\eps}\log\!\left(1+\frac{\eps^2T}{A}\right).
\]
\end{lemma}

\begin{proof}
Let $M:=\lfloor\eps T/12\rfloor$. Choose
\[
k:=\left\lfloor
\frac{A}{2C_0\eps}\log\!\left(1+\frac{\eps^2T}{A}\right)
\right\rfloor-A,
\]
where $C_0$ is the constant in \eqref{eq:1d-beta-exp}. For $C$ large enough, $k\ge A$ and
\[
\beta_k\ge c_0\frac{\eps}{A}\left(1+\frac{\eps^2T}{A}\right)^{-1/2}.
\]
Also $M\ge\eps T/24$, so $M\beta_k\ge\log 4$ after increasing $C$. Hence
\[
\Pbb\left(\max_{0\le r<M}P_r^{\mathrm{cyc}}\ge k\right)
\ge 1-(1-\beta_k)^M\ge\frac34.
\]
By \eqref{eq:1d-SigmaM-good} and the union bound, the event above intersects $\{\Sigma_M\le T\}$ with probability at least $1/2$. Lemma~\ref{lem:1d-regen} then gives $\Pbb(M_T\ge k)\ge1/2$, and therefore $\E M_T\ge k/2$. The definition of $k$ gives the result.
\end{proof}

\begin{lemma}\label{lem:1d-short}
There exist universal constants $c,C,c'>0$ such that, whenever
\[
C A\le T\le c'\frac{A}{\eps^2},
\]
one has
\[
\E[M_T]\ge c\sqrt{AT}.
\]
\end{lemma}

\begin{proof}
Let $W_T:=\sum_{t=0}^{T-1} X_t$ and $Z_T:=W_T+\eps T=\sum_{t=0}^{T-1}(X_t+\eps)$. The reflected walk dominates the unreflected walk, so $M_T\ge W_T^+$.
The centered increment $\xi:=X_0+\eps$ has variance $\sigma^2\asymp A$ and fourth moment $\E|\xi|^4\le C A^3$. Hence
\[
\E Z_T^2=T\sigma^2\asymp AT,
\qquad
\E Z_T^4\le C(TA^3+T^2A^2)\le C'T^2A^2
\]
for $T\ge CA$. Paley--Zygmund applied to $Z_T^2$ gives $\Pbb(|Z_T|\ge c_0\sqrt{AT})\ge c_1$, and therefore $\E|Z_T|\ge c_2\sqrt{AT}$. Since $\E Z_T=0$, $\E Z_T^+=\frac12\E|Z_T|$. Thus
\[
\E W_T^+
=\E(Z_T-\eps T)^+
\ge \E Z_T^+-\eps T
\ge c_3\sqrt{AT}-\eps T.
\]
Choosing $c'>0$ small enough makes $\eps T\le(c_3/2)\sqrt{AT}$, proving the claim.
\end{proof}

\begin{lemma}[A constant-scale lower bound]\label{lem:1d-big-jump}
There exist universal constants $c,C>0$ such that, whenever $T\ge C A$,
\[
\E[M_T]\ge cA.
\]
\end{lemma}

\begin{proof}
If at least one of the first $T$ increments is equal to $A$, then $M_T\ge A$. Since $\eps\le1/8$ and $A\ge1$,
\[
p=\frac{1-\eps}{A+1}\ge \frac{7}{16A}.
\]
Thus, for $T\ge CA$ with $C$ large enough,
\[
\Pbb(M_T\ge A)
\ge 1-(1-p)^T
\ge 1-\exp\!\left(-\frac{7T}{16A}\right)
\ge \frac12,
\]
and the claim follows.
\end{proof}

\begin{proof}[Proof of Theorem~\ref{thm:1d-lower-main}]
Fix $A\ge1$ and use the construction above. The increment law \eqref{eq:1d-X-dist} is realized by a queue with deterministic service $S_t=s_t=1$ and arrivals equal to $A+1$ with probability $(1-\eps)/(A+1)$ and $0$ otherwise. Thus $A_{\max}=A+1$, $S_{\max}=1$, the upward net-input jump is $A$, and the mean net input is $-\eps$.

Let $(c_s,C_s,c_s')$ be constants for Lemma~\ref{lem:1d-short}, let $(c_\ell,C_\ell)$ be constants for Lemma~\ref{lem:1d-long}, and let $(c_b,C_b)$ be constants for Lemma~\ref{lem:1d-big-jump}. Increase the universal lower-horizon constant in the theorem so that
\[
  T\ge C A
  \quad\Longrightarrow\quad
  T\ge C_sA\quad\text{and}\quad T\ge C_bA.
\]
It remains to prove
\begin{equation}\label{eq:1d-A-target}
\E[M_T]
\ge
c\min\left\{\sqrt{AT},\frac{A}{\eps}\log\!\left(1+\frac{\eps^2T}{A}\right)\right\}.
\end{equation}

If $T\le c_s'A/\eps^2$, Lemma~\ref{lem:1d-short} applies and gives the first term in \eqref{eq:1d-A-target}. If $T\ge C_\ell A/\eps^2$, Lemma~\ref{lem:1d-long} gives the second term. It remains to consider the interval
\[
  c_s'\frac{A}{\eps^2}<T<C_\ell\frac{A}{\eps^2}.
\]
First suppose $c_s'A/\eps^2\ge C_sA$. Then $T_s:=c_s'A/\eps^2$ lies in the range of Lemma~\ref{lem:1d-short}; monotonicity of $T\mapsto M_T$ gives
\[
  \E M_T\ge \E M_{T_s}\ge c_s\sqrt{AT_s}
  =c_s\sqrt{c_s'}\,\frac{A}{\eps}.
\]
On the same interval, the right-hand side of \eqref{eq:1d-A-target} is at most
\[
  \min\left\{\sqrt{C_\ell}\,\frac{A}{\eps},\frac{A}{\eps}\log(1+C_\ell)\right\},
\]
so \eqref{eq:1d-A-target} follows after reducing $c$.

It remains to handle the case $c_s'A/\eps^2<C_sA$. In this case $\eps$ is bounded below by the positive universal constant
\[
\eta:=\min\left\{\frac18,\sqrt{\frac{c_s'}{C_s}}\right\}.
\]
Indeed, if the case can occur under $\eps\le1/8$, then $\eps>\sqrt{c_s'/C_s}\ge\eta$. Since $T<C_\ell A/\eps^2\le C_\ell A/\eta^2$, the right-hand side of \eqref{eq:1d-A-target} is at most a universal multiple of $A$. Lemma~\ref{lem:1d-big-jump}, which applies because $T\ge C_bA$, gives $\E M_T\ge c_bA$ and proves \eqref{eq:1d-A-target} in the last case.

This is exactly the displayed lower bound in Theorem~\ref{thm:1d-lower-main} for all $T\ge CA$.
\end{proof}
\subsection[Geometry-threshold lower bound]{Theorem~\ref{thm:geometry-lower-main}: geometry-threshold lower bound}\label{app:geometry-lower-bound}
\begin{proof}[Proof of Theorem~\ref{thm:geometry-lower-main}]
The construction is the simplex server. Its essential feature is pathwise: after $t$ slots, no sequence of schedules can have touched more than $t$ coordinates.

Fix $d\ge2$ and let
\[
\Sset=\{e_1,\dots,e_d\}\subset\Rp^d.
\]
For this schedule set, the downward-closed capacity region is exactly
\[
\PiC=\left\{x\in\Rp^d:\sum_{i=1}^d x_i\le1\right\}.
\]
Indeed, if $x\le r$ coordinatewise for some $r\in\conv(\Sset)$, then $\sum_i x_i\le\sum_i r_i=1$. Conversely, if $x\in\Rp^d$ and $\sum_i x_i\le1$, then the vector $r$ defined by $r_1=x_1+1-\sum_i x_i$ and $r_i=x_i$ for $i\ge2$ belongs to $\conv(\Sset)$ and satisfies $x\le r$.

There is no post-decision service noise: for any schedule sequence, $S_t=s_t$ in every slot. Thus $S_{\max}=1$. Since the capacity support is evaluated only on nonnegative directions,
\[
H(q)=\max_{x\in\PiC}\ip{q}{x}=\max_{1\le i\le d}q_i,
\qquad q\in\Rp^d.
\]
Therefore
\[
\alphapi
=\inf\left\{\max_{1\le i\le d} w_i:w\in\Rp^d,\ \normtwo{w}=1\right\}
=\frac1{\sqrt d}.
\]
The lower bound follows from $1=\normtwo{w}^2\le d(\max_i w_i)^2$; equality is attained at $w=d^{-1/2}\one$.

For the prescribed $\eps\in[1/4,1/2]$, take deterministic arrivals
\[
A_t\equiv a\one,
\qquad
 a:=\frac{1-\eps}{d}.
\]
Then $\lambda_t=a\one$ and $\normtwo{A_t}=(1-\eps)/\sqrt d\le1$. Since $d^{-1}\one\in\conv(\Sset)\subset\PiC$, we have $\lambda_t=(1-\eps)d^{-1}\one\in(1-\eps)\PiC$. Thus Assumption~\ref{ass:interior} holds with slack $\eps$. Assumption~\ref{ass:bounded-primitives} also holds with $A_{\max}\le1$ and $S_{\max}=1$: the arrival vector is deterministic, $S_t=s_t$ is conditionally degenerate given $\F_t$, the identity $\E[S_t\mid\F_t]=s_t$ is exact, and the conditional independence requirement is therefore automatic.

It remains to prove the lower bound uniformly over scheduling rules. We prove the stronger deterministic claim in the theorem. Fix any sequence $s_0,\dots,s_{T_0-1}\in\Sset$, where $T_0:=\floor{d/2}$. The sequence need not be generated by a nonanticipative rule. For $0\le t\le T_0$, let
\[
N_t:=\left\{i\in\{1,\dots,d\}:s_{\tau,i}=0\text{ for every }0\le\tau<t\right\}.
\]
This is the set of coordinates that have received no service during the first $t$ slots. Each schedule $s_\tau$ is a unit vector, so in one slot at most one new coordinate can leave this set. Hence
\begin{equation}\label{eq:simplex-untouched-count}
|N_t|\ge d-t,
\qquad 0\le t\le T_0.
\end{equation}
For every $i\in N_t$, the coordinate $i$ has had zero service in all slots $0,\dots,t-1$. The queue recursion \eqref{eq:queue-recursion} then gives the exact identity
\begin{equation}\label{eq:simplex-untouched-backlog}
Q_{t,i}=ta,
\qquad i\in N_t.
\end{equation}
To verify \eqref{eq:simplex-untouched-backlog}, start from $Q_{0,i}=0$. If $i\in N_t$, then for each $0\le\tau<t$ we have $S_{\tau,i}=s_{\tau,i}=0$, and therefore $Q_{\tau+1,i}=(Q_{\tau,i}+a)^+=Q_{\tau,i}+a$; induction over $\tau$ gives $Q_{t,i}=ta$.

Combining \eqref{eq:simplex-untouched-count} and \eqref{eq:simplex-untouched-backlog}, for every $0\le t\le T_0$,
\begin{equation*}\label{eq:policy-independent-simplex-lower}
\normtwo{Q_t}^2
\ge \sum_{i\in N_t}Q_{t,i}^2
= |N_t|t^2a^2
\ge (d-t)t^2\left(\frac{1-\eps}{d}\right)^2.
\end{equation*}
Apply this at $t=T_0$. For $d\ge2$, $T_0\ge d/3$ and $d-T_0\ge d/2$. The first inequality is immediate when $d=2$, and for $d\ge3$ follows from $\floor{d/2}\ge(d-1)/2\ge d/3$. Thus
\[
\normtwo{Q_{T_0}}
\ge
\frac{1-\eps}{d}\,T_0\sqrt{d-T_0}
\ge
\frac{1-\eps}{3\sqrt2}\sqrt d
\ge
\frac{1}{6\sqrt2}\sqrt d,
\]
where the last step uses $\eps\le1/2$. Since $\alphapi=1/\sqrt d$ and $\eps\ge1/4$,
\[
\frac{1}{6\sqrt2}\sqrt d
=
\frac{\eps}{6\sqrt2}\frac{1}{\eps\alphapi}
\ge
\frac{1}{24\sqrt2}\frac{1}{\eps\alphapi}.
\]
Therefore
\[
\max_{0\le t\le T_0}\normtwo{Q_t}
\ge \normtwo{Q_{T_0}}
\ge \frac{1}{24\sqrt2}\frac{1}{\eps\alphapi}.
\]
This proves the deterministic sequence-level lower bound with $c=1/(24\sqrt2)$. A randomized nonanticipative policy produces, on each realization of its internal randomization, one admissible schedule sequence; the deterministic bound applies to that realization. Hence the lower bound holds almost surely and, by taking expectations, under every randomized nonanticipative policy. Finally, because $\alphapi=1/\sqrt d$, the horizon $T_0=\floor{d/2}$ is $\Theta(1/\alphapi^2)$ and is polynomial in $1/\alphapi$.
\end{proof}

\subsection[A slack-free square-root envelope]{Theorem~\ref{thm:sqrt-peak}: a slack-free square-root envelope}\label{app:upper-burnin}

\begin{proof}[Proof of Theorem~\ref{thm:sqrt-peak}]
Set
\[
\PeakT:=\max_{0\le s\le T}\normtwo{Q_s}.
\]
From \eqref{eq:queue-recursion} and $\normtwo{A_t-S_t}\le D$ a.s.,
\[
\normtwo{Q_{t+1}}^2
\le
\normtwo{Q_t}^2 + D^2 + 2\ip{Q_t}{A_t-S_t}.
\]
Taking conditional expectations and using $\lambda_t\in\PiC$ together with MaxWeight,
\[
\E[\ip{Q_t}{A_t-S_t}\mid \F_t]
=
\ip{Q_t}{\lambda_t} - \ip{Q_t}{s_t}
=
\ip{Q_t}{\lambda_t} - H(Q_t)
\le 0.
\]
Hence
\[
\E[\normtwo{Q_{t+1}}^2\mid \F_t]\le \normtwo{Q_t}^2 + D^2.
\]

Fix $T\ge 1$ and define, for $0\le t\le T$,
\[
Y_t:=\normtwo{Q_t}^2 + (T-t)D^2.
\]
Then $(Y_t)_{t=0}^T$ is a nonnegative supermartingale and $Y_0=TD^2$. Since $\PeakT^2\le \max_{0\le t\le T} Y_t$, Doob's maximal inequality gives
\[
\Pbb(\PeakT\ge r)
=
\Pbb(\PeakT^2\ge r^2)
\le
\Pbb\!\left(\max_{0\le t\le T}Y_t\ge r^2\right)
\le
\frac{\E[Y_0]}{r^2}
=
\frac{D^2T}{r^2}
\qquad\forall r>0.
\]
Integrating the tail,
\begin{align*}
\E[\PeakT]
&=
\int_0^\infty \Pbb(\PeakT\ge r)\,\dd r\\
&\le
\int_0^{D\sqrt{T}} 1\,\dd r
+
\int_{D\sqrt{T}}^\infty \frac{D^2T}{r^2}\,\dd r
=
2D\sqrt{T}.
\end{align*}
This proves the theorem.
\end{proof}

\section[Proofs for Section 3]{Proofs for Section \ref{sec:selfnorm}}\label{app:deferred-extensions}

\subsection{Queue-Bernstein calculus and examples}\label{app:qb-calculus}

This subsection records the elementary calculations behind Example~\ref{ex:qb-primitives}. The common inequality is
\begin{equation}\label{eq:bernstein-elementary}
e^x-1-x\le \frac{x^2}{2(1-|x|)}
\qquad\text{for } |x|<1.
\end{equation}
We use it conditionally throughout.

\begin{lemma}[Basic queue-Bernstein examples]\label{lem:qb-examples}
The primitives listed in Example~\ref{ex:qb-primitives} satisfy Definition~\ref{def:queue-bernstein} with the stated parameters.
\end{lemma}

\begin{proof}
Fix \(w\in[0,1]^d\).

\emph{Bernoulli and binomial coordinates.}
Let \(Z_i\) be a sum of conditionally independent unit Bernoulli variables, and let \(m_i=\E[Z_i\mid\mathcal G]\). Independence and \eqref{eq:bernstein-elementary} give, for \(|\theta|<1\),
\[
\log \E\!\left[
e^{\theta\ip{w}{Z-m}}\mid\mathcal G
\right]
\le
\sum_i \frac{\theta^2 w_i^2 m_i}{2(1-|\theta|)}
\le
\frac{\theta^2\ip{w}{m}}{2(1-|\theta|)}.
\]
Thus \((\nu,b)=(1,1)\).

\emph{Poisson coordinates.}
If \(Z_i\) are conditionally independent Poisson variables with means \(m_i\), then
\[
\log \E\!\left[
e^{\theta\ip{w}{Z-m}}\mid\mathcal G
\right]
=
\sum_i m_i\bigl(e^{\theta w_i}-1-\theta w_i\bigr).
\]
Using \eqref{eq:bernstein-elementary} and \(w_i^2\le w_i\) gives the same \((1,1)\) bound.

\emph{Bounded total mass.}
If \(\|Z\|_1\le L\) a.s., then \(0\le Y:=\ip{w}{Z}\le L\) and \(m_Y:=\E[Y\mid\mathcal G]=\ip{w}{m}\). For \(0\le\theta<1/L\), Lemma~\ref{lem:bounded-nonneg-mgf} gives
\[
\log \E[e^{\theta(Y-m_Y)}\mid\mathcal G]
\le
\frac{\theta^2 Lm_Y}{2(1-L\theta)}.
\]
For the lower tail, let \(\eta\in[0,1/L)\). Since \(e^{-x}\le1-x+x^2/2\) for \(x\ge0\) and \(Y^2\le LY\),
\begin{align*}
\E[e^{-\eta Y}\mid\mathcal G]
&\le 1-\eta m_Y+\frac{\eta^2Lm_Y}{2},\\
\log \E[e^{-\eta(Y-m_Y)}\mid\mathcal G]
&\le \eta m_Y+\log\!\left(1-\eta m_Y+\frac{\eta^2Lm_Y}{2}\right)\\
&\le \frac{\eta^2Lm_Y}{2}
\le \frac{\eta^2Lm_Y}{2(1-L\eta)}.
\end{align*}
Combining the two signs proves the asserted bound for \(|\theta|<1/L\). Hence \((\nu,b)=(L,L)\).

\emph{Bounded-batch compound Poisson input.}
Let \(Z=\sum_{k=1}^N B_k\), where \(N\) is conditionally Poisson and the batches are conditionally i.i.d. with \(\|B_k\|_1\le L\). Put \(Y=\ip{w}{B_1}\le L\). The compound-Poisson MGF gives
\[
\log \E[e^{\theta\ip{w}{Z-m}}\mid\mathcal G]
=
\eta\,\E\bigl[e^{\theta Y}-1-\theta Y\mid\mathcal G\bigr],
\]
where \(\eta\) is the conditional Poisson intensity. Applying \eqref{eq:bernstein-elementary} with \(|\theta|<1/L\) and using \(Y^2\le LY\) yields
\[
\log \E[e^{\theta\ip{w}{Z-m}}\mid\mathcal G]
\le
\frac{\theta^2 L\,\eta\E[Y\mid\mathcal G]}{2(1-L|\theta|)}
=
\frac{\theta^2 L\,\ip{w}{m}}{2(1-L|\theta|)}.
\]

\emph{Gamma and exponential coordinates.}
Let \(Z_i\) be conditionally independent Gamma variables with scales \(\beta_i\le\beta\), shapes \(k_i\), and means \(m_i=\beta_i k_i\). For \(|\theta|<1/\beta\),
\[
\log \E[e^{\theta w_i(Z_i-m_i)}\mid\mathcal G]
=
k_i\bigl(-\theta\beta_i w_i-\log(1-\theta\beta_i w_i)\bigr).
\]
The inequality \(-x-\log(1-x)\le x^2/[2(1-|x|)]\) for \(|x|<1\), together with \(w_i^2\le w_i\), gives
\[
\log \E[e^{\theta w_i(Z_i-m_i)}\mid\mathcal G]
\le
\frac{\theta^2\beta_i m_i w_i}{2(1-\beta|\theta|)}
\le
\frac{\theta^2\beta m_i w_i}{2(1-\beta|\theta|)}.
\]
Summing over \(i\) proves \((\nu,b)=(\beta,\beta)\). Exponential coordinates are the special case \(k_i=1\).
\end{proof}

\begin{lemma}[Elementary closure properties]\label{lem:qb-closure}
Conditionally independent queue-Bernstein conditions add: if \(Z^{(1)}\) and \(Z^{(2)}\) have parameters \((\nu_1,b_1)\) and \((\nu_2,b_2)\), then their sum is queue-Bernstein with parameters \((\nu_1+\nu_2,\max\{b_1,b_2\})\), and also with \((\nu_1+\nu_2,b_1+b_2)\). Deterministic primitives have parameters \((0,0)\).
\end{lemma}

\begin{proof}
The centered MGF of a conditionally independent sum factorizes. The denominator can be bounded using \(1-\max\{b_1,b_2\}|\theta|\le 1-b_i|\theta|\). The looser parameter \(b_1+b_2\) is sometimes more convenient and is the one used in the theorem statement.
\end{proof}

\subsection{Queue-Bernstein geometric threshold formula}\label{app:qb-entrance-scale}

Theorem~\ref{thm:qb-unbounded-main} involves a key geometric threshold $x_{\rm QB}$, whose precise formula and related properties are provided in this appendix.

When \(b_A+b_S+(\nu_A+\nu_S)/\eps>0\), set
\begin{equation}\label{eq:thetaQB-app}
  \theta_{\rm QB}:=
  c_0\left(b_A+b_S+\frac{\nu_A+\nu_S}{\eps}\right)^{-1},
\end{equation}
where \(c_0>0\) is a sufficiently small universal constant. For a candidate entrance level \(x\ge1\), define
\begin{equation*}\label{eq:rhoQB-app}
  \rho_x(z):=\min\left\{\frac{z^2}{2x},\,2z\right\},
  \qquad z\ge0.
\end{equation*}
The term \(z^2/(2x)\) is the usual Euclidean norm-expansion error for a one-slot change of length \(z\) when \(\|Q_t\|_2\ge x\). The term \(2z\) is a Lipschitz fallback for rare large changes.

For \(\theta>0\), set \(J_t:=\|A_t-S_t\|_2\) and define
\begin{align}
K_{\rm up}(\theta)
&:=
\sup_{t\ge0}\operatorname*{ess\,sup}
\E\!\left[e^{8\theta J_t}\mid\F_t\right],
\label{eq:Kup-qb}\\
\kappa_\theta(x)
&:=
\sup_{t\ge0}\operatorname*{ess\,sup}
\log\E\!\left[
\exp\!\left(2\theta\rho_x(J_t)\right)
\middle|\F_t\right],
\nonumber\\
x_{\rm curv}(\theta)
&:=
\inf\left\{x\ge1:\ \kappa_\theta(x)\le \frac{\theta\eps\alphapi}{4}\right\}.
\nonumber
\end{align}
The numerical factor \(8\) is inessential. It gives a little extra exponential moment, which is useful for the uniform integrability step below.
Finally define the queue-Bernstein entrance scale used in Theorem~\ref{thm:qb-unbounded-main} by
\begin{equation}\label{eq:xQB-app}
  x_{\rm QB}
  :=
  1+x_{\rm curv}(\theta_{\rm QB})
  +\frac{1}{\theta_{\rm QB}}\log K_{\rm up}(\theta_{\rm QB}),
  \qquad b_A+b_S+\frac{\nu_A+\nu_S}{\eps}>0.
\end{equation}
If \(b_A+b_S+(\nu_A+\nu_S)/\eps=0\), all queue-Bernstein conditions are conditionally deterministic. In that degenerate case define
\[
  D_0:=\sup_{t\ge0}\operatorname*{ess\,sup}\|A_t-S_t\|_2,
  \qquad
  x_{\rm QB}:=1+C_0\left(D_0+\frac{D_0^2}{\eps\alphapi}\right),
\]
where \(C_0\) is a sufficiently large universal constant. This deterministic convention is finite under the theorem assumptions, since \(\PiC\) is compact and \(\Sset\) is finite. If also \(\|A_t-S_t\|_2\le D\) a.s., then \(D_0\le D\), so this definition gives \(x_{\rm QB}\le C(1+D+D^2/(\eps\alphapi))\). The logarithmic term is absent; Appendix~\ref{app:qb-unbounded-proof} proves the resulting deterministic pathwise bound.

When the displayed queue-Bernstein scale is positive, \(x_{\rm QB}\) is finite under the assumptions of Theorem~\ref{thm:qb-unbounded-main}. Indeed, since \(\Sset\) is finite and \(\lambda_t\in(1-\eps)\PiC\), the conditional mean total arrival and service masses are uniformly bounded. The all-ones queue-Bernstein projection, applied with a sufficiently small universal constant \(c_0\) in \eqref{eq:thetaQB-app}, gives
\[
  \sup_t\operatorname*{ess\,sup}
  \E\!\left[e^{16\theta_{\rm QB}(\|A_t\|_1+\|S_t\|_1)}\mid\F_t\right]<\infty.
\]
Because \(J_t\le \|A_t\|_1+\|S_t\|_1\), this implies \(K_{\rm up}(\theta_{\rm QB})<\infty\). It also gives uniform integrability for \(e^{4\theta_{\rm QB}J_t}\). To see the required uniform convergence, fix \(R>0\). On \(\{J_t\le R\}\), \(\rho_x(J_t)\le R^2/(2x)\). On \(\{J_t>R\}\), \(e^{2\theta_{\rm QB}\rho_x(J_t)}\le e^{4\theta_{\rm QB} J_t}\), and the displayed exponential moment bounds the tail uniformly by \(Ce^{-4\theta_{\rm QB} R}\). Letting first \(x\to\infty\) and then \(R\to\infty\) gives \(\kappa_{\theta_{\rm QB}}(x)\downarrow0\). Hence \(x_{\rm curv}(\theta_{\rm QB})<\infty\).

If \(\|A_t-S_t\|_2\le D\) a.s., then \(\rho_x(\|A_t-S_t\|_2)\le D^2/(2x)\), so \(x_{\rm curv}(\theta_{\rm QB})\lesssim D^2/(\eps\alphapi)\), while \(\theta_{\rm QB}^{-1}\log K_{\rm up}(\theta_{\rm QB})\le 8D\). Thus
\[
  x_{\rm QB}\le C\left(1+D+\frac{D^2}{\eps\alphapi}\right).
\]
In a nontrivial bounded system \(\alphapi\le S_{\max}\le D\), so the additive \(D\) is absorbed by the displayed threshold.

\subsection[Proof of the unbounded queue-Bernstein theorem]{Proof of Theorem~\ref{thm:qb-unbounded-main}}\label{app:qb-unbounded-proof}

We prove a slightly sharper tilted statement and then choose \(\theta=\theta_{\rm QB}\).

\begin{lemma}[Tilted queue-Bernstein peak estimate]\label{lem:qb-tilted}
Under the assumptions of Theorem~\ref{thm:qb-unbounded-main}, assume \(b_A+b_S+(\nu_A+\nu_S)/\eps>0\). Let \(\theta>0\) satisfy
\begin{equation}\label{eq:theta-qb-app}
  \theta
  \le
  c\left(b_A+b_S+\frac{\nu_A+\nu_S}{\eps}\right)^{-1}
\end{equation}
for a sufficiently small universal constant \(c\). Assume in addition that \(K_{\rm up}(\theta)<\infty\) and \(x_{\rm curv}(\theta)<\infty\). Then, for every \(T\ge1\) and \(\delta\in(0,1)\), with probability at least $1-\delta$,
\begin{equation}\label{eq:qb-tilted-peak-app}
\PeakT
\lesssim
 x_{\rm curv}(\theta)
 +
 \frac1\theta
 \log\frac{(T+1)K_{\rm up}(\theta)}{\delta}.
\end{equation}
\end{lemma}

\begin{proof}
Let \(X_t:=\|Q_t\|_2\) and, on \(\{X_t>0\}\), set \(u_t:=Q_t/X_t\). On \(\{X_t=0\}\), choose any nonnegative unit vector. Write
\[
\xi_t:=A_t-S_t,
\qquad
J_t:=\|\xi_t\|_2,
\qquad
Z_t:=\ip{u_t}{A_t}+\ip{u_t}{s_t-S_t}.
\]
Since projection onto \(\Rp^d\) is nonexpansive, Lemma~\ref{lem:norm-inc} and the triangle inequality imply
\begin{equation}\label{eq:qB-norm-expansion}
X_{t+1}-X_t
\le
Z_t-H(u_t)+\rho_{X_t}(J_t)
\qquad\text{on }\{X_t>0\}.
\end{equation}

Let \(b_*:=b_A\vee b_S\). Choose \(c\) in \eqref{eq:theta-qb-app} so that
\[
2\theta b_*\le \frac12,
\qquad
\frac{\theta(\nu_A+\nu_S)}{1-2\theta b_*}\le \frac{\eps}{4}.
\]
Conditional independence and the queue-Bernstein bounds, applied to \(A_t\) with tilt \(2\theta\) and to \(S_t\) with tilt \(-2\theta\), give
\begin{align}
\frac12\log\E[e^{2\theta Z_t}\mid\F_t]
&\le
\theta\ip{u_t}{\lambda_t}
+
\frac{\theta^2(\nu_A\ip{u_t}{\lambda_t}+\nu_S\ip{u_t}{s_t})}{1-2\theta b_*}.
\label{eq:qB-dir-proof}
\end{align}
By Assumption~\ref{ass:interior} and MaxWeight,
\[
\ip{u_t}{\lambda_t}\le(1-\eps)H(u_t),
\qquad
\ip{u_t}{s_t}=H(u_t).
\]
Substituting into \eqref{eq:qB-dir-proof} yields
\begin{equation}\label{eq:qB-dir-final}
\frac12\log\E[e^{2\theta Z_t}\mid\F_t]
\le
\theta\left(1-\frac{3\eps}{4}\right)H(u_t).
\end{equation}

Let \(x_0:=2x_{\rm curv}(\theta)+1\). The map \(x\mapsto \rho_x(z)\) is nonincreasing for every fixed \(z\), hence \(x\mapsto\kappa_\theta(x)\) is nonincreasing. Therefore the feasible set in the definition of \(x_{\rm curv}(\theta)\) is upward closed. Since \(x_0>x_{\rm curv}(\theta)\) and \(x_{\rm curv}(\theta)<\infty\), the definition of the infimum gives some feasible \(y<x_0\), and upward closure then makes \(x_0\) feasible. Thus
\begin{equation}\label{eq:qB-curv-proof}
\frac12\log\E\!\left[e^{2\theta\rho_{x_0}(J_t)}\mid\F_t\right]
\le
\frac{\theta\eps\alphapi}{8}.
\end{equation}
On \(\{X_t\ge x_0\}\), combining \eqref{eq:qB-norm-expansion}, conditional Cauchy--Schwarz, \eqref{eq:qB-dir-final}, and \eqref{eq:qB-curv-proof} gives
\begin{align*}
\log\E[e^{\theta(X_{t+1}-X_t)}\mid\F_t]
&\le
-\theta H(u_t)
+\frac12\log\E[e^{2\theta Z_t}\mid\F_t]
+\frac12\log\E[e^{2\theta\rho_{x_0}(J_t)}\mid\F_t]\\
&\le
-\frac{3\theta\eps}{4}H(u_t)+\frac{\theta\eps\alphapi}{8}
\le
-\frac{5\theta\eps\alphapi}{8},
\end{align*}
where the last step uses \(H(u_t)\ge\alphapi\). Thus \(X_t\) satisfies the exponential drift condition \eqref{eq:ED} above threshold \(x_0\) at tilt \(\theta\).

Moreover, \((X_{t+1}-X_t)^+\le J_t\), so the upward-increment envelope \eqref{eq:UP} holds with \(M_+(\theta)\le K_{\rm up}(\theta)\). Since \(X_0=0\le x_0\), Lemma~\ref{lm:exponential-peak} gives
\[
\Pbb\!\left(\PeakT\ge x_0+u\right)
\le
(T+1)K_{\rm up}(\theta)e^{-\theta u}.
\]
Taking \(u=\theta^{-1}\log((T+1)K_{\rm up}(\theta)/\delta)\) proves \eqref{eq:qb-tilted-peak-app}.
\end{proof}

\begin{proof}[Proof of Theorem~\ref{thm:qb-unbounded-main}]
If \(b_A+b_S+(\nu_A+\nu_S)/\eps>0\), the finiteness argument in Appendix~\ref{app:qb-entrance-scale} gives \(K_{\rm up}(\theta_{\rm QB})<\infty\) and \(x_{\rm curv}(\theta_{\rm QB})<\infty\). Thus Lemma~\ref{lem:qb-tilted} applies with \(\theta=\theta_{\rm QB}\). By definition of \(x_{\rm QB}\),
\[
x_{\rm curv}(\theta_{\rm QB})
+
\frac1{\theta_{\rm QB}}\log K_{\rm up}(\theta_{\rm QB})
\le x_{\rm QB},
\qquad
\frac1{\theta_{\rm QB}}\lesssim b_A+b_S+\frac{\nu_A+\nu_S}{\eps}.
\]
This gives the displayed bound in Theorem~\ref{thm:qb-unbounded-main}. The bounded-jump reduction follows from the last paragraph of Appendix~\ref{app:qb-entrance-scale}.

It remains to handle the degenerate case
\(b_A+b_S+(\nu_A+\nu_S)/\eps=0\). The queue-Bernstein inequalities then have zero right-hand side. For each coordinate vector \(e_i\), Jensen's inequality gives
\[
  1
  \le
  \E\!\left[e^{\theta(A_i(t)-\lambda_i(t))}\mid\F_t\right]
  \le 1,
  \qquad \theta>0,
\]
and the equality case for the strictly convex exponential implies \(A_i(t)=\lambda_i(t)\) a.s. Applying the same argument to \(S_t\) gives \(S_t=s_t\) a.s. Thus the primitives are conditionally deterministic.

Let \(X_t:=\|Q_t\|_2\) and let \(\xi_t:=A_t-S_t\). If \(D_0=0\), then \(\xi_t=0\) a.s. for every \(t\), and the path stays at the origin. Assume therefore that \(D_0>0\), and put
\[
  a:=\frac{D_0^2}{\eps\alphapi}.
\]
If \(X_t>0\), set \(u_t=Q_t/X_t\). Lemma~\ref{lem:norm-inc}, Assumption~\ref{ass:interior}, and MaxWeight give
\[
  X_{t+1}-X_t
  \le
  \ip{u_t}{\xi_t}+\frac{\|\xi_t\|_2^2}{2X_t}
  \le
  -\eps H(u_t)+\frac{D_0^2}{2X_t}.
\]
Consequently, whenever \(X_t\ge a\), the last display is at most \(-\eps\alphapi/2\), so \(X_{t+1}\le X_t\). If \(X_t<a\), nonexpansiveness of projection gives \(X_{t+1}\le X_t+D_0<a+D_0\). Since \(X_0=0\), induction yields the pathwise bound
\[
  \PeakT\le a+D_0
  \le x_{\rm QB}
\]
after increasing \(C_0\), if necessary. This proves the theorem in the deterministic degenerate case, with no logarithmic term.
\end{proof}

\subsection{One-dimensional queue-Bernstein peaks and Kingman's scale}\label{app:one-dimensional-qb}

The one-dimensional recursion admits a sharper argument than Theorem~\ref{thm:qb-unbounded-main}: there is no Euclidean curvature term, and the reflected workload is the running maximum of net-input partial sums.

\begin{theorem}[One-dimensional queue-Bernstein peak]\label{thm:one-dimensional-qb}
Let
\[
Q_{t+1}=(Q_t+A_t-S_t)^+,
\qquad Q_0=0,
\]
where, conditional on \(\F_t\), \(A_t\) and \(S_t\) are independent nonnegative random variables. Let \(\lambda_t=\E[A_t\mid\F_t]\) and \(s_t=\E[S_t\mid\F_t]\), and assume \(\lambda_t\le(1-\eps)s_t\) a.s. for every \(t\). Suppose \(A_t\) and \(S_t\) satisfy the one-dimensional queue-Bernstein condition with parameters \((\nu_A,b_A)\) and \((\nu_S,b_S)\), uniformly over \(t\). Then there is a universal constant \(C\) such that, for every \(T\ge1\) and \(\delta\in(0,1)\),
\[
\Pbb\!\left(
\max_{0\le t\le T}Q_t
\le
C\left(b_A+b_S+\frac{\nu_A+\nu_S}{\eps}\right)
\log\frac{T+1}{\delta}
\right)
\ge1-\delta.
\]
\end{theorem}

\begin{proof}
Let
\[
K:=b_A+b_S+\frac{\nu_A+\nu_S}{\eps}.
\]
If $K=0$, then $b_A=b_S=\nu_A=\nu_S=0$. The queue-Bernstein inequalities have zero right-hand side. Jensen's inequality, applied to $e^{\theta(A_t-\lambda_t)}$ and $e^{\theta(S_t-s_t)}$ for $\theta>0$, forces $A_t=\lambda_t$ and $S_t=s_t$ a.s. conditionally on $\F_t$. Since $\lambda_t\le(1-\eps)s_t\le s_t$, the net input is nonpositive in every slot. Starting from $Q_0=0$, the queue remains identically zero, and the stated bound holds with its zero right-hand side. We may therefore assume $K>0$.
Write \(\xi_t:=A_t-S_t\). The Lindley representation gives the pathwise identity
\[
Q_t=\max_{0\le k\le t}\left(\sum_{r=k}^{t-1}\xi_r\right)^+ .
\]
Fix \(\theta>0\) satisfying
\[
\theta\le c\left(b_A+b_S+\frac{\nu_A+\nu_S}{\eps}\right)^{-1}
\]
with \(c>0\) sufficiently small. The queue-Bernstein bounds and conditional independence give
\begin{align*}
\log \E[e^{\theta \xi_t}\mid\F_t]
&\le
\theta(\lambda_t-s_t)
+\frac{\theta^2\nu_A\lambda_t}{2(1-b_A\theta)}
+\frac{\theta^2\nu_Ss_t}{2(1-b_S\theta)} \\
&\le
-\theta\eps s_t
+\frac{\theta^2(\nu_A+\nu_S)s_t}{2(1-(b_A+b_S)\theta)}
\le 0 .
\end{align*}
Here we used \(\lambda_t\le s_t\) and the choice of \(c\) in the last step. Hence, for each deterministic starting time \(k\), the process
\[
\exp\!\left(\theta\sum_{r=k}^{m-1}\xi_r\right),
\qquad m\ge k,
\]
is a nonnegative supermartingale. Doob's maximal inequality yields
\[
\Pbb\!\left(\max_{k\le m\le T}\sum_{r=k}^{m-1}\xi_r\ge u\right)
\le e^{-\theta u}.
\]
A union bound over \(k=0,1,\ldots,T\), together with the Lindley representation, gives
\[
\Pbb\!\left(\max_{0\le t\le T}Q_t\ge u\right)
\le (T+1)e^{-\theta u}.
\]
Taking \(u=\theta^{-1}\log((T+1)/\delta)\) proves the result.
\end{proof}

For i.i.d. primitives with means \(\E A=\rho\tau\), \(\E S=\tau\), and \(\rho=1-\eps\), a variance-to-mean queue-Bernstein scale has the same order as
\[
\nu_A+\nu_S
\asymp
\frac{\operatorname{Var}(A)}{\E A}
+\frac{\operatorname{Var}(S)}{\E S}
=
\rho c_a^2\tau+c_s^2\tau.
\]
Thus the dominant logarithmic coefficient in Theorem~\ref{thm:one-dimensional-qb} is of order
\[
\frac{\nu_A+\nu_S}{\eps}
\asymp
\frac{\tau(\rho c_a^2+c_s^2)}{1-\rho},
\]
which matches Kingman's scale up to universal constants when \(\rho=\Omega(1)\). The theorem is a finite-horizon high-probability peak statement, whereas Kingman's formula is a constant-level approximation for a steady-state mean.

\section[Proofs for Section 4]{Proofs for Section~\ref{sec:geometry}}\label{app:deferred-geometry}
\subsection[The two-queue collapse example]{Proposition~\ref{prop:twoqueue-peak}: the two-queue collapse example}\label{app:prep2queue}

\begin{proof}[Proof of Lemma~\ref{lem:twoqueue-ssc}]
The service is deterministic in Example~\ref{ex:two-queue}, so the realized service in each slot equals the selected vector.
Write
\[
G_t:=Q_1(t)-aQ_2(t).
\]
We first prove the collapse estimate $G_t^+\le1+a^2$ by induction. It holds at $t=0$. Suppose it holds at time $t$.

If MaxWeight serves queue~$1$, then $G_t\ge0$. The service vector is $(1,0)$, so
\[
Q_1(t+1)=(Q_1(t)+A_1(t)-1)^+,
\qquad
Q_2(t+1)=Q_2(t)+A_2(t).
\]
There are two cases. If $Q_1(t)+A_1(t)\ge1$, then one full unit of service is used in the first coordinate, and
\[
G_{t+1}
= G_t-1+A_1(t)-aA_2(t)
\le G_t
\le 1+a^2,
\]
where the last inequality uses $G_t=G_t^+$ and the induction hypothesis. If instead $Q_1(t)+A_1(t)<1$, then $Q_1(t)=A_1(t)=0$. Since queue~$1$ was selected, $G_t\ge0$, hence $Q_2(t)=0$, and therefore
\[
G_{t+1}=-aA_2(t)\le0.
\]
Thus service of queue~$1$ preserves the induction invariant $G_{t+1}^+\le1+a^2$.

If MaxWeight serves queue~$2$, then $G_t\le0$. The service vector is $(0,a)$, so
\[
Q_1(t+1)=Q_1(t)+A_1(t),
\qquad
Q_2(t+1)=(Q_2(t)+A_2(t)-a)^+.
\]
If $Q_2(t)+A_2(t)\ge a$, then the second coordinate receives the full service amount $a$, and
\[
G_{t+1}
= G_t+a^2+A_1(t)-aA_2(t)
\le1+a^2,
\]
because $G_t\le0$, $A_1(t)\le1$, and $A_2(t)\ge0$. If instead $Q_2(t)+A_2(t)<a$, then $Q_2(t)<a$. Since $G_t\le0$, we have $Q_1(t)\le aQ_2(t)<a^2\le1$. The first queue is integer-valued, so $Q_1(t)=0$, and hence
\[
G_{t+1}=A_1(t)\le1\le1+a^2.
\]
Thus service of queue~$2$ also preserves the induction invariant $G_{t+1}^+\le1+a^2$. The induction closes and proves the first claim.

We next prove the local drift statement. If MaxWeight serves queue~$1$, then $Q_1(t)\ge aQ_2(t)$. Since $Y_t\ge2$, this forces $Q_1(t)>0$; otherwise $Q_1(t)=0$ would imply $Q_2(t)=0$ and hence $Y_t=0$. Thus $Q_1(t)\ge1$, so the projection is inactive in the first coordinate and
\[
Y_{t+1}=a(Q_1(t)+A_1(t)-1)+Q_2(t)+A_2(t)
=Y_t-a+aA_1(t)+A_2(t).
\]

If MaxWeight serves queue~$2$, then $Q_1(t)\le aQ_2(t)$, so
\[
Y_t=aQ_1(t)+Q_2(t)\le(1+a^2)Q_2(t).
\]
Thus $Y_t\ge2$ implies
\[
Q_2(t)\ge \frac{2}{1+a^2}\ge a.
\]
The projection is therefore inactive in the second coordinate, and again
\[
Y_{t+1}=a(Q_1(t)+A_1(t))+Q_2(t)+A_2(t)-a
=Y_t-a+aA_1(t)+A_2(t).
\]
Therefore, on $\{Y_t\ge2\}$,
\[
Y_{t+1}=Y_t-a+aA_1(t)+A_2(t).
\]
Taking conditional expectations and using \eqref{eq:twoqueue-slack} gives
\[
\E[Y_{t+1}-Y_t\mid\F_t]\le -a\eps.
\]
\end{proof}

\begin{proof}[Proof of Proposition~\ref{prop:twoqueue-peak}]
Let
\[
W_t:=aA_1(t)+A_2(t).
\]
On $\{Y_t\ge2\}$, Lemma~\ref{lem:twoqueue-ssc} gives
\[
Y_{t+1}-Y_t=W_t-a.
\]
Moreover, by \eqref{eq:twoqueue-slack},
\[
m_t:=\E[W_t\mid\F_t]
=a\,\E[A_1(t)\mid\F_t]+\E[A_2(t)\mid\F_t]
\le a(1-\eps).
\]
Since $0\le W_t\le1+a\le2$ a.s., convexity of $x\mapsto e^{\theta x}$ on $[0,2]$ gives
\[
e^{\theta x}\le1+\frac{x}{2}(e^{2\theta}-1),
\qquad 0\le x\le2.
\]
Thus
\[
\E[e^{\theta W_t}\mid\F_t]
\le1+\frac{m_t}{2}(e^{2\theta}-1).
\]
Choose $\theta:=\eps/4\le1/4$. Using $\log(1+z)\le z$ and $e^{2\theta}-1\le2\theta+4\theta^2$ for $\theta\in[0,1/4]$, we obtain on $\{Y_t\ge2\}$
\begin{align*}
\log\E[e^{\theta(Y_{t+1}-Y_t)}\mid\F_t]
&=-a\theta+\log\E[e^{\theta W_t}\mid\F_t] \\
&\le -a\theta+\frac{m_t}{2}(e^{2\theta}-1) \\
&\le -a\theta+a(1-\eps)(\theta+2\theta^2) \\
&\le -a\eps\theta+2a\theta^2
= -\frac{a\eps^2}{8}.
\end{align*}
Hence $Y_t$ satisfies exponential drift above threshold $2$ at parameter $\theta$. Also $(Y_{t+1}-Y_t)^+\le W_t\le2$, so the upward MGF condition holds with $M_+(\theta)\le e^{2\theta}$. Since $Y_0=0\le2$, Lemma~\ref{lm:exponential-peak} gives
\[
\max_{0\le t\le T}Y_t
\;\lesssim\;
1+\frac1\eps\log\frac{T+1}{\delta}
\qquad\text{with probability at least }1-\delta.
\]

To convert back to queue length, note that $Q_2(t)\le Y_t$ and the collapse bound gives
\[
Q_1(t)\le aQ_2(t)+1+a^2\le Y_t+2.
\]
Thus
\[
\normtwo{Q_t}\le\normone{Q_t}=Q_1(t)+Q_2(t)\le2Y_t+2,
\]
which proves the proposition.
\end{proof}

\subsection[Finite-time state-space collapse under CRP]{Theorem~\ref{thm:finite_horizon_ssc_crp}: finite-time state-space collapse under CRP}\label{app:ssc}
\begin{proof}[Proof of Theorem~\ref{thm:finite_horizon_ssc_crp}]
The proof is a finite-horizon version of the geometric
SSC argument for MaxWeight, whose heavy-traffic form goes
back to \citet*{Stolyar04}.  We use the Lyapunov drift formulation of
this geometry, as in \citet*{EryilmazSrikant12}: the face
margin gives a feasible comparison point in the transverse direction, which
yields a MaxWeight gain proportional to \(\|Q_t^\perp\|_2\), while subtracting
the bottleneck projection removes the ordinary \(v\)-direction drift.  The
new features here are that the fixed nominal load is replaced by the
predictable, time-varying face point
\(\rcirc_t=\lambda_t/(1-\eps)\), and that the resulting scalar drift bound is
converted to a finite-horizon peak estimate by a nonasymptotic drift-to-peak
lemma.

Let
\[
P_\perp z:=z-\ip{v}{z}v,
\qquad
X_t:=\normtwo{Q_t^\perp}=\normtwo{P_\perp Q_t},
\qquad
Y_t:=\ip{v}{Q_t}.
\]
We prove a negative drift for $X_t$ above a deterministic threshold and then
apply Lemma~\ref{lm:Bernstein-peak}. If $v^\perp=\{0\}$, then
$Q_t^\perp=0$ for all $t$ and the theorem is immediate; hence we assume below
that $v^\perp$ contains a unit vector. In this case the face margin is finite
and satisfies
\begin{equation}\label{eq:face-margin-D-bound}
  \delta_{\mathrm{face}}\le 2S_{\max}\le 2D.
\end{equation}
Indeed, $\PiC\subseteq\{x\in\Rp^d:\normtwo{x}\le S_{\max}\}$: if
$x\in\PiC$, then $0\le x\le r$ coordinatewise for some
$r\in\conv(\Sset)$, so $\normtwo{x}\le\normtwo r\le S_{\max}$. Fix any time
$t$ and work on a full-probability event on which Assumption~\ref{ass:crp}
holds for $\rcirc_t$. For any $\rho<\delta_{\mathrm{face}}$ and any unit vector
$e\perp v$, the point $\rcirc_t+\rho e$ lies in $\Hface$ and hence, by the
uniform relative-inradius condition, in $F^\star\subseteq\PiC$; also
$\rcirc_t\in F^\star\subseteq\PiC$. Hence
\[
\rho=\normtwo{(\rcirc_t+\rho e)-\rcirc_t}\le2S_{\max}.
\]
Letting $\rho\uparrow\delta_{\mathrm{face}}$ gives
\eqref{eq:face-margin-D-bound}.

Fix $t$ and work on the event $X_t>0$. Put $e_t:=Q_t^\perp/X_t$ and set $\rho_*:=3\delta_{\mathrm{face}}/4$.
Since $e_t\perp v$, the point $\rcirc_t+\rho_*e_t$ belongs to $\Hface$ and,
because $\rho_*<\delta_{\mathrm{face}}$ and the fixed-face margin assumption
is monotone in the radius, to $F^\star\subseteq\PiC$. Because MaxWeight
maximizes the support function over $\PiC$ for nonnegative queue weights,
\begin{equation*}
\ip{Q_t}{s_t}=H(Q_t)
\ge
\ip{Q_t}{\rcirc_t+\rho_*e_t}
=
\ip{Q_t}{\rcirc_t}+\rho_*X_t.
\end{equation*}
Since $\rcirc_t\in F^\star$, we have $\ip{v}{\rcirc_t}=\mu$. Using \eqref{eq:nominal-load-def} and $Q_t=Y_t v+Q_t^\perp$, we obtain
\begin{align}
\ip{Q_t}{\lambda_t-s_t}
&\le
-\eps\ip{Q_t}{\rcirc_t}-\rho_*X_t  \nonumber\\
&=
-\eps\mu Y_t-\eps\ip{Q_t^\perp}{\rcirc_t}-\rho_*X_t  \nonumber\\
&\le
-\eps\mu Y_t-\bigl(\rho_*-\eps\normtwo{\rcirc_t}\bigr)X_t  \nonumber\\
&\le
-\eps\mu Y_t-\frac{\delta_{\mathrm{face}}}{4}X_t
\nonumber
\end{align}
where the last step uses $\eps\le \delta_{\mathrm{face}}/(2S_{\max}) \le  \delta_{\mathrm{face}}/(2\normtwo{\rcirc_t})$ and $\rho_*=3\delta_{\mathrm{face}}/4$.

We next compare the drifts of the full quadratic Lyapunov function and its
projection onto the bottleneck ray. Since coordinatewise projection onto
$\Rp^d$ can only reduce squared norm,
\begin{align}
\E[\normtwo{Q_{t+1}}^2-\normtwo{Q_t}^2\mid\F_t]
&\le
2\ip{Q_t}{\lambda_t-s_t}
+\E[\normtwo{A_t-S_t}^2\mid\F_t]  \nonumber\\
&\le
-2\eps\mu Y_t-\frac{\delta_{\mathrm{face}}}{2}X_t+D^2.
\label{eq:ssc-total-square}
\end{align}
For the parallel component, use $v\ge0$ to get
\[
Y_{t+1}=\ip{v}{(Q_t+A_t-S_t)^+}
\ge
\bigl(Y_t+\ip{v}{A_t-S_t}\bigr)^+.
\]
For every $y\ge0$ and every real $z$,
$((y+z)^+)^2-y^2\ge2yz$. Therefore
\begin{align}
\E[Y_{t+1}^2-Y_t^2\mid\F_t]
&\ge
2Y_t\,\ip{v}{\lambda_t-s_t}  \nonumber\\
&\ge
-2\eps\mu Y_t,
\label{eq:ssc-parallel-square}
\end{align}
because $\ip{v}{s_t}\le H(v)=\mu$ and
$\ip{v}{\lambda_t}
=
(1-\eps)\mu$.
Since $\normtwo{Q_t^\perp}^2=\normtwo{Q_t}^2-Y_t^2$, subtracting \eqref{eq:ssc-parallel-square} from \eqref{eq:ssc-total-square} gives
\begin{equation*}
\E[X_{t+1}^2-X_t^2\mid\F_t]
\le
-\frac{\delta_{\mathrm{face}}}{2}X_t+D^2.
\end{equation*}
For $X_t>0$,
\[
X_{t+1}-X_t
\le
\frac{X_{t+1}^2-X_t^2}{2X_t}.
\]
Consequently,
\[
\E[X_{t+1}-X_t\mid\F_t]
\le
-\frac{\delta_{\mathrm{face}}}{4}+\frac{D^2}{2X_t}.
\]
Thus, whenever $X_t\ge x_0:=4D^2/\delta_{\mathrm{face}}$,
\begin{equation*}
\E[X_{t+1}-X_t\mid\F_t]
\le
-\frac{\delta_{\mathrm{face}}}{8}.
\end{equation*}

Finally, $|X_{t+1}-X_t|\le\normtwo{Q_{t+1}-Q_t}\le D$ a.s. Hence the centered increments satisfy the conditional Bernstein condition with variance proxy $D^2$ and $b=0$, and the upward MGF condition holds with $M_+(\theta)\le e^{\theta D}$. Applying Lemma~\ref{lm:Bernstein-peak} with drift parameter $\Delta=\delta_{\mathrm{face}}/8$ yields, for a universal constant $C$,
\[
\max_{0\le t\le T}X_t
\le
x_0+D+C\frac{D^2}{\delta_{\mathrm{face}}}\log\frac{e(T+1)}{\delta}
\]
with probability at least $1-\delta$. It remains only to absorb the two nonlogarithmic terms into the stated scale. We have $x_0=4D^2/\delta_{\mathrm{face}}$, and \eqref{eq:face-margin-D-bound} gives $D\le2D^2/\delta_{\mathrm{face}}$. As $\log(e(T+1)/\delta)\ge1$, increasing the universal constant gives \eqref{eq:ssc-Rperp-def}.
\end{proof}

\subsection[Collapse-to-peak conversion near a ray]{Collapse-to-peak conversion near a ray}\label{app:collapse-to-peak}

The proposition below isolates the only place where the peak proof uses collapse. No CRP geometry appears in the statement. One supplies a deterministic high-probability tube around a fixed ray; the proposition converts that tube into a one-dimensional peak estimate plus a deterministic residual term.

Fix $v\in\Rp^d$ with $\normtwo{v}=1$ and set $\mu:=H(v)>0$. Define
\[
Y_t:=\ip{v}{Q_t},
\qquad
Q_t^\parallel:=Y_tv,
\qquad
Q_t^\perp:=Q_t-Q_t^\parallel .
\]
For this fixed direction, a deterministic number $R_\perp(T,\eta)$ is a finite-horizon \emph{single-bottleneck collapse radius} if
\begin{equation}\label{eq:collapse-radius-def}
\Pbb\!\left(
\max_{0\le t\le T}\normtwo{Q_t^\perp}>R_\perp(T,\eta)
\right)
\le \eta .
\end{equation}

\begin{proposition}[Peak law under finite-horizon single-bottleneck collapse]\label{prop:collapse-to-peak}
Consider MaxWeight scheduling under Assumptions~\ref{ass:interior} and~\ref{ass:bounded-primitives}. Let $D:=A_{\max}+S_{\max}$. Fix $v\in\Rp^d$ with $\normtwo{v}=1$, set $\mu:=H(v)>0$, and use the projection decomposition above for this $v$. Fix $T\ge1$ and $\delta\in(0,1)$. Suppose $R_\perp(T,\delta/2)$ is a finite deterministic single-bottleneck collapse radius, i.e., \eqref{eq:collapse-radius-def} holds with $\eta=\delta/2$.
Define
\[
\Delta_{\mathrm{gap}}
:=
\min\{\mu-\ip{v}{s}:s\in\Sset,\ \ip{v}{s}<\mu\},
\]
with the convention $\Delta_{\mathrm{gap}}=\infty$ if every schedule is $v$-optimal, and interpret $a/\infty:=0$ for finite $a$. Let
\[
v_{\min}:=\min\{v_i:v_i>0\}
\]
and set
\[
x_{\mathrm{loc}}(T,\delta)
:=
\max\left\{
\frac{4S_{\max}R_\perp(T,\delta/2)}{\Delta_{\mathrm{gap}}},
\frac{S_{\max}+R_\perp(T,\delta/2)}{v_{\min}}
\right\}.
\]
Then, with probability at least $1-\delta$, both bounds hold:
\[
\PeakT
\;\lesssim\;
\frac{D}{\eps}\log\frac{T+1}{\delta}
+
x_{\mathrm{loc}}(T,\delta)
+
R_\perp(T,\delta/2),
\]
and
\[
\max_{0\le t\le T}\normone{Q_t}
\;\lesssim\;
\frac{\normone{v}\,D}{\eps}\log\frac{T+1}{\delta}
+
\normone{v}\,x_{\mathrm{loc}}(T,\delta)
+
\sqrt d\,R_\perp(T,\delta/2).
\]
If service is deterministic, $S_t=s_t$ a.s., then the same two bounds hold with $D$ replaced by $A_{\max}$ in the logarithmic coefficient.
\end{proposition}

\begin{proof}
Set
\[
R:=R_\perp(T,\delta/2),
\qquad
\tau_\perp:=\inf\{t\ge0:\normtwo{Q_t^\perp}>R\}\wedge(T+1).
\]
The collapse estimate \eqref{eq:collapse-radius-def} gives $\Pbb(\tau_\perp\le T)\le\delta/2$. Let
\[
\widetilde Y_t:=Y_t\one_{\{t<\tau_\perp\}}
\]
be the stopped bottleneck projection.

We first use the collapse event deterministically. On $\{t<\tau_\perp\}$,
\[
Q_t=Y_t v+R_t,
\qquad
R_t:=Q_t^\perp,
\qquad
\normtwo{R_t}\le R .
\]
Let $s^\star\in\argmax_{s\in\Sset}\ip{v}{s}$, so $\ip{v}{s^\star}=\mu$. If every schedule is $v$-optimal, the schedule-gap condition below is void. Otherwise, for any $v$-suboptimal schedule $s$,
\[
\ip{Q_t}{s^\star-s}
=
Y_t\bigl(\mu-\ip{v}{s}\bigr)+\ip{R_t}{s^\star-s}
\ge
Y_t\Delta_{\mathrm{gap}}-2S_{\max}R .
\]
Thus the first term in $x_{\mathrm{loc}}$ rules out $v$-suboptimal MaxWeight maximizers. If $R=0$, the no-waste threshold below implies $Y_t>0$, so the inequality is still strict whenever a suboptimal schedule exists.

The second term in $x_{\mathrm{loc}}$ prevents bottleneck-relevant service from being lost. If
\[
Y_t\ge \frac{S_{\max}+R}{v_{\min}},
\]
then every coordinate with $v_i>0$ satisfies
\[
Q_{t,i}=Y_tv_i+R_{t,i}
\ge S_{\max}+R-|R_{t,i}|
\ge S_{\max}.
\]
Since $\normtwo{S_t}\le S_{\max}$, every coordinate of the realized service vector is at most $S_{\max}$. The positive-part map is therefore inactive in all coordinates with $v_i>0$. Coordinates with $v_i=0$ do not affect $Y_t$. Hence, on
\[
\{t<\tau_\perp,\,Y_t\ge x_{\mathrm{loc}}(T,\delta)\},
\]
we have the exact projection recursion
\begin{equation}\label{eq:local-projection-recursion}
Y_{t+1}=Y_t+\ip{v}{A_t-S_t}.
\end{equation}

On $\{\widetilde Y_t\ge x_{\mathrm{loc}}(T,\delta)\}$, the preceding recursion applies before stopping; if the stopping time occurs at $t+1$, then $\widetilde Y_{t+1}=0\le Y_{t+1}$. Therefore
\[
\widetilde Y_{t+1}-\widetilde Y_t
\le
\ip{v}{A_t}+\ip{v}{s_t-S_t}-\mu,
\]
because the selected MaxWeight schedule is $v$-optimal on this event. For Lemma~\ref{lem:netinput-mgf}, the conditional mean parameters satisfy
\[
 m_t(v)=\ip{v}{\lambda_t}\le(1-\eps)\mu,
 \qquad
 g_t(v)=\ip{v}{s_t}=\mu.
\]
The first inequality is Assumption~\ref{ass:interior}; the second is $v$-optimality. Thus, for every $\theta\in[0,1/D)$,
\[
\begin{aligned}
\log \E\!\left[e^{\theta(\widetilde Y_{t+1}-\widetilde Y_t)}\mid\F_t\right]
&\le
-\theta\mu+\theta m_t(v)+\frac{\theta^2D\mu}{2(1-\theta D)} \\
&\le
-\theta\eps\mu+\frac{\theta^2D\mu}{2(1-\theta D)} .
\end{aligned}
\]
With $\theta:=\eps/(4D)$, the last display is at most $-c\theta\eps\mu$ for a universal constant $c>0$. The upward increments obey
\[
(\widetilde Y_{t+1}-\widetilde Y_t)^+
\le
\ip{v}{A_t}+\ip{v}{s_t}
\le
A_{\max}+S_{\max}=D,
\]
where $v\ge0$ and $\normtwo{v}=1$ were used. Since $\widetilde Y_0=0$, Lemma~\ref{lm:exponential-peak} gives
\[
\max_{0\le t\le T}\widetilde Y_t
\lesssim
x_{\mathrm{loc}}(T,\delta)
+
\frac{D}{\eps}\log\frac{T+1}{\delta}
\]
with probability at least $1-\delta/2$.

If $S_t=s_t$ a.s., the service-noise term vanishes. The deterministic-service part of Lemma~\ref{lem:netinput-mgf} and the upward envelope $(\widetilde Y_{t+1}-\widetilde Y_t)^+\le\ip{v}{A_t}\le A_{\max}$ give the same estimate with $D$ replaced by $A_{\max}$ in the logarithmic coefficient; if $A_{\max}=0$, the projection is nonincreasing above the entrance level and the conclusion is immediate.

On $\{\tau_\perp>T\}$, stopped and unstopped projections agree for all $t\le T$. Since $\normtwo{v}=1$ and $Y_t\ge0$,
\[
\normtwo{Q_t}
\le
\normtwo{Q_t^\parallel}+\normtwo{Q_t^\perp}
\le
Y_t+R .
\]
Also,
\[
\normone{Q_t}
\le
\normone{Q_t^\parallel}+\normone{Q_t^\perp}
\le
\normone{v}\,Y_t+
\sqrt d\,R .
\]
These two deterministic conversions, together with a union bound using $\Pbb(\tau_\perp\le T)\le\delta/2$, prove both displayed bounds.
\end{proof}

\subsection[Local peak law under CRP]{Theorem~\ref{thm:crp-main}: local peak law under CRP}\label{app:locallaw}
\begin{proof}[Proof of Theorem~\ref{thm:crp-main}]
Theorem~\ref{thm:finite_horizon_ssc_crp} supplies exactly the tube required by Proposition~\ref{prop:collapse-to-peak}: the function $R_\perp(T,\eta)$ in \eqref{eq:ssc-Rperp-def} is a deterministic single-bottleneck collapse radius for the CRP normal $v$. Applying the proposition with this radius gives the threshold $x_{\mathrm{loc}}(T,\delta)$ appearing in Theorem~\ref{thm:crp-main} and yields both displayed estimates. The deterministic-service improvement in the logarithmic coefficient is inherited from the same proposition.
\end{proof}

\section[Proofs for Section 5]{Proofs for Section~\ref{sec:iqs}}\label{app:deferred-iqs}
\subsection[Generalized IQS upper bound]{Theorem~\ref{thm:iqs-upper-main}: generalized IQS upper bound}\label{app:iqs-general-upper}
\begin{proof}[Proof of Proposition~\ref{prop:iqs-alpha2}]
Every permutation matrix has exactly $n$ ones, so its Frobenius norm is $\sqrt n$. The inequality $\normone{Q}\le n\normtwo{Q}$ is Cauchy--Schwarz in dimension $n^2$.

For the lower bound on $\alphapi$, let $W\in\Rp^{n\times n}$ satisfy $\norm{W}_F=1$, and let $\sigma$ be a uniformly random permutation. Then
\[
Z:=\sum_{i=1}^n W_{i,\sigma(i)} \le H(W)\qquad\text{a.s.}
\]
Since $W\ge 0$, all cross terms in $Z^2$ are nonnegative, and therefore
\[
\E[Z^2]
\ge
\E\Big[\sum_{i=1}^n W_{i,\sigma(i)}^2\Big]
=
\frac{1}{n}\sum_{i,j}W_{ij}^2
=
\frac{1}{n}.
\]
Hence $H(W)\ge \sqrt{\E[Z^2]}\ge 1/\sqrt n$. Equality is attained by a matrix supported on a single row or column with equal entries, so $\alphapi=1/\sqrt n$.
\end{proof}

\begin{proof}[Proof of Theorem~\ref{thm:iqs-upper-main}]
Apply Theorem~\ref{thm:selfnorm-main} (deterministic service part) with $A_{\max}=\sqrt n$, $S_{\max}=\sqrt n$, and $\alphapi=1/\sqrt n$. This gives
\[
\PeakT
\;\lesssim\;
\frac{n^{3/2}}{\eps}
+
\frac{\sqrt n}{\eps}\log\frac{T+1}{\delta}
\qquad\text{with probability at least }1-\delta.
\]
Now multiply by $n$ using Proposition~\ref{prop:iqs-alpha2}.
\end{proof}
\subsection[Generalized IQS upper bound under CRP]{Corollary~\ref{cor:iqs-crp}: generalized IQS upper bound under CRP}\label{app:iqs-local-upper}
\begin{proof}[Proof of Corollary~\ref{cor:iqs-crp}]
Assume, without loss of generality, that the unique bottleneck face is row $i^\star$; the column case is identical. Let $e^{(i^\star)}$ denote the $n\times n$ row-indicator matrix for row $i^\star$. The normalized bottleneck vector is
\[
v=\frac1{\sqrt n}e^{(i^\star)},
\qquad
\normone{v}=\sqrt n,
\qquad
v_{\min}=\frac1{\sqrt n}, \qquad
\Delta_{\text{gap}}=\infty.
\]
The corresponding face is
\[
F^\star=
\left\{r\in\PiC:\sum_{j=1}^n r_{i^\star j}=1\right\}.
\]
The CRP assumption supplies a uniform lower bound on the relative distance from
$\rcirc_t$ to the boundary of this face. Pointwise in $t$, that distance is
\begin{align*}
\delta_t
:=
\min\Bigg\{&
\min_{i\neq i^\star,\;1\le j\le n}\rcirc_{ij}(t),
\sqrt{\frac{n}{n-1}}\min_{1\le j\le n}\rcirc_{i^\star j}(t),\\
&\frac{1}{\sqrt n}\min_{i\neq i^\star}\left(1-\sum_{j=1}^n\rcirc_{ij}(t)\right),
\sqrt{\frac{n}{n^2-1}}\min_{1\le j\le n}\left(1-\sum_{i=1}^n\rcirc_{ij}(t)\right)
\Bigg\}.
\end{align*}
Thus any deterministic constant satisfying
\[
0<\delta_{\mathrm{face}}
\le
\operatorname*{ess\,inf}_{\omega}\inf_{t\ge0}\delta_t(\omega)
\]
is a valid face margin in Assumption~\ref{ass:crp}.
Here is the distance calculation. The tangent space of the supporting hyperplane is
\[
\mathcal T:=\left\{h:\sum_{j=1}^n h_{i^\star j}=0\right\}.
\]
For a nonnegativity constraint outside the bottleneck row, the normal vector
already lies in \(\mathcal T\), so the relative distance is \(\rcirc_{ij}(t)\).
For a nonnegativity constraint in the bottleneck row, the normal
\(e_{i^\star j}\) projects onto \(\mathcal T\) with norm
\(\sqrt{(n-1)/n}\), giving distance
\(\sqrt{n/(n-1)}\,\rcirc_{i^\star j}(t)\). For a nonbottleneck row constraint,
the row-indicator normal has norm \(\sqrt n\) and lies in \(\mathcal T\), so the
distance is the row slack divided by \(\sqrt n\). For a column constraint, the
column-indicator normal has projection onto \(\mathcal T\) of squared norm
\(n-1/n=(n^2-1)/n\), giving the factor \(\sqrt{n/(n^2-1)}\). Since each
\(\rcirc_t\) lies in the relative interior of the face, the relative boundary is
the union of these active supporting hyperplanes. Taking the deterministic lower bound over both sample paths and time gives the uniform margin used above.

For the IQS, the ambient dimension is $d=n^2$, and the generalized arrival envelope gives $A_{\max}=\sqrt n$ and $S_{\max}=\sqrt n$. Thus $D=A_{\max}+S_{\max}=2\sqrt n$. Substituting this value into Theorem~\ref{thm:finite_horizon_ssc_crp} gives a universal constant $C_0$ such that
\[
R_\perp(T,\delta)
\le
\frac{C_0 n}{\delta_{\mathrm{face}}}
\log\frac{e(T+1)}{\delta}
=:R_{\perp,\mathrm{IQS}}(T,\delta).
\]
For a row or column bottleneck, every full permutation schedule is $v$-optimal, so the schedule-gap term is absent. The local threshold in Theorem~\ref{thm:crp-main} satisfies
\[
x_{\mathrm{loc}}(T,\delta)
=
\sqrt n\bigl(\sqrt n+R_\perp(T,\delta/2)\bigr)
\le C\bigl(n+\sqrt n\,R_{\perp,\mathrm{IQS}}(T,\delta/2)\bigr),
\]
where $R_{\perp,\mathrm{IQS}}(T,\delta/2)$ is the displayed IQS upper bound on $R_\perp(T,\delta/2)$. Applying Theorem~\ref{thm:crp-main} and using $\normone{v}=\sqrt n$, $A_{\max}=\sqrt n$, $d=n^2$, and $x_{\mathrm{loc}}(T,\delta)=O\bigl(n+\sqrt n\,R_{\perp,\mathrm{IQS}}(T,\delta/2)\bigr)$ yields
\begin{align*}
    \max_{0\le t\le T}\normone{Q_t}
&\;\lesssim\;
\frac{n}{\eps}\log\frac{T+1}{\delta}
+
\sqrt n\,x_{\mathrm{loc}}(T,\delta)
+
n\,R_\perp(T,\delta/2) \\
& \;\lesssim\; \frac{n}{\eps}\log\frac{T+1}{\delta}
+
n^{3/2}+n\,R_{\perp,\mathrm{IQS}}(T,\delta/2)
\end{align*}
with probability at least $1-\delta$.
\end{proof}

\subsection[Generalized IQS lower bounds]{Theorems~\ref{thm:iqs-lower-main} and~\ref{thm:iqs-local-lower-main}: generalized IQS lower bounds}\label{app:iqslowerbound}
\begin{proof}[Proof of Theorems~\ref{thm:iqs-lower-main} and~\ref{thm:iqs-local-lower-main}]
Fix an arbitrary nonanticipative full-permutation deterministic-service scheduling rule and start from $Q_0=0$. The lower-bound comparisons below use only that such a schedule can serve at most $n$ queues in total and at most one queue in any fixed row.
We first prove the all-ports lower bound. Fix $n\ge4$ and $\eps\in(0,1/16]$. Let $J$ be the $n\times n$ all-ones matrix, and define i.i.d. arrivals by
\[
A_t=
\begin{cases}
n^{-1/2}J,&\text{with probability }p:=\dfrac{1-\eps}{\sqrt n},\\[2mm]
0,&\text{with probability }1-p.
\end{cases}
\]
Then $\norm{A_t}_F\le\sqrt n$, and every row sum and every column sum of the mean arrival matrix is $1-\eps$. Thus Assumption~\ref{ass:gIQS} holds in the all-ports-loaded slack form.

Let $L_t:=\normone{Q_t}$. A permutation schedule can reduce total backlog by at most $n$ in one slot. Hence $L_t$ dominates the reflected random walk $R_t$ with $R_0=0$ and increments
\[
X_t=
\begin{cases}
n^{3/2}-n,&\text{with probability }p,\\[2mm]
-n,&\text{with probability }1-p.
\end{cases}
\]
After scaling by $n$, $Y_t:=R_t/n$ has increments
\[
Z_t=
\begin{cases}
\sqrt n-1,&\text{with probability }(1-\eps)/\sqrt n,\\[2mm]
-1,&\text{otherwise},
\end{cases}
\]
with mean $-\eps$. Since $n\ge4$, this reflected walk is the one-dimensional construction with deterministic unit service and upward net-input scale $A=\sqrt n-1\asymp\sqrt n$. Applying Theorem~\ref{thm:1d-lower-main} to this scaled walk and then multiplying by $n$ gives
\[
\E\Big[\max_{0\le t\le T}\normone{Q_t}\Big]
\gtrsim
\frac{n^{3/2}}{\eps}
\log\!\Big(1+\frac{\eps^2}{\sqrt n}T\Big)
\]
for $T\ge C\sqrt n/\eps^2$, after adjusting constants. This proves \eqref{inequ:IQSlower1}; finitely many smaller values of $n$ can be absorbed into the constants if desired.

We now prove the CRP lower bound. The point is to keep the row bottleneck in the relative interior of its face while preserving rare row-wide bursts. Let $E^{(1)}$ be the matrix whose first row is all ones and whose other entries are zero, and set $\bar E^{(1)}:=E^{(1)}/n$. Define a strictly positive reference point on the row-one face by
\[
\rho_{1j}=\frac1n,
\qquad
\rho_{ij}=\frac1{4n}\quad (i\ge2),
\qquad 1\le j\le n.
\]
Then $\rho\in\PiC_{\rm IQS}$, its first row sum is one, every other row sum is $1/4$, and every column sum is
\[
\frac1n+\frac{n-1}{4n}<1.
\]
Hence $\rho$ belongs to the relative interior of the row-one facet. Let $\eta:=\eps/4$ and
\[
\rcirc:=(1-\eta)\bar E^{(1)}+\eta\rho.
\]
Then $\rcirc$ also lies in the relative interior of the row-one facet, and no other row or column constraint is tight. Thus the CRP assumption holds with the bottleneck vector $v=E^{(1)}/\sqrt n$.

Define i.i.d. arrivals by
\[
A_t=
\begin{cases}
E^{(1)},&\text{with probability }p:=\dfrac{(1-\eps)(1-\eta)}{n},\\[2mm]
\rho,&\text{with probability }q:=(1-\eps)\eta,\\[2mm]
0,&\text{with probability }1-p-q.
\end{cases}
\]
The Frobenius norm is at most $\sqrt n$, since $\norm{E^{(1)}}_F=\sqrt n$ and $\norm{\rho}_F\le\sqrt n$. Moreover
\[
\E[A_t]=(1-\eps)\rcirc,
\]
so the generalized IQS slack condition holds and the nominal capacity point satisfies the CRP assumption above.

Let $R_t$ be the total backlog in the first row. A permutation schedule can serve at most one first-row queue per slot. Since the background arrival $\rho$ is nonnegative, $R_t$ dominates the reflected random walk with increments
\[
Z'_t=
\begin{cases}
n-1,&\text{with probability }p,\\[2mm]
-1,&\text{with probability }1-p.
\end{cases}
\]
The drift of this walk is
\[
pn-1
=
(1-\eps)(1-\eta)-1
=:-\eps',
\]
where $\eps\le\eps'\le5\eps/4$ for $\eps\in(0,1/16]$. In particular, $\eps'\le5/64<1/8$, so the slack range in Theorem~\ref{thm:1d-lower-main} applies. This is the one-dimensional construction with deterministic unit service and upward net-input scale $A=n-1\asymp n$. Applying Theorem~\ref{thm:1d-lower-main} with slack $\eps'$ gives
\[
\E\Big[\max_{0\le t\le T}\normone{Q_t}\Big]
\ge
\E\Big[\max_{0\le t\le T}R_t\Big]
\gtrsim
\frac{n}{\eps}
\log\!\Big(1+\frac{\eps^2}{n}T\Big)
\]
for $T\ge C'n/\eps^2$. This proves \eqref{inequ:IQSlower2}.
\end{proof}

\subsection[Bernoulli IQS upper bound]{Theorem~\ref{thm:bernoulli-iqs-upper}: Bernoulli IQS upper bound}\label{app:bernoulli-iqs}

\begin{proof}
We apply Theorem~\ref{thm:qb-unbounded-main} and the entrance-scale definition in Appendix~\ref{app:qb-entrance-scale}.
Let
\[
N_t:=\sum_{i,j}A_{ij}(t).
\]
The row and column assumptions imply $\lambda_t\in(1-\varepsilon)\Pi_{\rm IQS}$
a.s. for every $t$. Thus the uniform conditional interior-slack condition
holds. Service is deterministic in an input-queued switch, so
$\nu_S=b_S=0$.

We first verify the queue-Bernstein condition for arrivals. Fix
$w\in[0,1]^{n\times n}$ and write
\[
M_{w,t}:=\langle w,\lambda_t\rangle
=\sum_{i,j}\lambda_{ij}(t)w_{ij}.
\]
By conditional independence, for $|\theta|<1$,
\[
\mathbb E\left[
e^{\theta\langle w,A_t-\lambda_t\rangle}
\,\middle|\,\mathcal F_t
\right]
=
\prod_{i,j}
\exp\{-\theta \lambda_{ij}(t)w_{ij}\}
\left(1-\lambda_{ij}(t)+\lambda_{ij}(t)e^{\theta w_{ij}}\right).
\]
For $0\le w_{ij}\le 1$, the scalar Bernoulli Bernstein bound gives
\[
\log
\mathbb E\left[
e^{\theta\langle w,A_t-\lambda_t\rangle}
\,\middle|\,\mathcal F_t
\right]
\le
\frac{\theta^2}{2(1-|\theta|)}
\sum_{i,j}\lambda_{ij}(t)w_{ij}
=
\frac{\theta^2 M_{w,t}}{2(1-|\theta|)} .
\]
Therefore $A_t$ is queue-Bernstein with parameters
$(\nu_A,b_A)=(1,1)$, uniformly in $t$. The canonical tilt \(\theta_{\rm QB}\) in \eqref{eq:thetaQB-app} may hence be chosen as
\[
\theta=c\varepsilon
\]
for a sufficiently small universal constant $c>0$.

It remains to bound the two terms entering $x_{\rm QB}$. Every service
schedule is a permutation matrix, so
$\|S_t\|_F^2=n.$
Since the arrivals are Bernoulli,
$\|A_t\|_F^2=N_t.$
Thus,
\[
J_t^2:=\|A_t-S_t\|_F^2\le 2N_t+2n.
\]
We next bound the curvature term. Since $\rho_x(z)\le z^2/(2x)$,
for $x\ge 1$,
\[
\kappa_\theta(x)
\le
\sup_{t\ge 0}\operatorname*{ess\,sup}
\log\mathbb E\left[
\exp\left\{\frac{2\theta(N_t+n)}{x}\right\}
\,\middle|\,\mathcal F_t
\right].
\]
Conditional on $\mathcal F_t$, $N_t$ is a Poisson-binomial random variable.
Moreover,
\[
m_t:=\mathbb E[N_t\mid\mathcal F_t]
=
\sum_{i,j}\lambda_{ij}(t)
\le n(1-\varepsilon)\le n.
\]
Hence, for $a\ge 0$,
\[
\mathbb E[e^{aN_t}\mid\mathcal F_t]
=
\prod_{i,j}\left(1-\lambda_{ij}(t)+\lambda_{ij}(t)e^a\right)
\le
\exp\{(e^a-1)m_t\}
\le
\exp\{n(e^a-1)\}.
\]
Taking $a=2\theta/x$ gives, for $x\ge 1$ and $\theta\le 1$,
\[
\log\mathbb E\left[
\exp\left\{\frac{2\theta(N_t+n)}{x}\right\}
\,\middle|\,\mathcal F_t
\right]
\le
\frac{2\theta n}{x}
+
n\left(e^{2\theta/x}-1\right)
\le
C\frac{\theta n}{x}.
\]
Since $\alphapi=1/\sqrt n$ by Proposition~\ref{prop:iqs-alpha2}, choosing
\[
x\ge C\frac{n^{3/2}}{\eps}
\]
makes $\kappa_\theta(x)\le \theta\eps\alphapi/4$. Hence
\begin{equation}\label{eq:iqs-xcurv-qb}
x_{\rm curv}(\theta)
\le C\frac{n^{3/2}}{\eps}.
\end{equation}
For the upward-jump envelope, $J_t\le\sqrt{2N_t+2n}\le N_t+n+1$.
Since $K_{\rm up}(\theta)$ is defined with the harmless factor $8$ in
\eqref{eq:Kup-qb}, for $0<\theta\le1/8$,
\[
K_{\rm up}(\theta)
\le
\exp(8\theta(n+1))
\sup_{t\ge0}\operatorname*{ess\,sup}
\mathbb E[e^{8\theta N_t}\mid\mathcal F_t].
\]
Using the conditional Poisson-binomial MGF and $m_t\le n$,
\[
\mathbb E[e^{8\theta N_t}\mid\mathcal F_t]
\le
\exp\{n(e^{8\theta}-1)\}.
\]
Hence
\[
\log K_{\rm up}(\theta)\le C\theta n.
\]
With $\theta_{\rm QB}=c\eps$, \eqref{eq:iqs-xcurv-qb} and the preceding display give \(x_{\rm QB}\le Cn^{3/2}/\eps\). Theorem~\ref{thm:qb-unbounded-main} then gives
\[
\max_{0\le t\le T}\normtwo{Q_t}
\le
C\left[
\frac{n^{3/2}}{\eps}
+
\frac1\eps\log\frac{T+1}{\delta}
+
 n
\right]
\]
with probability at least $1-\delta$. Since $n\le n^{3/2}/\eps$ for $n\ge1$ and $\eps\in(0,1)$, this proves the Euclidean bound. Multiplying by $n$ using $\normone{Q}\le n\normtwo{Q}$ proves the total-backlog statement.
\end{proof}

\end{document}